\documentclass{amsart}
\usepackage{amssymb,euscript,tikz,units}
\usepackage[colorlinks,citecolor=blue,linkcolor=blue]{hyperref}
\usepackage{verbatim}
\usepackage[nameinlink]{cleveref}
\usepackage{colonequals}
\usepackage{tikz}
\usepackage[titletoc]{appendix}
\usepackage{yhmath}
\usepackage{enumerate}
\usepackage{hyperref}
\usepackage{url}
\usepackage{subcaption}

\allowdisplaybreaks[4]

\newcommand{\N}{\mathbb{N}}
\newcommand{\calT}{\mathcal{T}}
\newcommand{\T}{\calT}
\newcommand{\M}{\mathcal{M}}
\newcommand{\sM}{\mathcal{M}}

\newcommand{\sN}{\mathcal{N}}
\newcommand{\R}{\mathbb{R}}
\newcommand{\cR}{\mathcal{R}}
\newcommand{\cD}{\mathcal{D}}
\newcommand{\Z}{\mathbb{Z}}

\newcommand{\eps}{\varepsilon}
\newcommand{\wep}{Weil-Petersson}

\newcommand{\sbs}{\subset}

\newcommand{\E}{\mathbb{E}_{\rm WP}^g}

\def\sys{\mathop{\rm sys}}
\def\area{\mathop{\rm Area}}
\def\arcsinh{\mathop{\rm arcsinh}}
\def\arccosh{\mathop{\rm arccosh}}
\def\Vol{\mathop{\rm Vol}}

\DeclareMathOperator{\Prob}{Prob_{WP}^{g}}

\def\Mod{\mathop{\rm Mod}}

\def\sys{\mathop{\rm sys}}
\def\area{\mathop{\rm Area}}
\def\arcsinh{\mathop{\rm arcsinh}}
\def\arccosh{\mathop{\rm arccosh}}
\def\Vol{\mathop{\rm Vol}}
\def\Volwp{\mathop{\rm Vol_{\rm WP}}}

\def\Prob{\mathop{\rm Prob}\nolimits_{\rm WP}^g}
\def\Mod{\mathop{\rm Mod}}

\newcommand{\lss}{\ell_{\sys}^{\rm sep}}
\newcommand{\lns}{\ell_{\sys}^{\rm ns}}

\newcommand{\feight}{\rm f-8}

\newcommand{\MN}{\mathcal{N}}
\newcommand{\Nset}{\MN_{(0,3),\star}^{(g-2,3)}}
\newcommand{\Nnumber}{N_{(0,3),\star}^{(g-2,3)}}

\newcommand{\MC}{\mathcal{C}}
\newcommand{\MD}{\mathcal{D}}

\newcommand{\nsys}{\ell^{ns}_{sys}}

\numberwithin{equation}{section}

\theoremstyle{plain}
\newtheorem{theorem}{Theorem}[section]

\newtheorem{proposition}[theorem]{Proposition}
\newtheorem{lemma}[theorem]{Lemma}

\newtheorem{remark}[theorem]{Remark}

\newcommand{\be}{\begin{equation}}
\newcommand{\ene}{\end{equation}}
\newcommand{\br}{\begin{remark}}
\newcommand{\er}{\end{remark}}
\newcommand{\bl}{\begin{lem}}
\newcommand{\el}{\end{lem}}
\newcommand{\bcor}{\begin{cor}}
\newcommand{\ecor}{\end{cor}}
\newcommand{\bpro}{\begin{pro}}
\newcommand{\epro}{\end{pro}}
\newcommand{\ben}{\begin{enumerate}}
\newcommand{\een}{\end{enumerate}}
\newcommand{\bp}{\begin{proof}}
\newcommand{\ep}{\end{proof}}
\newcommand{\bpo}{\begin{pro}}
\newcommand{\epo}{\end{pro}}
\newcommand{\beq}{\begin{equation*}}
\newcommand{\eeq}{\end{equation*}}
\newcommand{\bear}{\begin{eqnarray}}
\newcommand{\eear}{\end{eqnarray}}
\newcommand{\beqar}{\begin{eqnarray*}}
\newcommand{\eeqar}{\end{eqnarray*}}
\newcommand{\bt}{\begin{theorem}}
\newcommand{\et}{\end{theorem}}
\newcommand{\bex}{\begin{excer}}
\newcommand{\eex}{\end{excer}}

\theoremstyle{definition}
\newtheorem{definition}[theorem]{Definition}

\theoremstyle{remark}

\newtheorem*{con*}{Construction}
\newtheorem*{rem*}{Remark}
\newtheorem*{exam*}{Example}
\newtheorem*{exams*}{Examples}
\newtheorem*{thm*}{\bf Theorem}
\newtheorem*{que*}{Question}

\newtheorem*{Def*}{Definition}
\newtheorem*{Cons*}{Construction}
\newtheorem*{Lem*}{Lemma}
\newtheorem*{Conj*}{\bf Conjecture}
\begin{document}
\title{Non-simple systoles on random hyperbolic surfaces for large genus}

\author{Yuxin He}
\address{Department of Mathematical Sciences and Yau Mathematical Sciences Center, Tsinghua University, Beijing, China}
\email[(Y.~H.)]{hyx21@mails.tsinghua.edu.cn}

\author{Yang Shen}
\address{School of Mathematical Sciences,
Fudan University, Shanghai, China}
\email[(Y.~S.)]{shenwang@fudan.edu.cn}

\author{Yunhui Wu}
\address{Department of Mathematical Sciences and Yau Mathematical Sciences Center, Tsinghua University, Beijing, China}
\email[(Y.~W.)]{yunhui\_wu@tsinghua.edu.cn}

\author{Yuhao Xue}
\address{Institut des Hautes \'Etudes Scientifiques (IHES), Bures-sur-Yvette, France}
\email[(Y.~X.)]{xueyh@ihes.fr}

\maketitle

\begin{abstract}
In this paper, we investigate the asymptotic behavior of the non-simple systole, which is the length of a shortest non-simple closed geodesic, on a random closed hyperbolic surface on the moduli space $\sM_g$ of Riemann surfaces of genus $g$ endowed with the Weil-Petersson measure. We show that as the genus $g$ goes to infinity, the non-simple systole of a generic hyperbolic surface in $\sM_g$ behaves exactly like $\log g$.   
\end{abstract}

 \tableofcontents

\section{Introduction}
The study of closed geodesics on hyperbolic surfaces has a deep connection to their spectral theory, dynamics, and hyperbolic geometry. Let $X=X_g$ be a closed hyperbolic surface of genus $g\geq 2$. The systole of $X$, the length of a shortest closed geodesic on $X$, is always realized by a simple closed geodesic, i.e., a closed geodesic without self-intersections. The \emph{non-simple systole} $\nsys(X)$ of $X$ is defined as
\[\ell^{ns}_{sys}(X)=\min\{\ell_{\alpha}(X); \ \textit{$\alpha\subset X$ is a non-simple closed geodesic} \}\]
where $\ell_{\alpha}(X)$ is the length of $\alpha$ in $X$. It is known that $\nsys(X)$ is always realized as a figure-eight closed geodesic in $X$ (see e.g. \cite[Theorem 4.2.4]{Buser10}). In this work, we view the non-simple systole as a random variable on the moduli space $\sM_g$ of Riemann surfaces of genus $g$ endowed with the Weil-Petersson probability measure $\Prob$. This subject was initiated by Mirzakhani in \cite{Mirz10, Mirz13}, based on her celebrated thesis works \cite{Mirz07,Mirz07-int}. Firstly, it is known that for all $g\geq 2$,  
$\inf\limits_{X\in \sM_g}\nsys(X)=2\arccosh (3)\sim 3.52...$ (see e.g. \cite{Yam82, Buser10}) 
and
$\sup\limits_{X\in \sM_g}\nsys(X)\asymp \log g$ (see e.g. \cite{BS94, Tor23}). In this paper, we show that as $g$ goes to infinity, a generic hyperbolic surface in $\sM_g$ has its non-simple systole behaving like $\log g$. More precisely, let $\omega:\{2,3,\cdots\}\to\R^{>0}$ be any function satisfying 
\be \label{eq-omega}
\lim \limits_{g\to \infty}\omega(g)= +\infty \ \textit{and} \ \lim \limits_{g\to \infty}\frac{\omega(g)}{\log\log g} = 0. 
\ene
\begin{theorem}\label{mt-1}
For any $\omega(g)$ satisfying \eqref{eq-omega}, the following limit holds:
\[\lim \limits_{g\to \infty}\Prob\left(X\in \sM_g; \ |\nsys(X)-(\log g-\log \log g)|< \omega(g) \right)=1.\]
\end{theorem}

\begin{rem*}
It was shown in \cite[Theorem 4]{NWX23} that for any $\epsilon>0$,
\[\lim \limits_{g\to \infty}\Prob\left(X\in \sM_g; \ (1-\epsilon)\log g< \nsys(X)<2\log g \right)=1.\]
\end{rem*}

As a direct consequence of Theorem \ref{mt-1}, as $g$ goes to infinity, the asymptotic behavior of the expected value of $\nsys(\cdot)$ over $\sM_g$ can also be determined.

\begin{theorem}\label{mt-2}
The following limit holds:
\[\lim \limits_{g\to \infty} \frac{\int_{\sM_g}\nsys(X)dX}{\Volwp(\sM_g) \log g}=1.\]
\end{theorem}

\begin{proof}
Take $\omega(g)=\log \log \log g$ and set $V_g=\Volwp(\sM_g)$. Define
\[A_{\omega}(g):=\{X\in \sM_g; \ |\nsys(X)-(\log g-\log \log g) |< \omega(g)\}.\]
By Theorem \ref{mt-1} we know that
\[\lim\limits_{g\to \infty}\Prob\left(X\in \sM_g; \ X\in A_{\omega}(g)\right)=1.\]
Then firstly from Markov's inequality, we have
\beqar
\liminf\limits_{g\to \infty}\frac{\int_{\sM_g}\nsys(X)dX}{V_g \cdot \log g}\geq \lim\limits_{g\to \infty}\Prob\left(X\in \sM_g; \ X\in A_{\omega}(g)\right)=1.
\eeqar
For the other direction, since $\sup\limits_{X\in \sM_g}\nsys(X)\leq C\cdot \log g$ for some universal constant $C>0$ (see e.g. Lemma \ref{thm f-8 <clogg}),
\beqar
&&\ \ \  \limsup\limits_{g\to \infty}\frac{\int_{\sM_g}\nsys(X)dX}{V_g\cdot \log g}\\
&&=\limsup\limits_{g\to \infty} \left(\frac{\int_{A_{\omega}(g)}\nsys(X)}{V_g\cdot \log g}
+\frac{\int_{A^c_{\omega}(g)}\nsys(X)}{V_g\cdot \log g}\right)\\
&& \leq 1+ C\cdot \limsup\limits_{g\to \infty}\Prob\left(X\in \sM_g; \ X\notin A_{\omega}(g)\right)=1.
\eeqar
The proof is complete.
\end{proof}

\begin{rem*}
\ben
\item Mirzakhani-Petri in \cite{MP19} showed that
\[\lim \limits_{g\to \infty} \frac{\int_{\sM_g}\ell_{\rm sys}(X)dX}{\Volwp(\sM_g)}=1.61498...\] where $\ell_{\rm sys}(X)$ is the systole of $X$.

\item Based on \cite{NWX23}, joint with Parlier, the third and fourth named authors in \cite{PWX22} showed that
\[\lim \limits_{g\to \infty} \frac{\int_{\sM_g}\lss(X)dX}{\Volwp(\sM_g) \log g}=2\] where $\lss(X)$ is the length of a shortest separating simple closed geodesic in $X$, an unbounded function over $\sM_g$.
\een
\end{rem*}

 The geometry and spectra of random hyperbolic surfaces under the Weil-Petersson measure have been widely studied in recent years. For examples, one may see \cite{GPY11} for Bers' constant, \cite{Mirz13,WX22-GAFA} for diameter, \cite{MP19} for systole, \cite{Mirz13, NWX23, PWX22} for separating systole, \cite{Mirz13, WX22-GAFA, LW21, AM23} for the first eigenvalue, \cite{GMST19} for eigenfunction, \cite{Monk22} for Weyl law, \cite{Rud22, RuW23} for GOE, \cite{WX22pgt} for prime geodesic theorem, \cite{Naud23} for determinant of Laplacian. One may also see \cite{MS20, MT20, Hide21, SW22, HW22, HHH22, HT22, DS23, Gong23, MS23} and the references therein for more related topics.\\

\noindent \textbf{Strategy on the proof of Theorem \ref{mt-1}.} The proof of Theorem \ref{mt-1} mainly consists of two parts.\\

A relatively easier part is to prove the lower bound on $\nsys$, that is to show that
\be\label{mt-low}
\lim \limits_{g\to \infty}\Prob\left(X\in \sM_g; \ \nsys(X)> \log g-\log \log g- \omega(g) \right)=1.
\ene 
We know that $\nsys(X)$ is realized by a figure-eight closed geodesic that is always filling in a unique pair of pants. And the length of such a figure-eight closed geodesic can be determined by the lengths of the three boundary geodesics of the pair of pants (see e.g., formula \eqref{eq figure-8 length}). For $L=L_g=\log g-\log \log g- \omega(g)$ and $X\in \sM_g$, denote by $N_{\feight}(X,L)$ the number of figure-eight closed geodesics of length $\leq L$ in $X$. We view it as a random variable on $\sM_g$. Then using Mirzakhani's integration formula and change of variables, a direct computation shows that (see Proposition \ref{thm E[N_f8]=}) its expected value $\E[N_{\feight}(X,L)] $ satisfies that as $g\to \infty$, 
\be
\E[N_{\feight}(X,L)] \sim \frac{Le^L}{8\pi^2 g}\to 0.
\ene
Here we say $f(g)\sim h(g)$ if $\lim\limits_{g\to \infty}\frac{f(g)}{h(g)}=1$. Thus, we have $$\Prob\left(X\in \sM_g; \ N_{\feight}(X,L_g)\geq 1 \right)\leq \E[N_{\feight}(X,L_g)]\to 0$$
as $g\to \infty$. This in particular implies \eqref{mt-low}.\\

The hard part of Theorem \ref{mt-1} is the upper bound, that is to show that
\be\label{mt-upp}
\lim \limits_{g\to \infty}\Prob\left(X\in \sM_g; \ \nsys(X)< \log g-\log \log g+ \omega(g) \right)=1.
\ene 
Set $L=L_g= \log g-\log \log g+ \omega(g),$ and for any $X\in \sM_g$, we consider the following set of pairs of \emph{ordered} simple closed multi-geodesics in $X$

$$\mathcal{N}_{(0,3),\star}^{(g-2,3)}(X,L)=\left\{(\gamma_1,\gamma_2,\gamma_3);\
\begin{matrix}(\gamma_1,\gamma_2,\gamma_3)\text{ is a pair of ordered simple closed}\\  \text{geodesics  such  that }X\setminus \bigcup_{i=1}^3\gamma_i\cong  S_{0,3} \cup S_{g-2,3},\\ \ell(\gamma_1)\leq L,\ \ell(\gamma_2)+\ell(\gamma_3)\leq L\\ \text{ and }\ell(\gamma_1),\ \ell(\gamma_2),\ \ell(\gamma_3)\geq 10\log L\end{matrix}\right\}$$
and denote by
$$N_{(0,3),\star}^{(g-2,3)}(X,L)=\#\mathcal{N}_{(0,3),\star}^{(g-2,3)}(X,L).$$
For any pair $(\gamma_1,\gamma_2,\gamma_3)\in\mathcal{N}_{(0,3),\star}^{(g-2,3)}(X,L)$, assume $P$ is the pair of pants in $X$ with boundary geodesics $\gamma_1,\ \gamma_2$ and $\gamma_3$. There exists a figure-eight closed geodesic in $P$ with length $\leq L+2\log 6$. Hence 
\begin{equation}
 \label{i-link f-8 with Nstar} 
\begin{aligned}
     &\Prob\left(X\in \M_g ;\  N_{\feight}(X,L+2\log 6)=0 \right)\\
     \leq &\Prob\left(X\in \M_g ;\ \Nnumber(X,L)=0 \right) \\
 \leq &\frac{\E\left[\big(\Nnumber(X,L)\big)^2\right]-\E\left[\Nnumber(X,L)\right]^2}{\E\left[\Nnumber(X,L)\right]^2}.
     \end{aligned}
 \end{equation}
To prove \eqref{mt-upp}, it suffices to show 
\be\label{i-e-aim}
\lim\limits_{g\to \infty} \mathrm{RHS} \textit{\rm{ of \eqref{i-link f-8 with Nstar}}}=0.
\ene

\noindent Using Mirzakhani's integration formula and known bounds on Weil-Petersson volumes, direct computations show that (see Proposition \ref{p-e-1}),
\be \label{i-a-1}
\E\left[N_{(0,3),\star}^{(g-2,3)}(X,L)\right]= \frac{1}{2\pi^2 g}Le^L\left(1+O\left(\frac{\log L}{L}\right)\right)\sim \frac{e^{\omega(g)}}{2\pi^2}.
\ene
Next we consider $\E\left[\big(\Nnumber(X,L)\big)^2\right]$. For any $\Gamma\in \mathcal{N}_{(0,3),\star}^{(g-2,3)}(X,L)$, denote by $P(\Gamma)$ the unique pair of pants bounded by $\Gamma$ and let $\overline{P(\Gamma)}$ be its completion in $X$. Split $\big(\Nnumber(X,L)\big)^2$ as follows:
\begin{align*}
&A(X,L)=\#\left\{(\Gamma_1,\Gamma_2);\ P(\Gamma_1)=P(\Gamma_2)\right\},\\
&B(X,L)=\#\left\{(\Gamma_1,\Gamma_2); \ \overline{P(\Gamma_1)}\cap \overline{P(\Gamma_2)}=\emptyset \right\},\\
&C(X,L)=\#\left\{(\Gamma_1,\Gamma_2);\  P(\Gamma_1)\neq P(\Gamma_2),\ {P(\Gamma_1)}\cap {P(\Gamma_2)}\neq\emptyset\right\},\\
&D(X,L)=\#\left\{(\Gamma_1,\Gamma_2);\  P(\Gamma_1)\cap P(\Gamma_2)=\emptyset,\ \overline{P(\Gamma_1)}\cap \overline{P(\Gamma_2)}\neq\emptyset\right\}.
\end{align*}
The sets $A(X,L),\ B(X,L),\ C(X,L),\ D(X,L)$, respectively, correspond to the cases where the two pairs of pants are equal, disjoint, distinct with non-empty intersection and distinct but touching along their boundary geodesics.
It is clear that 
\begin{equation}\label{i-a-1-1}
\begin{aligned}
\E\left[\left(\Nnumber(X,L)\right)^2\right]&=\E\left[A(X,L)\right]+\E\left[B(X,L)\right]\\
&+\E\left[C(X,L)\right]+\E\left[D(X,L)\right].
\end{aligned}
\end{equation}

Since for a pair of pants $P$ in $X$, there exist at most $6$ different $\Gamma'$s such that $P=P(\Gamma)$, it follows from \eqref{i-a-1} that
\begin{align}\label{i-a-1-2}
\E\left[A(X,L)\right]\leq 6\cdot \E\left[\mathcal{N}_{(0,3),\star}^{(g-2,3)}(X,L)\right]\prec \frac{Le^L}{g}\asymp e^{\omega(g)}.
\end{align}

Using Mirzakhani's integration formula and known bounds on Weil-Petersson volumes, direct computations can show that (see Proposition \ref{estimation B(X,L)} and \ref{D(x,L)})
\begin{align}\label{i-a-2}
\E\left[B(X,L)\right]=\frac{1}{4\pi^4 g^2}L^2e^{2L}\left(1+O\left(\frac{\log L}{L}\right)\right)\sim \E\left[N_{(0,3),\star}^{(g-2,3)}(X,L)\right]^2
\end{align}
and
\be\label{i-a-3}
    \E\left[D(X,L) \right]\prec \frac{e^{2L}}{g^2L^6}=o(1).
\ene
Combining \eqref{i-a-1} and \eqref{i-a-2}, we have the following important cancellation:
\be\label{i-a-canc}
\left|\E\left[B(X,L)\right]- \E\left[N_{(0,3),\star}^{(g-2,3)}(X,L)\right]^2 \right|\prec \frac{L^2e^{2L}}{g^2}\cdot \frac{\log L}{L}=e^{2\omega(g)}\cdot o(1).
\ene

The \emph{crucial} part is to bound $\E\left[C(X,L)\right]$. For this part, our method is inspired by the ones in \cite{MP19, WX22-GAFA, NWX23}. Denote by $S(\Gamma_1,\Gamma_2)$ the compact subsurface in $X$ of geodesic boundary such that $\Gamma_1\cup \Gamma_2$ is filling in it. Then one can divide $C(X,L)$ as follows:
\begin{align*}
&C_{0,4}(X,L)=\#\left\{(\Gamma_1,\Gamma_2)\in\mathcal{C}(X,L);\ S(\Gamma_1,\Gamma_2)\cong S_{0,4}\right\},\\
&C_{1,2}(X,L)=\#\left\{(\Gamma_1,\Gamma_2)\in\mathcal{C}(X,L);\ S(\Gamma_1,\Gamma_2)\cong S_{1,2}\right\},\\
&C_{\geq 3}(X,L)=\#\{(\Gamma_1,\Gamma_2)\in\mathcal{C}(X,L);\ |\chi(S(\Gamma_1,\Gamma_2))|\geq 3\}.
\end{align*}  
Through using the method in \cite{WX22-GAFA} and the counting results on filling multi-geodesics in \cite{WX22-GAFA, WX22pgt}, we can show that (see Proposition \ref{prop C geq3})
\be\label{i-a-4}
\E\left[C_{\geq 3}(X,L)\right]\prec \left(L^{67}e^{(2+\epsilon)\cdot L}\frac{1}{g^3}+\frac{L^3e^{8L}}{g^{11}}\right)=o(1).
\ene
For $C_{0,4}(X,L)$ and $C_{1,2}(X,L)$, through classifying all the accurate relative positions of $(\Gamma_1,\Gamma_2)$ in both $S_{1,2}$ and $S_{0,4}$, applying the McShane-Mirzakhani identity in \cite{Mirz07} as for counting closed geodesics (we warn here that both the general counting result and the counting result in \cite{WX22-GAFA, WX22pgt} on closed geodesics are inefficient to deal with these two cases), and then using  Mirzakhani's integration formula and known bounds on Weil-Petersson volumes, we can show that (see Proposition \ref{c12} and \ref{c04})
\be\label{i-a-5}
\E\left[C_{1,2}(X,L)\right]\prec \frac{e^{2L}}{g^2}=o(1)
\ene
and
\be\label{i-a-6}
     \E\left[C_{0,4}(X,L) \right]\prec \frac{Le^{2L}}{g^2}=o(1).
\ene

Then, combining all these equations \eqref{i-a-1}---\eqref{i-a-6}, one may finish the proof of \eqref{i-e-aim}, thus get \eqref{mt-upp} which is the upper bound in Theorem \ref{mt-1}. \\

\noindent \textbf{Notations.} For any two nonnegative functions $f$ and $h$ (may be of multi-variables), we say $f\prec h$ if there exists a uniform constant $C>0$ such that $f\leq Ch$. And we also say $f\asymp h$ if $f\prec h$ and $h\prec f$.  \\

\noindent \textbf{Plan of the paper.} Section \ref{s-pre} will provide a review of relevant and necessary background materials. In Section \ref{s-low} we compute the expected value of the number of figure-eight closed geodesics of length $\leq L$ over $\sM_g$ which will imply \eqref{mt-low}, i.e., the lower bound in Theorem \ref{mt-1}. In Section \ref{s-upp} we prove \eqref{mt-upp}, i.e., the upper bound in Theorem \ref{mt-1}. In which we apply the counting result on closed geodesics in \cite{WX22-GAFA}, and also apply the McShane-Mirzakhani identity in \cite{Mirz07} to count closed geodesics for $C_{1,2}(X,L)$ and $C_{0,4}(X,L)$.\\

\subsection*{Acknowledgement} We would like to thank all the participants in our seminar on Teichm\"uller theory for helpful discussions on this project. The second-named author is supported by the NSFC grant No. $12401081$. The third-named author is partially supported by the NSFC grant No. $12171263$ and $12361141813$.

\section{Preliminaries}\label{s-pre}

In this section, we set our notations and recall certain relevant necessary results used in this paper, including the Weil-Petersson metric, Mirzakhani's
integration formula, bounds on \wep \ volumes, figure-eight closed geodesics and three counting results on closed geodesics.

\subsection{Moduli space and Weil-Petersson metric}
Denote by $S_{g,n}$ an oriented topological surface with genus $g$ of $n$ punctures or boundaries where $2g-2+n\geq 1$. Let $\T_{g,n}$ be the Teichm\"uller space of surfaces with genus $g$ of $n$ punctures, and $\Mod_{g,n}$ be the mapping class group of $S_{g,n}$ fixing the order of punctures. The moduli space of Riemann surfaces is $\M_{g,n}=\T_{g,n}/\Mod_{g,n}$. Write $\T_{g}=\T_{g,0}$ and $\M_{g}=\M_{g,0}$ for simplicity. Given $L=(L_1,\cdots,L_n)\in \R^n_{\geq 0}$, let $\T_{g,n}(L)$ be the Teichm\"uller space of bordered hyperbolic surfaces with $n$ geodesic boundaries of lengths $L_1,\cdots,L_n$, and $\M_{g,n}(L)=\T_{g,n}(L)/\Mod_{g,n}$ be the weighted moduli space. In particular, $\T_{g,n}(0,\cdots,0)=\T_{g,n}$ and $\sM_{g,n}(0,\cdots,0)=\sM_{g,n}$.

Given a pants decomposition $\{\alpha_i\}_{i=1}^{3g-3+n}$ of $S_{g,n}$, the Fenchel-Nielsen coordinates on the Teichm\"uller space $\T_{g,n}(L)$ is given by the map $X\mapsto (\ell_{\alpha_i}(X),\tau_{\alpha_i}(X))_{i=1}^{3g-3+n}$. Here $\ell_{\alpha_i}(X)$ is the length of $\alpha_i$ on $X$ and $\tau_{\alpha_i}(X)$ is the twist along $\alpha_i$ (measured by length). The following magic formula is due to Wolpert \cite{Wolpert82}:

\begin{theorem}[Wolpert]\label{thm Wolpert}
	The Weil-Petersson symplectic form $\omega_{\mathrm{WP}}$ on $\T_{g,n}(L)$ is given by 
	$$\omega_{\mathrm{WP}} = \sum_{i=1}^{3g-3+n} d\ell_{\alpha_i} \wedge d\tau_{\alpha_i}.$$
\end{theorem}

The form $\omega_{\mathrm{WP}}$ is mapping class group invariant. So it induces the so-called \emph{Weil-Petersson volume form} on $\sM_{g,n}$ given by 
$$d\Vol\nolimits_{\mathrm{WP}}=\tfrac{1}{(3g-3+n)!}\cdot \underbrace{\omega_{\mathrm{WP}}\wedge\cdots\wedge\omega_{\mathrm{WP}}}_{3g-3+n\ \text{copies}}.$$
Denote by $V_{g,n}(L)$ the total volume of $\M_{g,n}(L)$ under the Weil-Petersson metric which is finite. Following \cite{Mirz13}, we view a measurable function $f:\sM_g\to \R$ as a random variable on $\sM_g$ with respect to the probability measure $\Prob$ on $\M_g$, and we also denote by $\E[f]$ the expected value of $f$ over $\sM_g$. Namely,
$$\Prob(\mathcal{A}):=\frac{1}{V_{g}}\int_{\M_{g}} \mathbf{1}_{\mathcal{A}} dX, \quad \E[f]:=\frac{1}{V_g}\int_{\sM_g}f(X)dX$$
where $\mathcal{A}\sbs\M_g$ is a Borel subset, $\mathbf{1}_{\mathcal{A}}:\M_{g}\to\{0,1\}$ is its characteristic function, and $dX$ is short for $d\Vol_{\mathrm{WP}}(X)$.

\subsection{Mirzakhani's integration formula}\label{sec mirz int fml}
In this subsection, we recall Mirzakhani's integration formula in \cite{Mirz07}.

Let $\gamma$ be a non-trivial and non-peripheral closed curve on the topological surface $S_{g,n}$ and $X\in\T_{g,n}$. Denote $\ell(\gamma)=\ell_\gamma(X)$ to be the hyperbolic length of the unique closed geodesic in
the homotopy class of $\gamma$ on $X$. Let $\Gamma=(\gamma_1,\cdots,\gamma_k)$ be an ordered k-tuple where the $\gamma_i$'s are distinct disjoint homotopy classes of nontrivial, non-peripheral, unoriented simple closed curves on $S_{g,n}$. Let $\mathcal{O}_\Gamma$ be the orbit containing $\Gamma$ under the $\Mod_{g,n}$-action, i.e.
$$\mathcal{O}_\Gamma = \{(h\cdot\gamma_1,\cdots,h\cdot\gamma_k);\ h\in\Mod\nolimits_{g,n}\}.$$
Given a function $F:\R_{\geq0}^k \to\R$, one may define a function on $\M_{g,n}$:
\begin{eqnarray*}
	F^\Gamma: \M_{g,n} &\to& \R \\
	X &\mapsto& \sum_{(\alpha_1,\cdots,\alpha_k)\in\mathcal{O}_\Gamma} F(\ell_{\alpha_1}(X),\cdots,\ell_{\alpha_k}(X)).
\end{eqnarray*}
Note that although $\ell_{\alpha_i}(X)$ can be only defined on $\T_{g,n}$, after taking a sum over all $(\alpha_1,\cdots,\alpha_k)\in\mathcal{O}_\Gamma$, the function $F^\Gamma$ is well-defined on the moduli space $\M_{g,n}$.

For any $x=(x_1,\cdots,x_k)\in\R_{\geq 0}^k$, we denote by $\M(S_{g,n}(\Gamma);\ell_{\Gamma}=x)$ the moduli space of the hyperbolic surfaces (possibly disconnected) that homeomorphic to $S_{g,n}\setminus\bigcup_{j=1}^k\gamma_j$ with $\ell(\gamma_i^1)=\ell(\gamma_i^2)=x_i$ for every $i=1,\cdots,k$, where $\gamma_i^1$ and $\gamma_i^2$ are the two boundary components of $S_{g,n}\setminus\bigcup_{j=1}^k\gamma_j$ given by cutting along $\gamma_i$. Assume $S_{g,n}\setminus\bigcup_{j=1}^k\gamma_j \cong \bigcup_{i=1}^s S_{g_i,n_i}$. Consider the Weil-Petersson volume
\begin{align}\label{e-wpvol}
V_{g,n}(\Gamma,x)=\Vol\nolimits_{\mathrm{WP}}\left(\M(S_{g,n}(\Gamma);\ell_{\Gamma}=x)\right) = \prod_{i=1}^s V_{g_i,n_i}(x^{(i)})
\end{align}
where $x^{(i)}$ is the list of those coordinates $x_j$ of $x$ such that $\gamma_j$ is a boundary component of $S_{g_i,n_i}$. The following integration formula is due to Mirzakhani. One may see \cite[Theorem 7.1]{Mirz07}, \cite[Theorem 2.2]{Mirz13}, \cite[Theorem 2.2]{MP19} and \cite[Theorem 4.1]{Wright-tour} for different versions.

\begin{theorem}[Mirzakhani's integration formula]\label{thm Mirz int formula}
	For any $\Gamma=(\gamma_1,\cdots,\gamma_k)$, the integral of $F^\Gamma$ over $\mathcal{M}_{g,n}$ with respect to the Weil-Petersson metric is given by
	$$\int_{\M_{g,n}} F^\Gamma dX = C_\Gamma\int_{\R_{\geq 0}^k}F(x_1,\cdots,x_k)V_{g,n}(\Gamma,x)x_1\cdots x_k dx_1\cdots dx_k$$
	where the constant $C_\Gamma\in(0,1]$ only depends on $\Gamma$. 
\end{theorem}

\begin{rem*}
	One may see \cite[Theorem 4.1]{Wright-tour} for the detailed explanation and expression for the constant $C_\Gamma$. We will give the exact value of $C_\Gamma$ only when required in this paper.
\end{rem*}

\subsection{Weil-Petersson volumes}
Let $V_{g,n}(x_1,\cdots,x_n)$ be the Weil-Petersson volume of $\M_{g,n}(x_1,\cdots,x_n)$ and $V_{g,n}=V_{g,n}(0,\cdots,0)$. In this subsection, we only list the bounds for $V_{g,n}(x_1,\cdots,x_n)$ that we will need in this paper.

\begin{theorem}[{\cite[Theorem 1.1]{Mirz07}}]\label{mirz07}
	The initial volume $V_{0,3}(x,y,z) = 1$. The volume $V_{g,n}(x_1,\cdots,x_n)$ is a polynomial in $x_1^2,\cdots,x_n^2$ with degree $3g-3+n$. Namely we have
	\begin{equation*}
		V_{g,n}(x_1,\cdots,x_n) = \sum_{\alpha;\,|\alpha|\leq 3g-3+n} C_\alpha \cdot x^{2\alpha}
	\end{equation*}
	where $C_\alpha>0$ lies in $\pi^{6g-6+2n-|2\alpha|} \cdot \mathbb Q$. Here $\alpha=(\alpha_1,\cdots,\alpha_n)$ is a multi-index and $|\alpha|=\alpha_1+\cdots+\alpha_n$, $x^{2\alpha}= x_1^{2\alpha_1}\cdots x_n^{2\alpha_n}$.
\end{theorem}

\begin{theorem}\label{thm Vgn/Vgn+1}
	\begin{enumerate}		
		\item {\rm{(\cite[Lemma 3.2]{Mirz13}).}}
		For any $g,n\geq 0$
		\begin{equation*}
			V_{g-1,n+4} \leq V_{g,n+2}
		\end{equation*}
		and
		\begin{equation*}
			b_0\leq \frac{V_{g,n+1}}{(2g-2+n)V_{g,n}} \leq b_1
		\end{equation*}
		for some universal constants $b_0,b_1>0$ independent of $g,n$.
		
		\item {\rm{(\cite[Theorem 3.5]{Mirz13}).}}
		$$\frac{(2g-2+n)V_{g,n}}{V_{g,n+1}} = \frac{1}{4\pi^2} + O_n\left(\frac{1}{g}\right),$$
		$$\frac{V_{g,n}}{V_{g-1,n+2}} = 1+O_n\left(\frac{1}{g}\right),$$
		where the implied constants for $O_n(\cdot)$ are related to $n$ and independent of $g$.
	\end{enumerate}
\end{theorem}

\noindent Part (2) above can also be derived by \cite{MZ15} of Mirzakhani-Zograf in which the precise asymptotic behavior of $V_{g,n}$ is provided for given $n$.

Set $r=2g-2+n$. We also use the following quantity $W_r$ to approximate $V_{g,n}$: 
\begin{equation}
	W_{r}:=
	\begin{cases}
		V_{\frac{r}{2}+1,0}&\text{if $r$ is even},\\[5pt]
		V_{\frac{r+1}{2},1}&\text{if $r$ is odd}.
	\end{cases}
\end{equation}

\noindent The estimation about the sum of products of Weil-Petersson volumes can be found in e.g. \cite{Mirz13, MP19, GMST19, NWX23}. Here we use the following version:
\begin{theorem}[{\cite[Lemma 24]{NWX23}}]\label{thm sum-prod-V} 
Assume $q\geq 1$, $n_1,\cdots,n_q\geq 0$, $r\geq2$. Then there exist two universal constants $c,D>0$ such that
\begin{equation*}
	\sum_{\{g_i\}} V_{g_1,n_1}\cdots V_{g_q,n_q} \leq c \left(\frac{D}{r}\right)^{q-1} W_r,
\end{equation*}
where the sum is taken over all $\{g_i\}_{i=1}^q \sbs \N$ such that $2g_i-2+n_i \geq 1$ for all $i=1,\cdots,q$, and $\sum_{i=1}^q (2g_i-2+n_i) = r$. 
\end{theorem}

The following asymptotic behavior of $V_{g,n}(x_1,\cdots,x_n)$ was firstly studied in \cite[Proposition 3.1]{MP19}. We use the following version in 
 \cite{NWX23}. One may also see sharper ones in \cite{AM22}.
\begin{theorem}[{\cite[Lemma 20]{NWX23}}]\label{thm Vgn(x) small x}
	There exists a constant $c(n)>0$ independent of $g$ and $x_i$'s such that
	$$\left(1-c(n)\frac{\sum_{i=1}^n x_i^2}{g}\right)\prod_{i=1}^n \frac{\sinh(x_i/2)}{x_i/2} \leq \frac{V_{g,n}(x_1,\cdots,x_n)}{V_{g,n}} \leq \prod_{i=1}^n \frac{\sinh(x_i/2)}{x_i/2}.$$
\end{theorem}

\subsection{Figure-eight closed geodesics}
Let $X$ be a hyperbolic surface. We say a closed geodesic in $X$ is a \emph{figure-eight closed geodesic} if it has exactly one self-intersection point. Given a figure-eight closed geodesic $\alpha$, it is filling in a pair of pants $P(\gamma_1,\gamma_2,\gamma_3)$ with three geodesic boundaries as shown in Figure \ref{fig:figure-8}. The length $\ell(\alpha)$ of the figure-eight closed geodesic  $\alpha$ is given by (see e.g. \cite[Equation (4.2.3)]{Buser10}):
\begin{equation}\label{eq figure-8 length}
	\cosh\tfrac{\ell(\alpha)}{2}= \cosh\tfrac{\ell(\gamma_3)}{2} + 2\cosh\tfrac{\ell(\gamma_1)}{2}\cosh\tfrac{\ell(\gamma_2)}{2}.
\end{equation}
It is clear that $\ell(\alpha)\geq 2\arccosh 3$.
\begin{figure}[b]
	\centering
	\includegraphics[width=2.0 in]{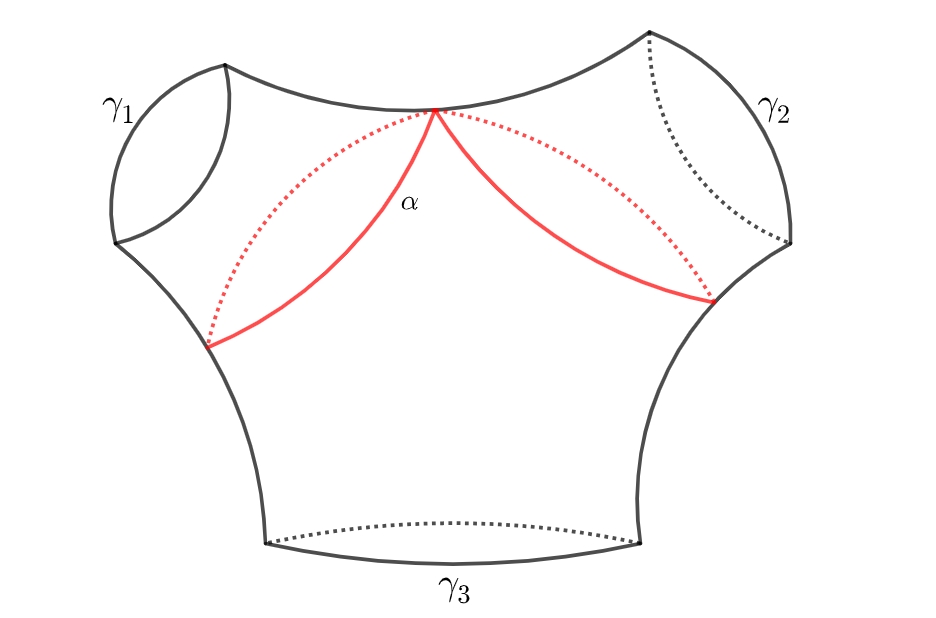}
	\caption{A figure-eight closed geodesic $\alpha$ in the pair of pants $P(\gamma_1,\gamma_2,\gamma_3)$.}
	\label{fig:figure-8}
\end{figure}

\begin{rem*}
	In a pair of pants $P(\gamma_1,\gamma_2,\gamma_3)$, there are exactly three different figure-eight closed geodesics. Here $\alpha$ is the figure-eight closed geodesic winding around $\gamma_1$ and $\gamma_2$ as shown in Figure \ref{fig:figure-8}.
\end{rem*}

It is known (see e.g. \cite[Lemma 5.2]{Tor23},  or \cite[Section 5.2]{Buser10} ) that the length of the shortest figure-eight closed geodesic in any $X\in \sM_g$ is bounded.
\begin{lemma}\label{thm f-8 <clogg}
	There exists a universal constant $c>0$ independent of $g$ such that for any $X\in\M_g$, the shortest figure-eight closed geodesic in $X$ has length $\leq c\log g$.
\end{lemma}
\bp[Outline of the proof of Lemma \ref{thm f-8 <clogg}]
Let $\gamma\subset X$ be a systolic curve. It is known that $\ell_\gamma(X)\prec \log g$. Next, consider the maximal collar around $\gamma$ and then one may get a pair of pants such that its boundary contains $\gamma$ and each of the three boundary geodesics has length $\prec \log g$. Then the conclusion follows from \eqref{eq figure-8 length}. One may see \cite{Tor23, Buser10} for more details. 
\ep

\begin{rem*}
From \cite[Theorem 4]{NWX23}, we know that the growth rate $\log g$ in the upper bound in Lemma \ref{thm f-8 <clogg} holds for generic hyperbolic surfaces in $\sM_g$ as $g\to \infty$. In this paper, we study its precise asymptotic behavior.
\end{rem*}

Recall that  the \emph{non-simple systole} $\nsys(X)$ of $X\in \sM_g$ is defined as
\[\ell^{ns}_{sys}(X)=\min\{\ell_{\alpha}(X); \ \textit{$\alpha\subset X$ is a non-simple closed geodesic} \}.\]

\noindent We focus on figure-eight closed geodesics because the non-simple systole of a hyperbolic surface is always achieved by a figure-eight closed geodesic (see e.g. \cite[Theorem 4.2.4]{Buser10}). That is,
\[\ell^{ns}_{sys}(X)=\textit{length of a shortest figure-eight closed geodesic in $X$.}\]
In particular, as introduced above, we have
\[\sup\limits_{X\in \sM_g}\ell^{ns}_{sys}(X)\asymp \log g.\]

\subsection{Three counting results on closed geodesics}
In this subsection, we mainly introduce three results about counting closed geodesics in hyperbolic surfaces. In this paper, we only consider primitive closed geodesics without orientations.

\subsubsection{On closed hyperbolic surfaces}
Firstly, by the Collar Lemma (see e.g. \cite[Theorem 4.1.6]{Buser10}), we have that a closed hyperbolic surface of genus $g$ has most $3g-3$  pairwisely disjoint simple closed geodesics of length $\leq 2\arcsinh 1 \approx 1.7627$. It is also known from  \cite[Theorem 6.6.4]{Buser10} that for all $L>0$ and $X\in \sM_g$, there are at most $(g-1)e^{L+6}$ closed geodesics in $X$ of length $\leq L$ which are not iterates of closed geodesics of length $\leq 2\arcsinh 1$. As a consequence, we have 
\begin{theorem}\label{thm count ge^L upp}
For any $L>0$ and $X\in\M_g$, there are at most $(g-1)e^{L+7}$ primitive closed geodesics in $X$ of length $\leq L$. 
\end{theorem}

\subsubsection{On compact hyperbolic surfaces with geodesic boundaries}\label{ss-count-2}
For compact hyperbolic surfaces with geodesic boundaries, the following result for filling closed multi-geodesics is useful.  
\begin{definition}\label{def fill k-tuple}
	Let $Y\in\M_{g,n}(L_1,\cdots,L_n)$ be a hyperbolic surface with boundaries. Let $\Gamma=(\gamma_1,\cdots,\gamma_k)$ be an ordered $k$-tuple where $\gamma_i$'s are non-peripheral closed geodesics in $Y$. We say $\Gamma$ is \emph{filling} in $Y$ if each connected component of the complement $Y\setminus \bigcup_{i=1}^k \gamma_i$ is homeomorphic to either a disk or a cylinder which is homotopic to a boundary component of $Y$.
	
	In particular, a filling $1$-tuple is a filling closed geodesic in $Y$.
	Define the length of a $k$-tuple $\Gamma=(\gamma_1,\cdots,\gamma_k)$ to be the total length of $\gamma_i$'s, that is, 
	$$\ell_\Gamma(Y):=\sum_{i=1}^k \ell_{\gamma_i}(Y).$$
	For $L >0$, define the counting function $N_k^{\text{fill}}(Y,L)$ by
	$$N_k^{\text{fill}}(Y,L):= \#\left\{\Gamma=(\gamma_1,\cdots,\gamma_k);\ 
	\begin{aligned}
		&\Gamma\ \text{is a filling}\ k\text{-tuple in}\ Y \\
		&\text{and}\ \ell_{\Gamma}(Y)\leq L
	\end{aligned} \right\}.$$
\end{definition}

\begin{theorem}[{\cite[Theorem 4]{WX22-GAFA}} or {\cite[Theorem 18]{WX22pgt}}]\label{thm count fill k-tuple}
	For any $k\in\Z_{\geq 1}$, $0<\eps<\frac{1}{2}$ and $m=2g-2+n\geq 1$, there exists
	a constant $c(k,\eps,m)>0$ only depending on $k,\eps$ and $m$ such that for all $L>0$ and any compact hyperbolic surface $Y$ of genus $g$ with $n$ boundary simple closed geodesics, the following holds:
	$$N_k^{\text{fill}}(Y,L)\leq c(k,\eps,m)\cdot (1+L)^{k-1} e^{L-\frac{1-\eps}{2}\ell(\partial Y)},$$
	where $\ell(\partial Y)$ is the total length of the boundary closed geodesics of $Y$.
\end{theorem}

\subsubsection{On $S_{1,2}$ and $S_{0,4}$}
For the cases of $S_{1,2}$ and $S_{0,4}$, we will apply the McShane-Mirzakhani identity to give more refined counting results on certain specific types of simple closed geodesics. The McShane-Mirzakhani identity states as:
\begin{theorem}[{\cite[Theorem 1.3]{Mirz07}}] \label{thm McShane id}
	For $Y\in\M_{g,n}(L_1,\cdots,L_n)$ with $n$ geodesic boundaries $\beta_1,\cdots,\beta_n$ of length $L_1,\cdots,L_n$ respectively, we have
	\begin{equation*}
		\sum_{\{\gamma_1,\gamma_2\}} \mathcal D(L_1, \ell(\gamma_1), \ell(\gamma_2)) +
		\sum_{i=2}^n \sum_\gamma \mathcal R(L_1,L_i,\ell(\gamma)) = L_1
	\end{equation*}
	where the first sum is over all unordered pairs of simple closed geodesics $\{\gamma_1, \gamma_2\}$ bounding a pair of pants with $\beta_1$, and the second sum is over all simple closed geodesics $\gamma$ bounding a pair of pants with $\beta_1$ and $\beta_i$. Here $\mathcal D$ and $\mathcal R$ are given by
	\begin{equation*}
		\mathcal D(x,y,z) = 2\log\left( \frac{e^{\frac{x}{2}}+e^{\frac{y+z}{2}}}{e^{\frac{-x}{2}}+e^{\frac{y+z}{2}}} \right)
	\end{equation*}
and
	\begin{equation*}
		\mathcal R(x,y,z) = x - \log\left( \frac{\cosh(\frac{y}{2})+\cosh(\frac{x+z}{2})}{\cosh(\frac{y}{2})+\cosh(\frac{x-z}{2})} \right).
	\end{equation*}
\end{theorem}

The functions $\mathcal D(x,y,z)$ and $\mathcal R(x,y,z)$ have the following elementary properties:
\begin{lemma}[{\cite[Lemma 27]{NWX23}}] \label{thm estimation R,D}
	Assume that $x,y,z> 0$, then the following properties hold:
	\begin{enumerate}[1.]
		\item $\mathcal R(x,y,z)\geq 0$ and $\mathcal D(x,y,z)\geq0$.
		\item $\mathcal R(x,y,z)$ is decreasing with respect to $z$ and increasing with respect to $y$. $\mathcal D(x,y,z)$ is decreasing with respect to $y$ and $z$ and increasing with respect to $x$.
		\item We have
		\begin{equation*}
			\frac{x}{\mathcal R(x,y,z)} \leq 100(1+x)\left(1+e^{\frac{z}{2}}e^{-\frac{x+y}{2}}\right)
		\end{equation*}
		and
		\begin{equation*}
			\frac{x}{\mathcal D(x,y,z)} \leq 100(1+x)\left(1+e^{\frac{y+z}{2}}e^{-\frac{x}{2}}\right).
		\end{equation*}
	\end{enumerate}
\end{lemma}

As a direct consequence  of Theorem \ref{thm McShane id} and  the monotonicity in Part (2) of Lemma \ref{thm estimation R,D}, we have 
\begin{theorem}
 On a surface $Y\in \M_{0,4}(L_1,L_2,L_3,L_4)$, the number of simple closed geodesics of length $\leq L$ which bound a pair of pants with the two boundaries of lengths $L_1$ and $L_2$ has the upper bound
\begin{equation}\label{eq counting in S04}
	\leq \min\left\{ \frac{L_1}{\mathcal R(L_1,L_2,L)}, \ \frac{L_2}{\mathcal R(L_2,L_1,L)},\ \frac{L_3}{\mathcal R(L_3,L_4,L)},\ \frac{L_4}{\mathcal R(L_4,L_3,L)} \right\}.
\end{equation}

 On a surface $Y\in \M_{1,2}(L_1,L_2)$, the number of simple closed geodesics of length $\leq L$ which bound a pair of pants with the two boundaries of lengths $L_1$ and $L_2$ has the upper bound
\begin{equation}\label{eq counting in S12 single}
	\leq \min\left\{ \frac{L_1}{\mathcal R(L_1,L_2,L)},\ \frac{L_2}{\mathcal R(L_2,L_1,L)} \right\}.
\end{equation}

 On a surface $Y\in \M_{1,2}(L_1,L_2)$, the number of unordered pairs of simple closed geodesics of total length $\leq L$ which bound a pair of pants with the boundary of length $L_1$ has the upper bound
\begin{equation}\label{eq counting in S12 pairs}
	\leq \min\left\{ \frac{L_1}{\mathcal D(L_1,L,0)},\ \frac{L_2}{\mathcal D(L_2,L,0)} \right\}.
\end{equation}
\end{theorem}
\bp
These three bounds follow from Theorem \ref{thm McShane id} and the monotonicity of $\mathcal D(x,y,z)$ and $\mathcal R(x,y,z)$. For \eqref{eq counting in S12 pairs}, we also apply the fact that $\mathcal D(x,y,z)$ only depends on $y+z$ but not on $y,z$ respectively. 
\ep

Then by applying the estimates in Lemma \ref{thm estimation R,D}, one may get upper bounds for \eqref{eq counting in S04}, \eqref{eq counting in S12 single} and \eqref{eq counting in S12 pairs}.

\section{Lower bound}\label{s-low}

In this section we compute the expected value of the number of figure-eight closed geodesics of length $\leq L$ over $\sM_g$, and hence deduce the lower bound of the non-simple systole for random hyperbolic surfaces.

For $X\in \sM_g$ and $L>0$, denote by 
\begin{equation}
	\sN_{\feight}(X,L) := \left\{\alpha\sbs X ;\ 
	\begin{array}{l}
		\alpha\ \text{is a figure-eight closed geodesic in}\ X  \\
		\text{and}\ \ell(\alpha)\leq L
	\end{array} \right\} \nonumber
\end{equation}
and 
\begin{equation}
	N_{\feight}(X,L) := \#
 \sN_{\feight}(X,L). \nonumber
\end{equation}
 Then $\ell_{sys}^{ns}(X)> L$ if and only if $N_{\feight}(X,L)=0$. 

The main result of this section is as follows.
\begin{proposition}\label{thm E[N_f8]=}
	For any $L\geq 2\arccosh 3$ and $g>2$, 
	\begin{equation*}
		\E[N_{\feight}(X,L)] = \frac{1}{8\pi^2 g}(L-3-4\log 2)e^L +O\left(\tfrac{1}{g} L^2 e^{\frac{1}{2}L} + \tfrac{1}{g^2} L^{4} e^{L}\right)
	\end{equation*}
where the implied constant is independent of $L$ and $g$.
\end{proposition}

Based on Proposition \ref{thm E[N_f8]=}, we prove the following result, which is the lower bound in Theorem \ref{mt-1}.
\begin{theorem}\label{thm lower bound of ns-sys}
	For any function $\omega(g)$ satisfying \[\lim \limits_{g\to \infty}\omega(g)= +\infty \ \textit{and} \ \lim \limits_{g\to \infty}\frac{\omega(g)}{\log\log g} = 0,\]  we have 
	$$
	\lim \limits_{g\to \infty} \Prob\left(X\in \M_g ;\  \lns(X) > \log g -\log\log g - \omega(g) \right)=1.
	$$ 
\end{theorem}
\begin{proof}
	Taking $L=L_g=\log g -\log\log g - \omega(g)$ in Proposition \ref{thm E[N_f8]=}, it is clear that $$\E[N_{\feight}(X,L_g)]\to 0, \textit{ as $g\to\infty$}.$$ Then it follows from Markov's inequality that
	$$\Prob\left(X\in \sM_g; \ N_{\feight}(X,L_g)\geq 1 \right)\leq \E[N_{\feight}(X,L_g)]\to 0, \ \textit{as $g\to \infty$}.$$
This also means that
\[\lim\limits_{g\to \infty}\Prob\left(X\in \sM_g; \ N_{\feight}(X,L_g)=0 \right)=1.\] Then the conclusion follows because for any $X\in \sM_g$, $\ell^{ns}_{sys}(X)$ is equal to the length of a shortest figure-eight closed geodesic in $X$. 
\end{proof}

\begin{figure}[htbp]
  \centering
  \begin{subfigure}[b]{0.45\linewidth}
      \centering
      \includegraphics[width=\linewidth]{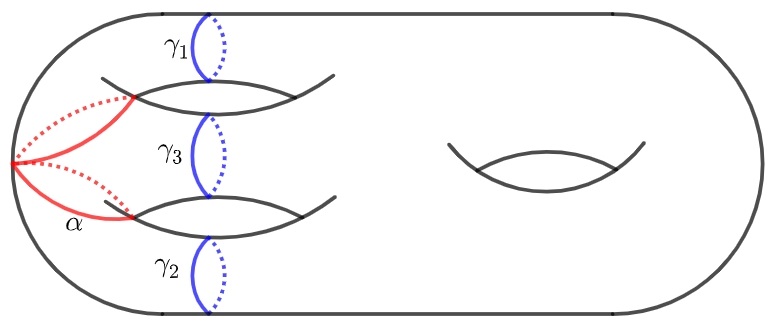}
      \caption{$X\setminus \overline{P(\gamma_1,\gamma_2,\gamma_3)}\cong S_{1,3}$.}
      \label{}
  \end{subfigure}
  \hfill
  \begin{subfigure}[b]{0.45\linewidth}
      \centering
      \includegraphics[width=\linewidth]{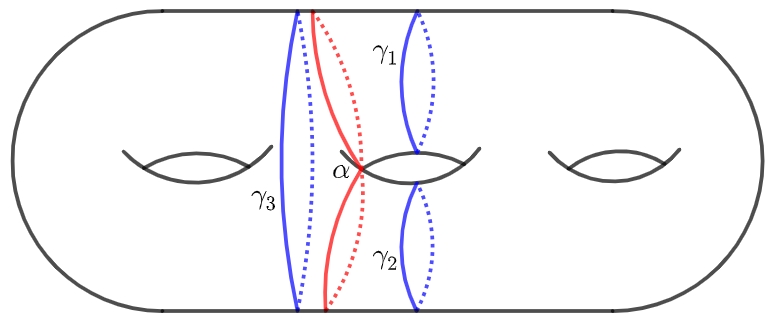}
      \caption{$X\setminus \overline{P(\gamma_1,\gamma_2,\gamma_3)}\cong S_{1,1}\cup S_{1,2}$.}
      \label{}
  \end{subfigure}
  \\
  \begin{subfigure}[b]{0.45\linewidth}
      \centering
      \includegraphics[width=\linewidth]{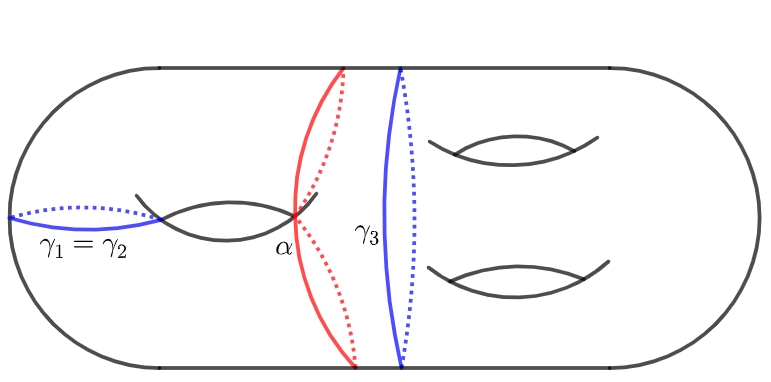}
      \caption{$X\setminus \overline{P(\gamma_1,\gamma_2,\gamma_3)}\cong S_{2,1}$.}
      \label{}
  \end{subfigure}
  \hfill
  \begin{subfigure}[b]{0.45\linewidth}
      \centering
      \includegraphics[width=\linewidth]{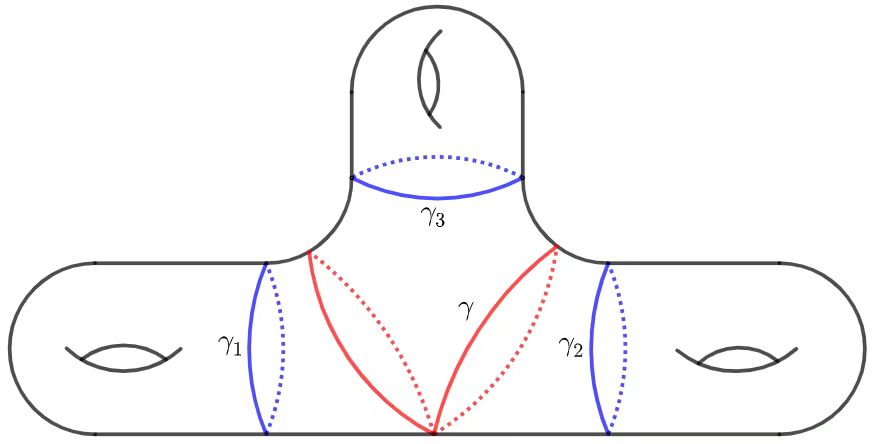}
      \caption{$X\setminus \overline{P(\gamma_1,\gamma_2,\gamma_3)}\cong S_{1,1}\cup S_{1,1}\cup S_{1,1}$.}
      \label{}
  \end{subfigure}
  \caption{Figure-eight geodesics}
  \label{fig:figure-8-in-surface}
\end{figure}

 A figure-eight closed geodesic $\alpha\subset X\in \sM_g$ is always filling in a unique pair of pants $P(\gamma_1,\gamma_2,\gamma_3)$ as shown in Figure \ref{fig:figure-8}. If two of the boundary geodesics of $P(\gamma_1,\gamma_2,\gamma_3)$ coincide in $X$, then the completion $\overline{P(\gamma_1,\gamma_2,\gamma_3)}\subset X$ of $P(\gamma_1,\gamma_2,\gamma_3)$ is a hyperbolic torus with one geodesic boundary (see picture (C) in Figure \ref{fig:figure-8-in-surface}); otherwise $\overline{ P(\gamma_1,\gamma_2,\gamma_3)}\subset X$ is still a pair of pants. So the complement $X\setminus\overline{P(\gamma_1,\gamma_2,\gamma_3)}$ may have one or two, or three connected components (see pictures (A), (B), (D) in Figure \ref{fig:figure-8-in-surface} for illustrations). We classify all figure-eight closed geodesics by the topology of $X\setminus \overline{P(\gamma_1,\gamma_2,\gamma_3)}$. For hyperbolic surface $X$ and  $L>0$, denote
\begin{eqnarray*}
	\sN_{\feight}^{(g-2,3)}(X,L)&:=&\left\{\alpha\in\sN_{\feight}(X,L);\ X\setminus\overline{P(\gamma_1,\gamma_2,\gamma_3)}\cong S_{g-2,3} \right\}, 
\end{eqnarray*}
\begin{eqnarray*}
	\sN_{\feight}^{(g-1,1)}(X,L)&:=&\left\{\alpha\in\sN_{\feight}(X,L);\ X\setminus \overline{P(\gamma_1,\gamma_2,\gamma_3)}\cong S_{g-1,1} \right\}, 
\end{eqnarray*}
\begin{eqnarray*}
	\sN_{\feight}^{(g_1,1)(g_2,2)}(X,L):=\left\{\alpha\in\sN_{\feight}(X,L);\ \begin{matrix}X\setminus \overline{P(\gamma_1,\gamma_2,\gamma_3)}\\ \cong S_{g_1,1}\cup S_{g_2,2} \end{matrix}\right\}, \nonumber
\end{eqnarray*}
\begin{eqnarray*}
	\sN_{\feight}^{(g_1,1)(g_2,1)(g_3,1)}(X,L):=\left\{\alpha\in\sN_{\feight}(X,L);\ \begin{matrix}X\setminus \overline{P(\gamma_1,\gamma_2,\gamma_3)}\\ \cong S_{g_1,1}\cup S_{g_2,1}\cup S_{g_3,1} \end{matrix}\right\},\nonumber
\end{eqnarray*}
and
\begin{eqnarray*}
	&N_{\feight}^{(g-2,3)}(X,L):= \#\sN_{\feight}^{(g-2,3)}(X,L), \\
	&N_{\feight}^{(g-1,1)}(X,L):= \#\sN_{\feight}^{(g-1,1)}(X,L), \\
	&N_{\feight}^{(g_1,1)(g_2,2)}(X,L):= \#\sN_{\feight}^{(g_1,1)(g_2,2)}(X,L), \\
	&N_{\feight}^{(g_1,1)(g_2,1)(g_3,1)}(X,L):= \#\sN_{\feight}^{(g_1,1)(g_2,1)(g_3,1)}(X,L),
\end{eqnarray*}
where $(g_1,g_2)$ satisfies
$$g_1+g_2=g-1\text{ and }g_1,g_2\geq 1;$$ and $(g_1,g_2,g_3)$ satisfies  $$g_1+g_2+g_3=g\text{ and }g_1\geq g_2\geq g_3\geq 1.$$
Then 
\begin{equation}
\label{eq N_f8=}\begin{aligned}
	N_{\feight}(X,L) &= N_{\feight}^{(g-2,3)}(X,L) + \sum_{(g_1,g_2)} N_{\feight}^{(g_1,1)(g_2,2)}(X,L) \\
	&+N_{\feight}^{(g-1,1)}(X,L) +\sum_{(g_1,g_2,g_3)} N_{\feight}^{(g_1,1)(g_2,1)(g_3,1)}(X,L)
    \end{aligned}
\end{equation}
where $(g_1,g_2)$, $(g_1,g_2,g_3)$ are taken over all possible pairs of positive integers.

We now compute $\E\left[N_{\feight}^{(g-2,3)}\right]$, $\E\left[N_{\feight}^{(g-1,1)}\right]$ and sum of all possible \\ $\E\left[N_{\feight}^{(g_1,1)(g_2,2)}\right]$ and $\E\left[N_{\feight}^{(g_1,1)(g_2,1)(g_3,1)}\right]$ in the following Lemmas.

\begin{lemma}\label{thm E[figure-8] main term}
	For any $L\geq 2\arccosh 3$ and $g>2$, 
\begin{equation*}
	\E\left[N_{\feight}^{(g-2,3)}(X,L)\right] = \frac{1}{8\pi^2 g}(L-3-4\log 2)e^L +O\left(\tfrac{1}{g} L e^{\frac{1}{2}L} + \tfrac{1}{g^2} L^{4} e^{L}\right)
\end{equation*}
where the implied constant is independent of $L$ and $g$.
\end{lemma}

\begin{proof}

   Consider function $\ell:\mathbb{R}^3_{\geq 0}\to\mathbb{R}_{\geq 0}$ which is determined by $$\cosh\frac{\ell(x_1,x_2,x_3)}{2}=\cosh\frac{x_3}{2}+2\cosh\frac{x_1}{2}\cosh\frac{x_2}{2}$$
    and for $L>0$, define $$\phi_L(x_1,x_2,x_3)=\mathbf{1}_{[0,L]}(\ell(x_1,x_2,x_3)) + \mathbf{1}_{[0,L]}(\ell(x_3,x_1,x_2)) + \mathbf{1}_{[0,L]}(\ell(x_2,x_3,x_1)).$$
    Then from \eqref{eq figure-8 length}, one may conclude that for any pair of pants $P(\gamma_1,\gamma_2,\gamma_3)$, it contains exactly $\phi_L(\ell(\gamma_1),\ell(\gamma_2),\ell(\gamma_3))$ figure-eight closed geodesics of lengths $\leq L$. Conversely, for any figure-eight closed geodesic in $X$, it is contained in a unique pair of pants. Assume $\Gamma=(\alpha_1,\alpha_2,\alpha_3)$ is a simple closed multi-curve  such that $$X\setminus\{\alpha_1,\alpha_2,\alpha_3\}\cong S_{0,3}\cup S_{g-2,3}.$$ 
Then we have 
$$N_{\feight}^{(g-2,3)}(X,L)=\frac{1}{6}\sum\limits_{(\beta_1,\beta_2,\beta_3)\in\mathcal{O}_{\Gamma}}\phi_L(\ell_{\beta_1}(X),\ell_{\beta_2}(X),\ell_{\beta_3}(X))$$
where $\mathcal{O}_{\Gamma}$ is the orbit of $\Gamma$ introduced in subsection \ref{sec mirz int fml}.
    
    Applying Mirzakhani's integration formula Theorem \ref{thm Mirz int formula} to function $\phi_L(x_1,x_2,x_3)$ and $\Gamma$ together with the fact that $C_\Gamma = 1$, we have
	\begin{equation}
	\label{eq E[N_f8^(g-2,3)]}
    \begin{aligned}
		& \E\left[N_{\feight}^{(g-2,3)}(X,L)\right] \\
           =&\frac{1}{6} \E\left[\sum\limits_{(\beta_1,\beta_2,\beta_3)\in\mathcal{O}_{\Gamma}}\phi_L(\ell_{\beta_1}(X),\ell_{\beta_2}(X),\ell_{\beta_3}(X))\right]\\
		=& \frac{1}{6} \int_{x_1,x_2,x_3\geq 0} \phi_{L}(x_1,x_2,x_3)\cdot
        \frac{V_{g-2,3}(x_1,x_2,x_3)V_{0,3}(x_1,x_2,x_3)}{V_g}\\
        \cdot &x_1x_2x_3\ dx_1dx_2dx_3 \\
		=& \frac{1}{2} \int_{x_1,x_2,x_3\geq 0} \mathbf{1}_{[0,L]}(\ell(x_1,x_2,x_3))\cdot \frac{V_{g-2,3}(x_1,x_2,x_3)}{V_g} x_1x_2x_3\ dx_1dx_2dx_3 
        \end{aligned}
	\end{equation}
	where  in the last equation we apply that $V_{0,3}(x_1,x_2,x_3)=1$ and the product $V_{g-2,3}(x_1,x_2,x_3)x_1x_2x_3$ is symmetric with respect to $x_1,x_2,x_3$. By Part (2) of Theorem \ref{thm Vgn/Vgn+1} we have
\be\label{e-vol-1}
\frac{V_{g-2,3}}{V_g}=\frac{1}{8\pi^2 g}\cdot \left(1+O\left(\frac{1}{g}\right)\right).
\ene
This together with Theorem \ref{thm Vgn(x) small x} implies that 
\begin{equation}\label{eq V_g-2,3(x,y,z)}
\begin{aligned}
	&\frac{V_{g-2,3}(x_1,x_2,x_3)}{V_g} x_1x_2x_3 \\
    =& \frac{1}{\pi^2 g} \sinh\tfrac{x_1}{2}\sinh\tfrac{x_2}{2}\sinh\tfrac{x_3}{2} \cdot \left(1+O\left(\tfrac{1+x_1^2+x_2^2+x_3^2}{g}\right)\right).
    \end{aligned}
\end{equation}
From \eqref{eq V_g-2,3(x,y,z)} and the fact that $\ell(x_1,x_2,x_3)>\tfrac{1}{2}(x_1+x_2+x_3)$, we have
\begin{equation}\label{eq E[N_f8^(g-2,3)] remainder}
\begin{aligned}
	&\left|\frac{1}{2} \int_{\{\ell(x_1,x_2,x_3)\leq L\}}\! \frac{1}{\pi^2 g}\!\sinh\tfrac{x_1}{2}\!\sinh\tfrac{x_2}{2}\!\sinh\tfrac{x_3}{2}\!\cdot O\left(\tfrac{1+x_1^2+x_2^2+x_3^2}{g}\right) dx_1dx_2dx_3\right| \\
	\prec& \frac{1}{g^2} \int_{\{x_1+x_2+x_3\leq 2L\}} \left(1+x_1^2+x_2^2+x_3^2\right) e^{\tfrac{1}{2}(x_1+x_2+x_3)} dx_1dx_2dx_3 \\
    \prec& \frac{1}{g^2} L^{4} e^{L}. 
\end{aligned}
\end{equation}
Now we consider the term
\begin{equation}
	\frac{1}{2} \int_{x_1,x_2,x_3\geq 0} \mathbf{1}_{[0,L]}(\ell(x_1,x_2,x_3)) \frac{1}{\pi^2 g}\sinh\tfrac{x_1}{2}\sinh\tfrac{x_2}{2}\sinh\tfrac{x_3}{2} dx_1dx_2dx_3.
\end{equation}
 Change the variables $(x_1,x_2,x_3)$ into $(x_1,x_2,t)$ with $t=\ell(x_1,x_2,x_3)$. By \eqref{eq figure-8 length}, 
\begin{equation}
	\tfrac{1}{2}\sinh\tfrac{t}{2}dt = \tfrac{1}{2}\sinh\tfrac{x_3}{2} dx_3 + * dx_1 + * dx_2. \nonumber
\end{equation}
So 
\begin{equation}\label{eq-e-1}
\begin{aligned}
	&\frac{1}{2} \int_{x_1,x_2,x_3\geq 0} \mathbf{1}_{[0,L]}(\ell(x_1,x_2,x_3)) \frac{1}{\pi^2 g}\sinh\tfrac{x_1}{2}\sinh\tfrac{x_2}{2}\sinh\tfrac{x_3}{2} dx_1dx_2dx_3 \\
	=& \int_{\textbf{Cond}}  \frac{1}{2\pi^2 g} \sinh\tfrac{x_1}{2}\sinh\tfrac{x_2}{2}\sinh\tfrac{t}{2} dx_1dx_2dt \nonumber 
    \end{aligned}
\end{equation}
where the integration region $\textbf{Cond}$ is
\begin{equation}
	\begin{cases}
		x_1,x_2,x_3\geq 0 \\ \ell(x_1,x_2,x_3)\leq L 
	\end{cases} 
	\Longleftrightarrow 
	\begin{cases}
		x_1\geq 0, \ \cosh\tfrac{x_1}{2}\leq \frac{\cosh\tfrac{t}{2}-1}{2\cosh\tfrac{x_2}{2}} \\
		x_2\geq 0, \ 2\cosh\tfrac{x_2}{2}\leq \cosh\tfrac{t}{2}-1 \\ 
		0\leq t\leq L,\ \cosh\tfrac{t}{2}\geq 3
	\end{cases}.\nonumber
\end{equation}
We consider the integral for $x_1,\ x_2$ and $t$ in order. First taking an integral for $x_1$,
\begin{equation}\label{eq int Cond_x}
	\int_{x_1\geq 0} \mathbf{1}_{\left\{\cosh\tfrac{x_1}{2}\leq \frac{\cosh\tfrac{t}{2}-1}{2\cosh\tfrac{x_2}{2}}\right\}}(x_1,x_2,t) \cdot  \frac{1}{2\pi^2 g}\sinh\tfrac{x_1}{2} dx_1 
	= \frac{1}{\pi^2 g} \left(\frac{\cosh\tfrac{t}{2}-1}{2\cosh\tfrac{x_2}{2}}-1\right). \nonumber
\end{equation}
Then taking an integral for $x_2$,
\begin{eqnarray}\label{eq int Cond_y}
	&&\int_{x_2\geq 0} \mathbf{1}_{\left\{ 2\cosh\tfrac{x_2}{2}\leq \cosh\tfrac{t}{2}-1 \right\}}(x_2,t) \cdot \sinh\tfrac{x_2}{2} \cdot\frac{1}{\pi^2 g} \left(\frac{\cosh\tfrac{t}{2}-1}{2\cosh\tfrac{x_2}{2}}-1\right) dy \nonumber \\
	&&= \frac{1}{2\pi^2 g}\left(\cosh\tfrac{t}{2}-1\right) \left(2\log\left(\tfrac{\cosh\tfrac{t}{2}-1}{2}\right) \right)  -  \frac{2}{\pi^2 g}\left(\tfrac{\cosh\tfrac{t}{2}-1}{2}-1\right) \nonumber\\
	&&= \frac{4}{\pi^2 g}\sinh^2\tfrac{t}{4} \log\left(\sinh\tfrac{t}{4}\right)- \frac{2}{\pi^2 g}\sinh^2\tfrac{t}{4} +\frac{2}{\pi^2 g}. \nonumber
\end{eqnarray}
 It is clear that $$\log\left(\sinh\left(\tfrac{t}{4}\right)\right) =\tfrac{t}{4}-\log 2 +O\left(e^{-\frac{t}{2}}\right).$$
Finally taking an integral for $t$, we have
\begin{align}\label{eq E[N_f8^(g-2,3)] main}
\nonumber	&\int_{\textbf{Cond}}  \frac{1}{2\pi^2 g} \sinh\tfrac{x_1}{2}\sinh\tfrac{x_2}{2}\sinh\tfrac{t}{2} dx_1dx_2dt \\
\nonumber	=& \int_{2\arccosh 3}^{L} \sinh\tfrac{t}{2}\bigg( \frac{4}{\pi^2 g} \sinh^2\tfrac{t}{4} \log\left(\sinh\tfrac{t}{4}\right) \\
	&\quad\quad  - \frac{2}{\pi^2 g}\sinh^2\tfrac{t}{4} +\frac{2}{\pi^2 g} \bigg) dt\\
\nonumber	=& \int_{2\arccosh 3}^{L} \sinh\tfrac{t}{2} \bigg( \frac{4}{\pi^2 g}\sinh^2\tfrac{t}{4}\cdot \left(\tfrac{t}{4}-\log 2 +O\left(e^{-\frac{t}{2}}\right)\right)\\
\nonumber	&\quad\quad  - \frac{2}{\pi^2 g}\sinh^2\tfrac{t}{4} +\frac{2}{\pi^2 g} \bigg) dt \\
\nonumber	=& \frac{1}{8\pi^2 g}(L-3-4\log 2)e^L +O\left(\tfrac{1}{g}L e^{\frac{L}{2}}\right). 
\end{align}
Combining \eqref{eq E[N_f8^(g-2,3)]}, \eqref{eq E[N_f8^(g-2,3)] remainder}, \eqref{eq-e-1} and \eqref{eq E[N_f8^(g-2,3)] main}, we obtain 
\begin{equation}
	\E[N_{\feight}^{(g-2,3)}(X,L)] = \frac{1}{8\pi^2 g}(L-3-4\log 2)e^L +O\left(\tfrac{1}{g}L e^{\frac{L}{2}} + \tfrac{1}{g^2} L^{4} e^{L}\right) \nonumber
\end{equation}
as desired.
\end{proof}

\begin{rem*}
Similar computations were taken in \cite{AM23}.
\end{rem*}

\begin{lemma}\label{thm E[figure-8] small term}
For any $L\geq 2\arccosh 3$ and $g>2$ we have 
\begin{equation*}
	\E\left[N_{\feight}^{(g-1,1)}(X,L)\right] \prec \frac{1}{g}L^2 e^{\frac{L}{2}},
\end{equation*}
\begin{equation*}
	\sum_{(g_1,g_2)} \E\left[N_{\feight}^{(g_1,1)(g_2,2)}(X,L)\right] \prec \frac{1}{g^2}L^2 e^L,
\end{equation*}
\begin{equation*}
	\sum_{(g_1,g_2,g_3)} \E\left[N_{\feight}^{(g_1,1)(g_2,1)(g_3,1)}(X,L)\right] \prec \frac{1}{g^3}L^2 e^L,
\end{equation*}
where the first sum is taken over all $(g_1,g_2)$ with $g_1+g_2=g-1$ and $g_1,g_2\geq 1$; the second sum is taken over all $(g_1,g_2,g_3)$ with $g_1+g_2+g_3=g$ and $g_1\geq g_2\geq g_3\geq 1$. The implied constants are independent of $L$ and $g$.
\end{lemma}

\begin{proof}
Consider the following three special types of simple closed multi-curves:
\begin{enumerate}
    \item $\Gamma_0=(\gamma_1,\gamma_3
    )$ such that $X\setminus\{\gamma_1,\gamma_3\}\cong S_{0,3}\cup S_{g-1,1}$ where $\gamma_3$ is the  boundary curve of the part $S_{g-1,1}$ (see  picture (C) in Figure \ref{fig:figure-8-in-surface});
    \item $\Gamma_{g_1,g_2}=(\gamma_1,\gamma_2,\gamma_3)$ such that $X\setminus\{\gamma_1,\gamma_2,\gamma_3\}\cong S_{0,3}\cup S_{g_1,1}\cup S_{g_2,2}$ and $\gamma_3$ is the boundary curve of the part $S_{g_1,1}$, $\gamma_1,\gamma_2$ are the boundary curves of the part $S_{g_2,2}$. $(g_1,g_2)$ is a pair of positive integers such that $g_1+g_2=g-1$ (see picture (B) in Figure \ref{fig:figure-8-in-surface}); 
    \item $\Gamma_{g_1,g_2,g_3}=(\gamma_1,\gamma_2,\gamma_3)$ such that $X\setminus\{\gamma_1,\gamma_2,\gamma_3\}\cong S_{0,3}\cup S_{g_1,1}\cup S_{g_2,1}\cup S_{g_3,1}$ and $\gamma_i$ is the boundary curve of the part $S_{g_i,1}\ (i=1,2,3)$. $(g_1,g_2,g_3)$ is a pair of positive integers such that $g_1+g_2+g_3=g$ (see picture (D) in Figure \ref{fig:figure-8-in-surface}).
\end{enumerate}
 From \eqref{eq figure-8 length}, one may deduce that if a pair of pants $P(\gamma_1,\gamma_2,\gamma_3)$ contains a figure-eight closed geodesic of length $\leq L$, then  
\begin{align}\label{e-condition}
\sum\limits_{i=1}^3\ell(\gamma_i)\leq 2L\text{ and }\ell(\gamma_i)\leq L\text{ for }i=1,2,3.
\end{align}
Since for any pair of pants $P(\gamma_1,\gamma_2,\gamma_3)$, it contains exactly three figure-eight closed geodesics, together with \eqref{e-condition}, 
we have
\begin{equation*}
	N_{\feight}^{(g-1,1)}(X,L)\leq 3\cdot\sum\limits_{(\gamma_1,\gamma_3)\in\mathcal{O}_{\Gamma_0}}1_{[0,L]^2}(\ell(\gamma_1), \ell(\gamma_3));
\end{equation*} 
\begin{equation*}
	N_{\feight}^{(g_1,1)(g_2,2)}(X,L)\leq 3\cdot \sum\limits_{(\gamma_1,\gamma_2,\gamma_3)\in\mathcal{O}_{\Gamma_{g_1,g_2}}}1_{[0,2L]}\left(\ell(\gamma_1)+\ell(\gamma_2)+\ell(\gamma_3)\right);
\end{equation*}
\begin{equation*}
	N_{\feight}^{(g_1,g_2,g_3)}(X,L)\leq 3\cdot
    \sum\limits_{(\gamma_1,\gamma_2,\gamma_3)\in\mathcal{O}_{\Gamma_{g_1,g_2,g_3}}}1_{[0,2L]}\left(\ell(\gamma_1)+\ell(\gamma_2)+\ell(\gamma_3)\right).
\end{equation*}

Applying Mirzakhani's integration formula Theorem \ref{thm Mirz int formula} to function $$\phi_L(x_1,x_2)=1_{[0,L]^2}(x_1,x_2)$$ and simple closed multi-curve $\Gamma_0$, together with Theorem \ref{thm Vgn(x) small x}, we have
\begin{align}\label{e-q-+1}
	\E\left[N_{\feight}^{(g-1,1)}(X,L)\right] 
	\leq& 3 \int_{\substack{0\leq x_1\leq L \\ 0\leq x_2\leq L}} \frac{V_{g-1,1}(x_1)V_{0,3}(x_1,x_2,x_2)}{V_g} x_1x_2\ dx_1dx_2  \nonumber  \\
	\prec& \int_{\substack{0\leq x_1\leq L \\ 0\leq x_2\leq L}} x_2\sinh\tfrac{x_1}{2}\cdot\frac{V_{g-1,1}}{V_g} dx_1dx_2\\
	\prec& \frac{V_{g-1,1}}{V_g} L^2 e^{\frac{1}{2}L}.  \nonumber 
\end{align}

Applying Mirzakhani's integration formula Theorem \ref{thm Mirz int formula} to function $$\psi_L(x_1,x_2,x_3)=1_{[0,2L]}(x_1+x_2+x_3)$$ and simple closed multi-curves $\Gamma_{g_1,g_2}$ for all possible $(g_1,g_2)$'s, together with Theorem \ref{thm Vgn(x) small x}, we have
\begin{equation}\begin{aligned}
	&\sum_{(g_1,g_2)} \E\left[N_{\feight}^{(g_1,1)(g_2,2)}(X,L)\right
    ] \\
	\leq& 3 \sum_{(g_1,g_2)} \int_{\substack{x_1,x_2,x_3\geq0 \\ x_1+x_2+x_3\leq 2L}} \frac{V_{g_1,1}(x_1)V_{g_2,2}(x_2,x_3)V_{0,3}(x_1,x_2,x_3)}{V_g}\\
    \cdot&x_1x_2x_3\ dx_1dx_2dx_3 \\
	\prec& \sum_{(g_1,g_2)}  \int_{\substack{x_1,x_2,x_3\geq0 \\ x_1+x_2+x_3\leq 2L}} \sinh\tfrac{x_1}{2}\sinh\tfrac{x_2}{2}\sinh\tfrac{z_3}{2} \frac{V_{g_1,1}V_{g_2,2}}{V_g} dx_1dx_2dx_3  \\
	\prec& \sum_{(g_1,g_2)}  \frac{V_{g_1,1}V_{g_2,2}}{V_g} L^2 e^{L}.  
\end{aligned}\end{equation}

Applying Mirzakhani's integration formula Theorem \ref{thm Mirz int formula} to function $$\psi_L(x_1,x_2,x_3)=1_{[0,2L]}(x_1+x_2+x_3)$$ and simple closed multi-curve $\Gamma_{g_1,g_2,g_3}$ for all possible $(g_1,g_2,g_3)$'s, together with Theorem \ref{thm Vgn(x) small x} and similar computations as above, we have
\begin{equation}
	\sum_{(g_1,g_2,g_3)} \E\left[N_{\feight}^{(g_1,1)(g_2,1)(g_3,1)}(X,L)\right] \prec \sum_{(g_1,g_2,g_3)}  \frac{V_{g_1,1}V_{g_2,1}V_{g_3,1}}{V_g} L^2 e^{L}.
\end{equation}
Applying Theorem \ref{thm Vgn/Vgn+1} and Theorem \ref{thm sum-prod-V} for $q=2,3$ we have
\begin{equation}
	\frac{V_{g-1,1}}{V_g} \prec \frac{1}{g},
\end{equation}
\begin{equation}
	\sum_{(g_1,g_2)} \frac{V_{g_1,1}V_{g_2,2}}{V_g} \prec \frac{1}{g} \frac{W_{2g-3}}{V_g} \prec \frac{1}{g^2},
\end{equation}
\begin{equation}\label{e-q--1}
	\sum_{(g_1,g_2,g_3)}  \frac{V_{g_1,1}V_{g_2,1}V_{g_3,1}}{V_g} \prec \frac{1}{g^2} \frac{W_{2g-3}}{V_g} \prec \frac{1}{g^3}.
\end{equation}
Then the conclusion follows from all these equations \eqref{e-q-+1}-\eqref{e-q--1}.
\end{proof}

\begin{proof}[Proof of Proposition \ref{thm E[N_f8]=}]
    The conclusion clearly follows from \eqref{eq N_f8=}, Lemma \ref{thm E[figure-8] main term} and Lemma \ref{thm E[figure-8] small term}.
\end{proof}

\section{Upper bound}\label{s-upp}

In this section, we will prove the upper bound of the length of the non-simple systole for random hyperbolic surfaces. That is, we show   
\begin{theorem}\label{thm upper bound of ns-sys}
		For any function $\omega(g)$ satisfying \[\lim \limits_{g\to \infty}\omega(g)= +\infty \ \textit{and} \ \lim \limits_{g\to \infty}\frac{\omega(g)}{\log\log g} = 0,\]  we have 
	$$
	\lim \limits_{g\to \infty} \Prob\left(X\in \M_g ;\  \lns(X) < \log g -\log\log g +  \omega(g) \right)=1.
	$$ 
\end{theorem}

In order to prove Theorem \ref{thm upper bound of ns-sys}, it suffices to show that \begin{equation}\label{eq up f8}
     \lim_{g\to\infty}\Prob\left(X\in \M_g; \ N_{\feight}(X,L_g)=0 \right)=0
 \end{equation}
where $$L_g=\log g-\log\log g+\omega(g).$$ Instead of working on $N_{\feight}(X,L_g)$, we consider $\Nnumber(X,L_g)$ defined as follows.

 \begin{definition}
     For any $L>1$ and $X\in\mathcal{M}_g$, denote by 
$$\mathcal{N}_{(0,3),\star}^{(g-2,3)}(X,L)=\left\{(\gamma_1,\gamma_2,\gamma_3);\
\begin{matrix}(\gamma_1,\gamma_2,\gamma_3)\text{ is a pair of ordered simple closed}\\  \text{geodesic  such  that }X\setminus \bigcup_{i=1}^3\gamma_i\cong  S_{0,3} \cup S_{g-2,3},\\ \ell(\gamma_1)\leq L,\ \ell(\gamma_2)+\ell(\gamma_3)\leq L\\ \text{ and }\ell(\gamma_1),\ \ell(\gamma_2),\ \ell(\gamma_3)\geq 10\log L\end{matrix}\right\}$$
and 
$$N_{(0,3),\star}^{(g-2,3)}(X,L)=\#\mathcal{N}_{(0,3),\star}^{(g-2,3)}(X,L).$$
 \end{definition}
\begin{rem*}
    Since for any figure-eight closed geodesic, it is always contained in a pair of pants, it is natural to consider simple closed multi-geodesic $(\gamma_1,\gamma_2,\gamma_3)$ such that 
    \begin{align}\label{e-con03}
    X\setminus \bigcup_{i=1}^3 \gamma_i\cong S_{0,3}\cup S_{g-2,3}.
    \end{align}
\end{rem*}
For any ordered simple closed geodesic $(\gamma_1,\gamma_2,\gamma_3)$ satisfying \eqref{e-con03}, it bounds a pair of pants $P(\gamma_1,\gamma_2,\gamma_3)$. Consider the figure-eight closed geodesic $\alpha$ in $P(\gamma_1,\gamma_2,\gamma_3)$ which winds around $\gamma_2$ and $\gamma_3$. From Equation (\ref{eq figure-8 length}), we have 
$$\cosh\frac{\ell(\alpha)}{2}= \cosh\frac{\ell(\gamma_1)}{2} + 2\cosh\frac{\ell(\gamma_2)}{2}\cosh\frac{\ell(\gamma_3)}{2}.$$
If $\ell(\alpha)\leq L$, then we have
$$\ell(\gamma_1)\leq L\text{ and }\ell(\gamma_2)+\ell(\gamma_3)\leq L,$$
which is the second condition in the definition of $\Nset(X,L)$. 

For any $(\gamma_1,\gamma_2,\gamma_3)\in \Nset(X,L)$, we have
\begin{align*}
    e^{\frac{\ell(\alpha)}{2}}\leq 2e^{\frac{\ell(\gamma_1)}{2}}+4e^{\frac{\ell(\gamma_2)}{2}+\frac{\ell(\gamma_3)}{2}}\leq 6e^{\frac{L}{2}}
\end{align*}
and 
\begin{align}\label{length-f-8}
\ell(\alpha)\leq L+2\log 6.
\end{align}
It means that the pair of pants $P(\gamma_1,\gamma_2,\gamma_3)$ contains a figure-eight closed geodesic of length $\leq L+2\log 6$. It follows that
\begin{equation}
 \label{link f-8 with Nstar}\begin{aligned}
     &\Prob\left( X\in \M_g ;\ N_{\feight}(X,L_g)=0 \right)\\
     \leq &\Prob\left(X\in \M_g ;\ \Nnumber(X,L_g-2\log 6)=0 \right). \\
     \end{aligned}
 \end{equation}
 For any $L>1$, we view $\Nnumber(X,L)$ as a nonnegative integer-valued random variable on $\sM_g$. Then by the standard Cauchy-Schwarz inequality we have
\[\Prob\left(X\in \sM_g; \ \Nnumber(X,L)>0\right)\geq \frac{\E\left[\Nnumber(X,L)\right]^2}{\E\left[\big(\Nnumber(X,L)\big)^2\right]},\]
which implies
\begin{equation}\label{ineq Nset markov}
\begin{aligned}
&\ \ \ \  \Prob\left(X\in \sM_g; \ \Nnumber(X,L)=0   \right)\\
& \leq \frac{\E\left[\big(\Nnumber(X,L)\big)^2\right]-\E\left[\Nnumber(X,L)\right]^2}{\E\left[\Nnumber(X,L)\right]^2}.
\end{aligned} 
 \end{equation}
Such a second moment method has been used in \cite{MP19,NWX23}.
  \noindent For any $\Gamma=(\gamma_1,\gamma_2,\gamma_3)\in\mathcal{N}_{(0,3),\star}^{(g-2,3)}(X,L)$, denote by $P(\Gamma)$ the pair of pants with boundary geodesics $\gamma_1,\ \gamma_2$ and $\gamma_3$.  Set 
\begin{align*}
&\mathcal{A}(X,L)=\left\{(\Gamma_1,\Gamma_2)\in \left(\mathcal{N}_{(0,3),\star}^{(g-2,3)}(X,L)\right)^2;\ P(\Gamma_1)=P(\Gamma_2)\right\},\\
&\mathcal{B}(X,L)=\left\{(\Gamma_1,\Gamma_2)\in\left(\mathcal{N}_{(0,3),\star}^{(g-2,3)}(X,L)\right)^2; \ \overline{P(\Gamma_1)}\cap \overline{P(\Gamma_2)}=\emptyset \right\},\\
&\mathcal{C}(X,L)=\left\{(\Gamma_1,\Gamma_2)\in\left(\mathcal{N}_{(0,3),\star}^{(g-2,3)}(X,L)\right)^2;\ \begin{matrix} P(\Gamma_1)\neq P(\Gamma_2),\\ {P(\Gamma_1)}\cap {P(\Gamma_2)}\neq\emptyset\end{matrix}\right\},\\
&\mathcal{D}(X,L)=\left\{(\Gamma_1,\Gamma_2)\in\left(\mathcal{N}_{(0,3),\star}^{(g-2,3)}(X,L)\right)^2;\ \begin{matrix} \overline{P(\Gamma_1)}\cap \overline{P(\Gamma_2)}\neq\emptyset,\\
P(\Gamma_1)\cap P(\Gamma_2)=\emptyset,\end{matrix}\right\}.\\
\end{align*}
\begin{figure}[htbp]
  \centering
  \begin{subfigure}[b]{0.45\linewidth}
      \centering
      \includegraphics[width=\linewidth]{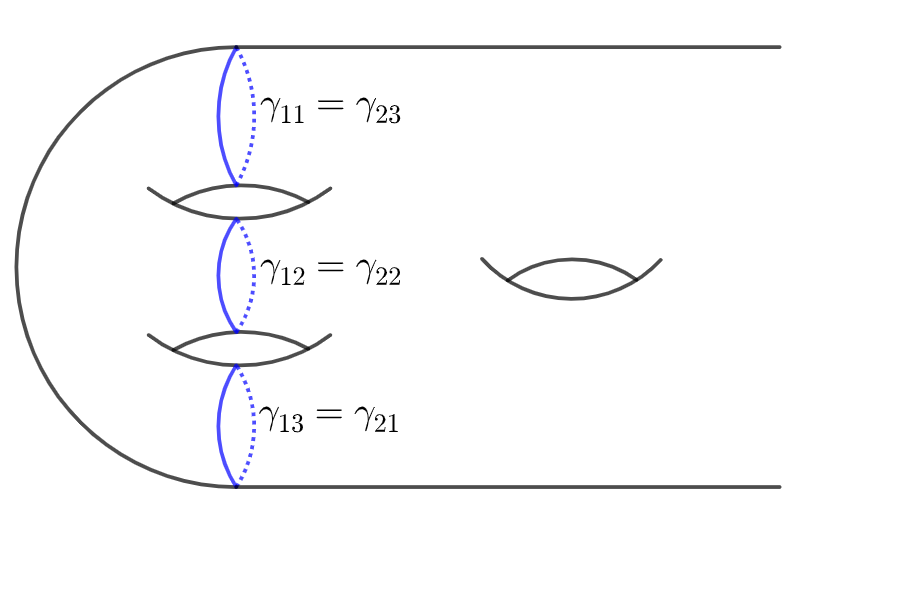}
      \caption{$\mathcal{A}(X,L)$.}
      \label{}
  \end{subfigure}
  \hfill
  \begin{subfigure}[b]{0.45\linewidth}
      \centering
      \includegraphics[width=\linewidth]{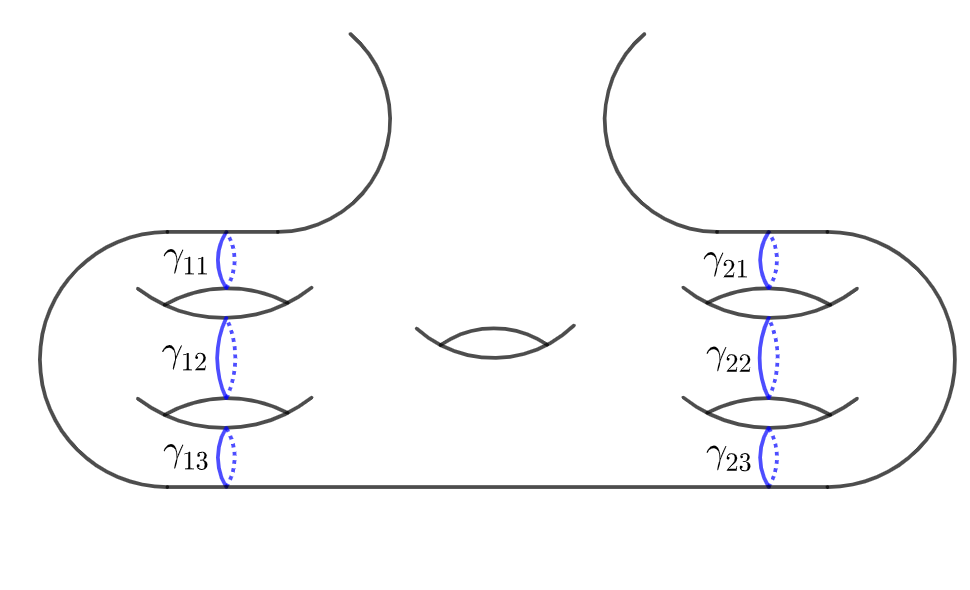}
      \caption{$\mathcal{B}(X,L)$.}
      \label{}
  \end{subfigure}
  \\
  \begin{subfigure}[b]{0.45\linewidth}
      \centering
      \includegraphics[width=\linewidth]{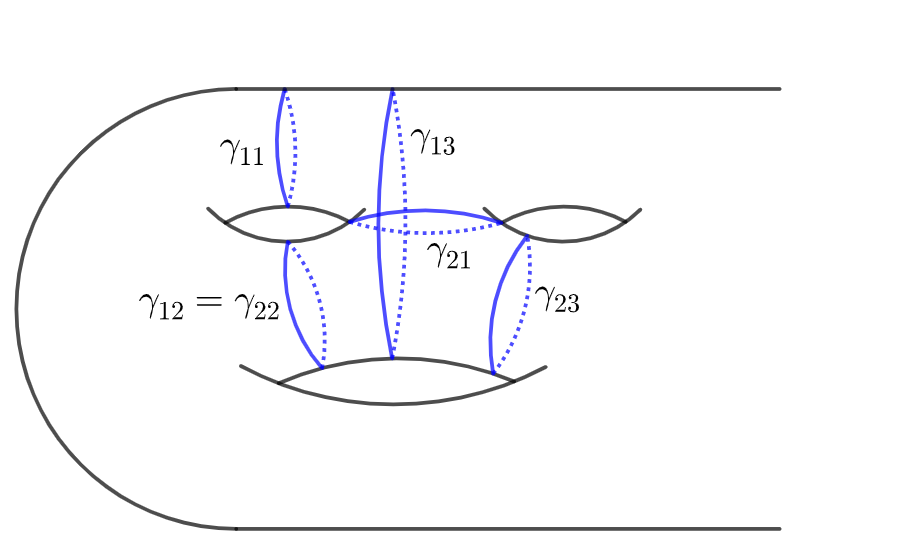}
      \caption{$\mathcal{C}(X,L)$.}
      \label{}
  \end{subfigure}
  \hfill
  \begin{subfigure}[b]{0.45\linewidth}
      \centering
      \includegraphics[width=\linewidth]{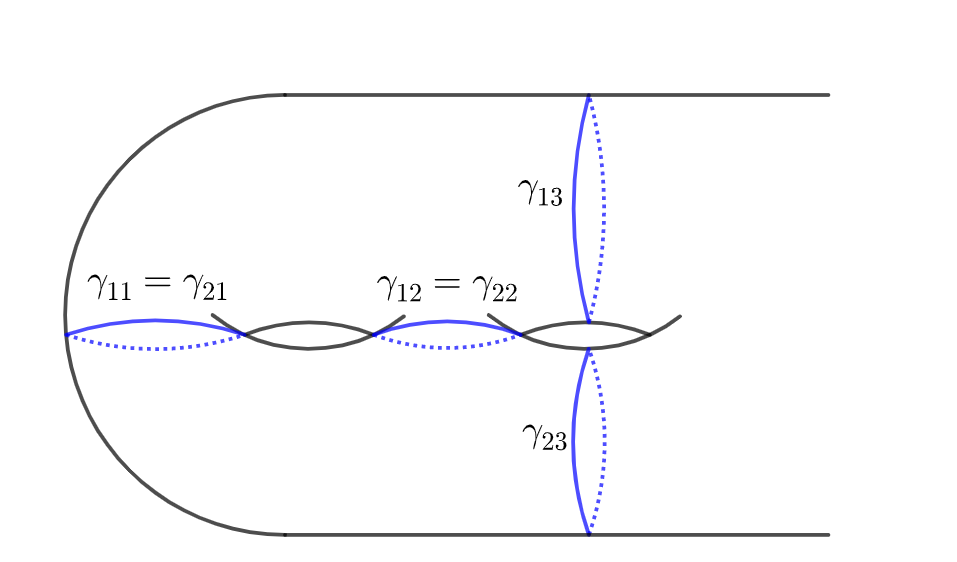}
      \caption{$\mathcal{D}(X,L)$.}
      \label{}
  \end{subfigure}
  \caption{Four types of $(\Gamma_1,\Gamma_2)$ on $X$}
  \label{fig:classifyABCD}
\end{figure}
\noindent Assume $\Gamma_1=(\gamma_{11},\gamma_{12},\gamma_{13})$ and $\Gamma_2=(\gamma_{21},\gamma_{22},\gamma_{23})$, as shown in Figure \ref{fig:classifyABCD}:
\begin{enumerate}[1.]
\item In Picture $(A)$, $\gamma_{1,i}=\gamma_{2,4-i}\ (i=1,2,3)$. Then we have $\Gamma_1\neq\Gamma_2$ and $P(\Gamma_1)=P(\Gamma_2)$. Hence $(\Gamma_1,\Gamma_2)\in\mathcal{A}(X,L)$;
\item In Picture $(B)$, $\overline{P(\Gamma_1)}\cap\overline{P(\Gamma_2)}=\emptyset$. Hence $(\Gamma_1,\Gamma_2)\in\mathcal{B}(X,L)$ and $$X\setminus(\Gamma_1\cup\Gamma_2)\cong S_{0,3}\cup S_{0,3}\cup S_{g-4,n+6};$$
\item In Picture $(C)$, ${P(\Gamma_1)}\cap{P(\Gamma_2)}\neq\emptyset$. Hence $(\Gamma_1,\Gamma_2)\in\mathcal{C}(X,L)$;
\item In Picture $(D)$,
$P(\Gamma_1)\cap P(\Gamma_2)=\emptyset$ and $\gamma_{1i}=\gamma_{2i} \ (i=1,2)$.  Hence $(\Gamma_1,\Gamma_2)\in \MD(X,L).$
\end{enumerate}
Denote by
\begin{align*}
    &A(X,L)=\#\mathcal{A}(X,L),\ B(X,L)=\#\mathcal{B}(X,L),\\
    &C(X,L)=\#\mathcal{C}(X,L),\ D(X,L)=\#\mathcal{D}(X,L).
\end{align*}
It is clear that 
\begin{equation}\label{separate N into ABCD}
\begin{aligned}
\E\left[\big(\Nnumber(X,L)\big)^2\right]&=\E\left[A(X,L)\right]+\E\left[B(X,L)\right]\\&+\E\left[C(X,L)\right]+\E\left[D(X,L)\right].
\end{aligned}
\end{equation}
\begin{rem*}
One important part in the proof of the upper bound of $\lns(X)$ for a random hyperbolic surface $X$ is to demonstrate that
 $$\E\left[C(X,L)\right]=o\left(E\left[\Nnumber(X,L)\right]^2\right).$$
 The condition $\ell(\gamma_i)\geq 10\log L\ (i=1,2,3)$ for elements in $\Nnumber(X,L)$ is a technical condition to make the computation easier. It still works without this condition, but with more complicated procedures.
\end{rem*}
Since any $\Gamma\in\mathcal{N}^{(g-2,3)}_{(0,3),\star}(X,L)$ is an ordered simple closed multi-curve, it follows that for any pair of pants $P$ in $X$, there exist at most $3!=6$ different $\Gamma'$s such that $P=P(\Gamma)$ and 
\begin{align}\label{compare AN}
\E\left[A(X,L)\right]\leq 6\cdot \E\left[\Nnumber(X,L)\right].
\end{align}

Now we split the proof of Theorem \ref{thm upper bound of ns-sys} into the following several subsections. The estimate for $\E\left[C(X,L)\right]$ is the hard part. And the estimates for $\E\left[A(X,L)\right]$,\\ $\E\left[B(X,L)\right]$ and 
$\E\left[D(X,L)\right]$ are relatively easier. 

\subsection{Estimates of $\E\left[\Nnumber(X,L)\right]$}

Recall that by Proposition \ref{thm E[N_f8]=} we have 
\be
\E\left[N_{\feight}(X,L_g) \right]\sim \frac{L_ge^{L_g}}{8\pi^2g}
\ene 
where $L_g=\log g-\log\log g+\omega(g)$ and $\omega(g)=o(\log\log g)$. We will see that\\ $\E\left[ \Nnumber(X,L_g) \right]$ is of the same growth rate.

\begin{proposition}\label{p-e-1}
Assume $L>1$ and $L=O\left(\log g\right)$, then
$$\E\left[N_{(0,3),\star}^{(g-2,3)}(X,L)\right]=\frac{1}{2\pi^2 g}Le^L\left(1+O\left(\frac{\log L}{L}\right)\right)$$  
where the implied constant is independent of $L$ and $g$.
\end{proposition}
\begin{proof}
For any $L>1$, let $D_L\subset\mathbb{R}_{\geq 0}^3$ be a domain defined by
$$D_L:=\left\{(x_1,x_2,x_3)\in\mathbb{R}^3_{\geq 0};\ x_1\leq L,\ x_2+x_3\leq L,\ x_1,\ x_2,\ x_3\geq 10\log L\right\}.$$
Assume $\phi_{L}:\mathbb{R}^3_+\to\mathbb{R}_{\geq 0}$ is the characteristic function of $D_L$, i.e.
$$\phi_L(u)=\begin{cases}0 & u\notin D_L,\\ 1 & u\in D_L.\end{cases}$$
Assume $\Gamma=(\gamma_1,\gamma_2,\gamma_3)$ is a pair of \emph{ordered} simple closed curves in $X$ such that
$$X\setminus\bigcup_{i=1}^3\gamma_i\cong S_{0,3}\cup S_{g-2,3}.$$
Applying Mirzakhani's integration formula Theorem \ref{thm Mirz int formula} to the ordered simple closed multi-curve $\Gamma$ and function $\phi_L$, together with \eqref{eq V_g-2,3(x,y,z)}, we have
\begin{align}\label{e-exp}
&\ \ \ \  \E\left[N_{(0,3),\star}^{(g-2,3)}(X,L)\right]\nonumber\\
&=\frac{1}{V_g}\int_{D_L}V_{0,3}(x_1,x_2,x_3)V_{g-2,3}(x_1,x_2,x_3)x_1x_2x_3\ dx_1dx_2dx_3 \\
&=\int_{D_L}\prod\limits_{i=1}^3\sinh\frac{x_i}{2}dx_1dx_2dx_3\cdot\left(1+O\left(\frac{L^2}{g}\right)\right).\nonumber
\end{align}
Direct computation implies that
\begin{align}\label{e-int-1}
\int_{10\log L}^L\sinh\frac{x_1}{2}dx_1=2\left(\cosh\frac{L}{2}-\cosh (5\log L)\right)=e^{\frac{L}{2}}+O\left(L^5\right)
\end{align}
and
\begin{align}\label{e-int-2}
& \int_{\mbox{\tiny$\begin{array}{c}x_2+x_3\leq L\\x_2,x_3\geq 10\log L\end{array}$}}\sinh \frac{x_2}{2}\sinh \frac{x_3}{2}dx_2dx_3 \nonumber\\
=&\int_{10\log L}^{L-10\log L}\sinh\frac{x_3}{2}\int_{10\log L}^{L-x_3}\sinh\frac{x_2}{2}dx_2dx_3\nonumber\\
=&2\int_{10\log L}^{L-10\log L}\sinh\frac{x_3}{2}\left(\cosh\frac{L-x_3}{2}-\cosh (5\log L)\right)dx_3\\
=&2\int_{10\log L}^{L-10\log L}\sinh\frac{x_3}{2}\cosh\frac{L-x_3}{2}dx_3+O\left( e^{\frac{L}{2}}\right)\nonumber\\
=&\frac{1}{2}Le^{\frac{L}{2}}+O\left(e^{\frac{L}{2}}\log L\right)\nonumber,
\end{align}
where the implied constants are independent of $L$. From \eqref{e-int-1} and \eqref{e-int-2} we have
\begin{equation}\label{e-int-3}
\begin{aligned}
& \int_{D_L}\sinh\frac{x_1}{2}\sinh\frac{x_2}{2}\sinh\frac{x_3}{2}dx_1dx_2dx_3\\
=&\int_{10\log L}^L\sinh\frac{x_1}{2}dx_1\times\int_{\mbox{\tiny$\begin{array}{c}x_2+x_3\leq L\\x_2,x_3\geq 10\log L\end{array}$}}\sinh \frac{x_2}{2}\sinh \frac{x_3}{2}dx_2dx_3\\
=&\left(e^{\frac{L}{2}}+O\left(L^5\right)\right)\times\left(\frac{1}{2}Le^{\frac{L}{2}}+O\left(e^{\frac{L}{2}}\log L\right)\right)\\
=&\frac{1}{2}Le^L+O\left(e^L\log L\right).
\end{aligned}
\end{equation}
Combining \eqref{e-exp}, \eqref{e-int-3} and the assumption that $L=O\left(\log g\right)$, we obtain
\begin{align*}
&\ \ \ \ \E\left[N_{(0,3),\star}^{(g-2,3)}(X,L)\right]\\
&=\frac{1}{\pi^2 g}\cdot\left(\frac{1}{2}Le^L+O\left(e^L\log L\right)\right)\cdot\left(1+O\left(\frac{L^2}{g}\right)\right)\\
&=\frac{1}{2\pi^2 g}Le^L\left(1+O\left(\frac{\log L}{L}\right)\right)
\end{align*}as desired.
\end{proof}

\subsection{Estimates of $\E\left[B(X,L)\right]$}

For $B(X,L)$, we will show that 
\[\E\left[B(X,L_g)\right]\sim \E\left[\Nnumber(X,L_g)\right]^2, \quad \textit{as $g\to \infty$},\]
where $L_g=\log g-\log \log g+\omega(g)$ and $\omega(g)=o(\log\log g)$. More precisely,

\begin{proposition}\label{estimation B(X,L)}
Assume $L>1$ and $L=O(\log g)$, then
\begin{align}
\mathbb{E}_{\textnormal{WP}}^g\left[B(X,L)\right]=\frac{1}{4\pi^4 g^2}L^2e^{2L}\left(1+O\left(\frac{\log L}{L}\right)\right),
\end{align}
where the implied constant is independent of $g$ and $L$.
\end{proposition}
\begin{proof}
Assume $(\Gamma_1,\Gamma_2)\in\mathcal{B}(X,L)$ and $\Gamma_1=(\alpha_1,\alpha_2,\alpha_3)$, $\Gamma_2=(\beta_1,\beta_2,\beta_3)$. From the definition of $\mathcal{B}(X,L)$, it is not hard to check that 
$$X\setminus\left(\Gamma_1\cup\Gamma_2\right)=P(\Gamma_1)\cup P(\Gamma_2)\cup Y,\text{ where } Y\cong S_{g-4,6}.$$
Define a function $\phi_{L,2}:\mathbb{R}^3_+\times\mathbb{R}^3_+\to\mathbb{R}_{\geq 0}$ as follows:
$$\phi_{L,2}(u,v):=\phi_L(u)\cdot \phi_L(v)$$
where $\phi_L$ is defined in the proof of Proposition \ref{p-e-1}. Assume $(\gamma_1,\gamma_2, \cdots, \gamma_6)$ is an ordered simple closed multi-curve in $S_g$ such that $$X\setminus\bigcup\limits_{i=1}^6\gamma_i\cong S_{0,3}^1\cup S_{0,3}^2\cup S_{g-4,6}$$
where the boundary of $S_{0,3}^1$ consists of $\gamma_i\ (1\leq i\leq 3)$ and the boundary of $S_{0,3}^2$ consists of $\gamma_j\ (4\leq j\leq 6)$. By Part $(2)$ of Theorem \ref{thm Vgn/Vgn+1} we have \begin{equation}\label{e-vol-2}
    \frac{V_{g-4,6}}{V_g}=\frac{1}{64\pi^4g^2}\left(1+O\left(\frac{1}{g}\right)\right).
\end{equation}
By Theorem \ref{thm Vgn(x) small x} we have
\begin{align}\label{e-vol-q}
V_{g-4,6}(x_1,...,x_6)=V_{g-4,6}\cdot\prod\limits_{i=1}^6\frac{2\sinh\frac{x_i}{2}}{x_i}\cdot\left(1+O\left(\frac{L^2}{g}\right)\right)
\end{align}
for $(x_1,\cdots,x_6)\in \textbf{supp}(\phi_{L,2})$,
where the implied constant is independent of $L$ and $g$. Applying Mirzakhani's integration formula Theorem \ref{thm Mirz int formula}  to ordered simple closed multi-curve $(\gamma_1,\gamma_2, \cdots, \gamma_6)$ and function $\phi_{L,2}$, together with \eqref{e-int-3}, \eqref{e-vol-2} and \eqref{e-vol-q}, we have 
\begin{align}
& \mathbb{E}_{\textnormal{WP}}^g\left[B(X,L)\right]\nonumber\\
=&\frac{1}{V_g}\int_{D_L\times D_L}V_{g-4,6}(x_1,...,x_6)\prod\limits_{i=1}^6x_idx_1...dx_6\nonumber\\
=&\frac{V_{g-4,6}}{V_g}\left(\int_{D_L}8\prod_{i=1}^3\sinh\frac{x_i}{2}dx_1dx_2dx_3\right)^2\cdot\left(1+O\left(\frac{L^2}{g}\right)\right)\\
=&\frac{1}{64\pi^4 g^2}\left(1+O\left(\frac{1}{g}\right)\right)\cdot\left(4Le^L+O\left(e^L\log L\right)\right)^2\cdot\left(1+O\left(\frac{L^2}{g}\right)\right)\nonumber\\
=&\frac{1}{4\pi^4 g^2}L^2e^{2L}\left(1+O\left(\frac{\log L}{L}\right)\right)\nonumber
\end{align}
where the implied constant is independent of $L$ and $g$.
\end{proof}

\subsection{Estimates of $\E\left[C(X,L)\right]$}

For $(\Gamma_1,\Gamma_2)\in \mathcal{C}(X,L)$, $\Gamma_1\cup\Gamma_2$ is not the union of disjoint simple closed geodesics but with intersections. This is the hard case. We will need the following construction as in \cite{MP19, WX22-GAFA, NWX23} to deform $\Gamma_1\cup\Gamma_2$.

\begin{con*} Fix a closed hyperbolic surface $X\in\mathcal{M}_g$ and let $X_1,X_2$ be two distinct connected, precompact subsurfaces of $X$ with geodesic boundaries, such that $X_1\cap X_2\neq\emptyset$ and neither of them contains the other. Then the union $X_1\cup X_2$ is a subsurface whose boundary consists of only piecewise geodesics. We can construct from it a new subsurface, with geodesic boundary, by deforming each of its boundary components $\xi\subset\partial\left(X_1\cup X_2\right)$ as follows:
\begin{enumerate}[1.]
\item if $\xi$ is homotopically nontrivial, we deform $X_1\cup X_2$ by shrinking $\xi$ to the unique simple closed geodesic homotopic to it;
\item if $\xi$ is homotopically trivial, we fill into $X_1\cup X_2$ the disc bounded by $\xi$.
\end{enumerate}
Denote by $S(X_1,X_2)$ the subsurface with geodesic boundary constructed from $X_1$ and $X_2$ as above. For any $(\Gamma_1,\Gamma_2)\in\mathcal{C}(X,L)$, denote by $S(\Gamma_1,\Gamma_2)=S\left(P(\Gamma_1),P(\Gamma_2)\right)$ the subsurface of $X$ constructed from $P(\Gamma_1)$ and $P(\Gamma_2)$. 
\end{con*}

By the  construction, it is clear that 
\begin{equation}\label{ell partial S(X_1,X_2)}
\ell(\partial S(X_1,X_2))\leq \ell(\partial X_1)+\ell(\partial X_2),\end{equation}
and by isoperimetric inequality (see e.g. \cite[Section 8.1]{Buser10} or \cite{WX18}) we have\begin{equation}\label{area S(X_1,X_2)}
\area(S(X_1,X_2))\leq \area(X_1)+\area(X_2)+\ell(\partial X_1)+\ell(\partial X_2).
\end{equation}

\begin{figure}[htbp]
  \centering
  \begin{subfigure}[b]{0.45\linewidth}
      \centering
      \includegraphics[width=\linewidth]{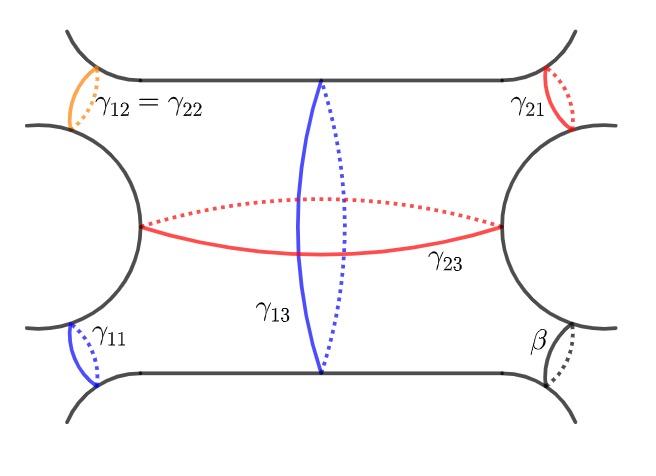}
      \caption{$S(\Gamma_1,\Gamma_2)\cong S_{0,4}$.}
      \label{}
  \end{subfigure}
  \hfill
  \begin{subfigure}[b]{0.45\linewidth}
      \centering
      \includegraphics[width=\linewidth]{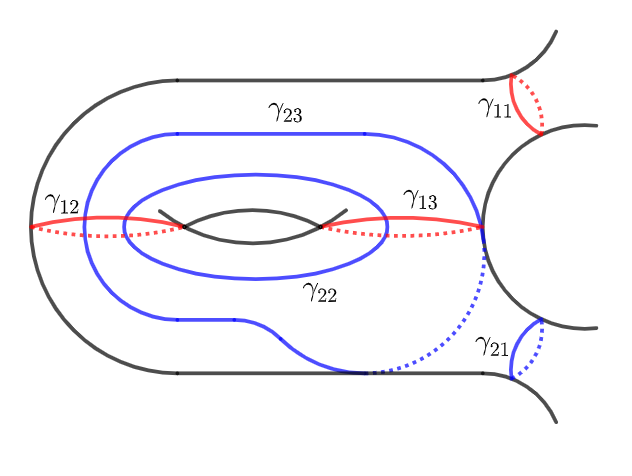}
      \caption{$S(\Gamma_1,\Gamma_2)\cong S_{1,2}$.}
      \label{}
  \end{subfigure}
  \\
  \begin{subfigure}[b]{0.45\linewidth}
      \centering
      \includegraphics[width=\linewidth]{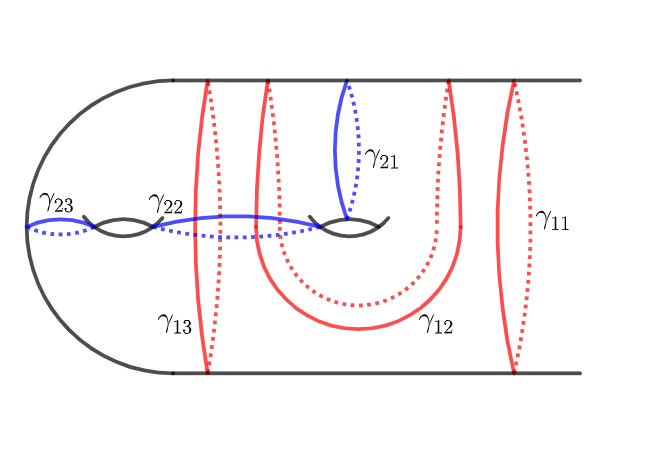}
      \caption{$|\chi(S(\Gamma_1,\Gamma_2))|\geq 3$.}
      \label{}
  \end{subfigure}
  \caption{Three types of $S(\Gamma_1,\Gamma_2)$}
  \label{fig:classifyC}
\end{figure}
Recall that by the definition of $\mathcal{N}_{(0,3),\star}^{(g-2,3)}(X,L)$, we know that for any $\Gamma \in \mathcal{N}_{(0,3),\star}^{(g-2,3)}(X,L)$, $$X\setminus \Gamma \cong  S_{0,3} \cup S_{g-2,3}.$$
Thus we may divide $\mathcal{C}(X,L)$ into following pairwisely disjoint three parts:
\be\label{e-C-3-p}
\mathcal{C}(X,L)=\mathcal{C}_{0,4}(X,L)\cup\mathcal{C}_{1,2}(X,L)\cup\mathcal{C}_{\geq 3}(X,L)
\ene
where 
\begin{align*}
&\mathcal{C}_{0,4}(X,L)=\left\{(\Gamma_1,\Gamma_2)\in\mathcal{C}(X,L);\ S(\Gamma_1,\Gamma_2)\cong S_{0,4}\right\},\\
&\mathcal{C}_{1,2}(X,L)=\left\{(\Gamma_1,\Gamma_2)\in\mathcal{C}(X,L);\ S(\Gamma_1,\Gamma_2)\cong S_{1,2}\right\},\\
&\mathcal{C}_{\geq 3}(X,L)=\{(\Gamma_1,\Gamma_2)\in\mathcal{C}(X,L);\ |\chi(S(\Gamma_1,\Gamma_2))|\geq 3\}.
\end{align*}
Denote by \begin{align*}
    C_{0,4}(X,L)=\#\mathcal{C}_{0,4}(X,L),\ C_{1,2}(X,L)=\#\mathcal{C}_{1,2}(X,L),\  C_{\geq 3}(X,L)=\#\mathcal{C}_{\geq 3}(X,L).
\end{align*}

Assume $\Gamma_1=(\gamma_{11},\gamma_{12},\gamma_{13})$ and $\Gamma_2=(\gamma_{21},\gamma_{22},\gamma_{23})$. As in Figure \ref{fig:classifyC}:
\begin{enumerate}[1.]
\item In Picture $(A)$, the simple closed geodesic $\gamma_{12}$ coincides with $\gamma_{22}$. We have $S(\Gamma_1,\Gamma_2)\cong S_{0,4}$ of geodesic boundaries $\gamma_{11},\ \gamma_{12}=\gamma_{22},\ \gamma_{21}$ and $\beta$. Hence $(\Gamma_1,\Gamma_2)\in\mathcal{C}_{0,4}(X,L)$;
\item In Picture $(B)$, we have $S(\Gamma_1,\Gamma_2)\cong S_{1,2}$ with boundary geodesics $\gamma_{11}$ and $\gamma_{21}$. Hence $(\Gamma_1,\Gamma_2)\in\mathcal{C}_{1,2}(X,L)$;
\item In picture $(C)$, we have $S(\Gamma_1,\Gamma_2)\cong S_{0,5}$ with boundary geodesics $\gamma_{11},\gamma_{21},\gamma_{23}$ where $\gamma_{21}$ and $\gamma_{23}$ appear twice in the boundary of $S(\Gamma_1,\Gamma_2)$. Hence $(\Gamma_1,\Gamma_2)\in\mathcal{C}_{\geq 3}(X,L)$.
\end{enumerate}

For $L>0$ and $X\in\mathcal{M}_g$, define 
$$\textnormal{Sub}_{L}(X)\overset{\textnormal{def}}{=}\left\{\begin{matrix}Y\subset X\text{ is a connected subsurface of geodesic boundary}\\ \text{such that }\ell(\partial Y)\leq 2L\text{ and }\textnormal{Area}(Y)\leq 4L+4\pi\end{matrix}\right\}.$$

\begin{lemma}\label{l-cor}
For any $(\Gamma_1,\Gamma_2)\in\mathcal{C}(X,L)$, there exists a triple $(\alpha_1,\alpha_2,Y)$ such that
\begin{enumerate}
\item $Y=S(\Gamma_1,\Gamma_2)\in\textnormal{Sub}_L(X)$;
\item $\alpha_i$ is a  figure-eight closed geodesic contained in $P(\Gamma_i)$ for $i=1,2$;
\item $(\alpha_1,\alpha_2)$ is a filling 2-tuple in $Y$ and $\ell(\alpha_1),\ell(\alpha_2)\leq L+2\log 6$.
\end{enumerate}
\end{lemma}
\begin{proof}
For Part (1), it follows from \eqref{ell partial S(X_1,X_2)} and \eqref{area S(X_1,X_2)}  that for any $(\Gamma_1,\Gamma_2)\in \mathcal{C}(X,L)$, $$S(\Gamma_1,\Gamma_2)\in \textnormal{Sub}_{L}(X).$$

Part (2) is clear.

For Part (3), we first assume $\Gamma_i=(\gamma_{i1},\gamma_{i2},\gamma_{i3})\ (i=1,2)$. For $i=1,2$, let $\alpha_i$ be the figure-eight closed geodesic contained in $P(\Gamma_i)$ winding around $\gamma_{i2}$ and $\gamma_{i3}$. Then from the assumption that $(\Gamma_1,\Gamma_2)\in\mathcal{C}(X,L)$ and inequality \eqref{length-f-8}, we have
$$\ell(\alpha_1),\ell(\alpha_2)\leq L+2\log 6.$$ \noindent Now we show that $(\alpha_1,\alpha_2)$ is a filling 2-tuple in $S(\Gamma_1,\Gamma_2)$. Suppose not, then there exists a simple closed geodesic $\beta$ in $Y$ such that $\beta\cap (\alpha_1\cup\alpha_2)=\emptyset.$ Since $\alpha_i$ fills $P(\Gamma_i)\ (i=1,2)$, it follows that  $\beta\cap (P(\Gamma_1)\cup P(\Gamma_2))=\emptyset$.  Then by the construction of $S(\Gamma_1,\Gamma_2)$ we have $\beta\cap S(\Gamma_1,\Gamma_2)=\emptyset$, which is a contradiction. 

The proof is complete.
\end{proof}
Similar to \cite{WX22-GAFA}, we set the following assumption.

\textbf{Assumption $(\star)$.} Let $Y_0\in\textnormal{Sub}_{L}(X)$ satisfy
\begin{enumerate}
\item $Y_0$ is homeomorphic to $S_{g_0,k}$ for some $g_0\geq 0$ and $k>0$ with $$m=|\chi(Y_0)|=2g_0-2+k\geq 1;$$
\item the boundary $\partial Y_0$ is a simple closed multi-geodesic in $X$ consisting of $k$ simple closed geodesics which has $n_0$ pairs of simple closed geodesics for some $n_0\geq 0$ such that each pair corresponds to a single simple closed geodesic in $X$;
\item the interior of its complement $X\setminus S_{g_0,k}$ consists of $q$ components \\ $S_{g_1,n_1},...,S_{g_q,n_q}$ for some $q\geq 1$ where $\sum_{i=1}^q n_i=k-2n_0$.
\end{enumerate}

Our aim is to bound $\E\left[C(X,L)\right]$, from \eqref{e-C-3-p} it suffices to bound the three terms $\E\left[C_{\geq 3}(X,L)\right], \ \E\left[C_{0,4}(X,L)\right]$ and $\E\left[C_{1,2}(X,L)\right]$
separately.

\subsubsection{Bounds for $\E\left[C_{\geq 3}(X,L)\right]$}
We first bound $\E\left[C_{\geq 3}(X,L)\right]$ through using the method in \cite{WX22-GAFA}.
\begin{proposition}\label{prop C geq3}
Assume $L>1$ and $L=O (\log g)$, then for any fixed small $\epsilon>0$,
$$\E\left[C_{\geq 3}(X,L)\right]\prec \left(L^{67}e^{2L+\epsilon L}\frac{1}{g^3}+\frac{L^3e^{8L}}{g^{11}}\right).$$
\end{proposition}
\begin{proof}
For $(\Gamma_1,\Gamma_2)\in \mathcal{C}_{\geq 3}(X,L)$, by Lemma \ref{l-cor}, there exist a subsurface $Y\in\textnormal{Sub}_L(X)$ and a filling 2-tuple
 $(\alpha_1,\alpha_2)$   in $Y$  with total length $\leq 2L+4\log 6$, and $\alpha_i$ is a filling figure-eight closed geodesic in $P(\Gamma_i)$ for $i=1,2$. Since for any such triples $(\alpha_1,\alpha_2,Y)$, there are at most $36$ different pairs  $(\Gamma_1,\Gamma_2)\in\mathcal{C}_{\geq 3}(X,L)$ corresponding to it, 
 it follows that \begin{equation*}
C_{\geq 3}(X,L)\\
\leq \sum_{\substack{Y \in \operatorname{Sub}_L(\mathrm{X}) ; \\ 3 \leq|\chi(Y)|\leq \left[\frac{4 L+4\pi}{2 \pi}\right]}}36\cdot N_2^{\text{fill}}(Y,2L+4\log 6)
 \end{equation*}
where $N_2^{\text{fill}}(Y,2L+4\log 6)$ is defined in Subsection \ref{ss-count-2}. Therefore we have \begin{equation}\label{eq C geq3 expectation}
\begin{aligned}
     &\ \ \ \ \E\left[C_{\geq 3}(X,L)\right]\\
     &\leq \frac{1}{V_g}\int_{\M_g}\sum_{\substack{Y \in \operatorname{Sub}_L(\mathrm{X}) ; \\ 3 \leq|\chi(Y)|\leq \left[\frac{4 L+4\pi}{2 \pi}\right] }}36N_2^{\text{fill}}(Y,2L+4\log 6)\cdot 1_{[0,2L]}(\ell(\partial Y))dX.
 \end{aligned}
 \end{equation}
 Now we divide the summation above into following two parts: the first part consists of all subsurfaces $Y\in\textnormal{Sub}_L(X)$ such that $3\leq |\chi(Y)|\leq 10$; the second part consists of all subsurfaces $Y\in\textnormal{Sub}_L(X)$ such that $10< |\chi(Y)|\leq \left[\frac{4 L+4\pi}{2 \pi}\right]$.

 For the first part, assume that $Y\cong S_{g_0,k}\in\textnormal{Sub}_L(X)$ satisfies \textit{Assumption} $(\star)$ with an additional assumption that
\begin{align}\label{e-ass}
3\leq m=2g_0-2+k\leq 10.
\end{align}
From \cite[Proposition 34]{WX22-GAFA} and Theorem \ref{thm count fill k-tuple}, we have that for any fixed $0<\epsilon<\frac{1}{2}$,
\begin{equation}\label{e-in}\begin{aligned}
& \frac{1}{V_g}\int_{\mathcal{M}_g}\sum_{\substack{Y \in \operatorname{Sub}_L(\mathrm{X}) ; \\ Y\cong S_{g_0,k}}}36N_2^{\text{fill}}(Y,2L+4\log 6)\cdot 1_{[0,2L]}(\ell(\partial Y))dX\\
\prec&\frac{1}{V_g}\int_{\mathcal{M}_g}\sum_{\substack{Y \in \operatorname{Sub}_L(\mathrm{X}) ; \\ Y\cong S_{g_0,k}}}Le^{2L-\frac{1-\epsilon}{2}\ell(\partial Y)}\cdot 1_{[0,2L]}(\ell(\partial Y))dX \ (\textit{\rm{by Theorem \ref{thm count fill k-tuple}}})\\
=&Le^{\frac{3}{2}L}\times\frac{1}{V_g}\int_{\mathcal{M}_g}\sum_{\substack{Y \in \operatorname{Sub}_L(\mathrm{X}) ; \\ Y\cong S_{g_0,k}}}e^{\frac{1}{2}L-\frac{1-\epsilon}{2}\ell(\partial Y)}\cdot 1_{[0,2L]}(\ell(\partial Y))dX\\
\prec& Le^{\frac{3}{2}L}\times L^{66}e^{\frac{1}{2}L+\epsilon L}\frac{1}{g^m}  \quad (\textit{\rm{by \cite[Proposition 34]{WX22-GAFA}}})\\
=&L^{67}e^{2L+\epsilon L}\frac{1}{g^m}.
\end{aligned} \end{equation}
Since there are at most finite pairs $(g_0,k)$ satisfying the assumption \eqref{e-ass}, take sum over all possible subsurfaces $Y$ for inequality \eqref{e-in}, we have 
\begin{equation}\label{part-1}
\begin{aligned}
&\frac{1}{V_g}\int_{\mathcal{M}_g}\sum_{\substack{Y \in \operatorname{Sub}_L(\mathrm{X}) ; \\ 3 \leq|\chi(Y)| \leq 10}}36N_2^{\text{fill}}(Y,2L+4\log 6)\cdot 1_{[0,2L]}(\ell(\partial Y))dX\\
\prec& L^{67}e^{2L+\epsilon L}\frac{1}{g^3}.
\end{aligned}
\end{equation}

For the second part, firstly by \eqref{area S(X_1,X_2)} and the Gauss-Bonnet formula we know that $|\chi(Y)|\prec L$. Since $\ell(\partial Y)\leq 2L$, by  Theorem \ref{thm count ge^L upp} we have
\begin{align}\label{e-count-3}
N_2^{\text{fill}}(Y,2L+4\log 6)&\prec \left(\left|\chi(Y)\right|e^{2L}\right)^2
\prec L^2e^{\frac{9}{2}L-\frac{1}{4}\ell(\partial Y)}.
\end{align}
From \eqref{e-count-3} and \cite[Proposition 33]{WX22-GAFA}, we have
\begin{equation}\label{part-2}
\begin{aligned}
& \frac{1}{V_g}\int_{\mathcal{M}_g}\sum_{\substack{Y \in \operatorname{Sub}_L(\mathrm{X}) ; \\ 11 \leq|\chi(Y)| \leq \left[\frac{4 L+4\pi}{2 \pi}\right] }}36N_2^{\text{fill}}(Y,2L+8\log 2)\cdot 1_{[0,2L]}(\ell(\partial Y))dX\\
\prec&\frac{1}{V_g}\int_{\mathcal{M}_g}\sum_{\substack{Y \in \operatorname{Sub}_L(\mathrm{X}) ; \\ 11 \leq|\chi(Y)| \leq \left[\frac{4 L+4\pi}{2 \pi}\right] }}L^2e^{\frac{9}{2}L-\frac{1}{4}\ell(\partial Y)}\cdot 1_{[0,2L]}(\ell(\partial Y))dX\\
\prec& L^2e^{\frac{9}{2}L}\times \frac{1}{V_g}\int_{\mathcal{M}_g}\sum_{\substack{Y \in \operatorname{Sub}_L(\mathrm{X}) ; \\ 11 \leq|\chi(Y)| \leq \left[\frac{4 L+4\pi}{2 \pi}\right] }}e^{-\frac{1}{4}\ell(\partial Y)}\cdot 1_{[0,2L]}(\ell(\partial Y))dX\\
\prec& L^2e^{\frac{9}{2}L}\times Le^{\frac{7}{2}L}\frac{1}{g^{11}}\quad (\textit{\rm{by \cite[Proposition 33]{WX22-GAFA}}})\\
=&\frac{L^3e^{8L}}{g^{11}}.
\end{aligned}
\end{equation}
Then combining \eqref{eq C geq3 expectation}, \eqref{part-1} and \eqref{part-2}, we complete the proof.
\end{proof}

\subsubsection{Bounds for $\mathbb{E}_{\textnormal{WP}}^g\left[C_{1,2}(X,L)\right]$}
In this subsection, we give an estimate for $\E\left[C_{1,2}(X,L)\right]$, the method is different from the one in Proposition \ref{prop C geq3}.
\begin{proposition}\label{c12}
      For $L>1$ and large $g$,  
    $$
\E\left[C_{1,2}(X,L)\right]\prec \frac{e^{2L}}{g^2}.
    $$
\end{proposition}

The estimates for pairs $(\Gamma_1,\Gamma_2)$ with  $S(\Gamma_1,\Gamma_2)=Y$  in  Lemma \ref{l-cor} are not good enough. We need to accurately classify the relative position of $(\Gamma_1,\Gamma_2)$ in $Y$. For any $(\Gamma_1,\Gamma_2)\in\mathcal{C}_{1,2}(X,L)$, we have $S(\Gamma_1,\Gamma_2)\cong S_{1,2}$, hence $\partial S(\Gamma_1,\Gamma_2)=\{\gamma_1,\gamma_2\}$ consists of two simple closed geodesics. 
 Recall that both $\Gamma_1$ and $\Gamma_2$ are simple closed multi-geodesics satisfying $X\setminus \Gamma_i\cong  S_{0,3} \cup S_{g-2,3}$, so each of $\Gamma_i$ contains at least one of $\{\gamma_1, \gamma_2\}$. There are three different possible cases: 
 \begin{enumerate}
     \item One of $\Gamma_i$ contains both two curves $\{\gamma_1, \gamma_2\}$;
     \item $\Gamma_1$ contains exactly one of $\{\gamma_1,\gamma_2\}$ and $\Gamma_2$ contains the other;
     \item Exactly one of $\{\gamma_1,\gamma_2\}$ is contained in both  $\Gamma_1$ and $\Gamma_2$.
 \end{enumerate}
 For $i=1,2,3$, denote by 
 $$\mathcal{C}_{1,2}^i(X,L)=\{(\Gamma_1,\Gamma_2)\in \MC_{1,2}(X,L);\ \Gamma_1,\Gamma_2 \text{ satisfy the i-th condition above}\}$$
 and 
 $$C_{1,2}^i(X,L)=\#\mathcal{C}_{1,2}^i(X,L).$$
 We estimate $C_{1,2}^i(X,L)$ for $L>1$ and  $i=1,2,3$ in the following lemmas.
\begin{lemma}\label{lemma121}For compact hyperbolic surface $X\in \M_g$ and $L>1$, 
\begin{equation}\label{equation121}
    \begin{aligned}
C_{1,2}^1(X,L)\prec&\sum_{(\gamma_1,\gamma_2,\gamma_3)\in\mathcal{G}_{1,2}^1(X)} 1_{[0,L]^3}(\ell(\gamma_1),\ell(\gamma_2),\ell(\gamma_3)) \\
\cdot &\left(\frac{\ell(\gamma_1)}{\cR(\ell(\gamma_1),\ell(\gamma_2),L)}+\frac{\ell(\gamma_1)}{\cD(\ell(\gamma_1),2L,0)}\right),\\
\end{aligned}
\end{equation} 
where $\mathcal{G}_{1,2}^1(X)$ consists of all simple closed multi-geodesics $(\gamma_1,\gamma_2,\gamma_3)$ such that $\gamma_1,\ \gamma_2$ cut off a subsurface $Y_0\cong S_{1,2}$ from $X$ and
$Y_0\setminus \gamma_3\cong S_{0,3}\cup S_{1,1}$.
\end{lemma}
\begin{proof}
    For any $(\Gamma_1,\Gamma_2)\in\mathcal{C}_{1,2}^1(X,L)$, WLOG, one may assume that $Y={S(\Gamma_1,\Gamma_2)}$ and  $\Gamma_1$ contains $\partial Y=\{\gamma_1,\gamma_2\}.$ Denote  by $\gamma_3$ the remaining simple closed geodesic in $\Gamma_1$. Then $(\gamma_1,\gamma_2,\gamma_3)\in\mathcal{G}_{1,2}^1(X)$ and 
    \begin{equation}\label{condition s121}
\ell(\gamma_1),\ \ell(\gamma_2),\ \ell(\gamma_3)\in [0,L].
 \end{equation}
 Consider the map $$\pi: (\Gamma_1,\Gamma_2)\mapsto (\gamma_1,\gamma_2,\gamma_3).$$
 Then we have 
 \begin{align}\label{count-s1}
     C_{1,2}^1(X,L)\leq\sum\limits_{(\gamma_1,\gamma_2,\gamma_3)\in\mathcal{G}_{1,2}^1(X)}1_{[0,L]^3}(\ell(\gamma_1),\ell(\gamma_2),\ell(\gamma_3))\cdot\#\pi^{-1}(\gamma_1,\gamma_2,\gamma_3).
 \end{align}
  Now we estimate $\#\pi^{-1}(\gamma_1,\gamma_2,\gamma_3)$ for fixed $(\gamma_1,\gamma_2,\gamma_3)$. Since $\Gamma_1$ is fixed, it suffices to count the number of all $\Gamma_2$'s such that $$(\Gamma_1,\Gamma_2)\in \MC_{1,2}(X,L) \text{ and }P(\Gamma_2)\subset Y.$$ 
  Since $\Gamma_1, \Gamma_2\in \mathcal{N}_{(0,3),\star}^{(g-2,3)}(X,L)$ and $(\Gamma_1,\Gamma_2)\in \MC_{1,2}(X,L)$, it follows that $\Gamma_2$ must contain at least one of $\gamma_1$ and $\gamma_2$. 

 \underline{Case-1: $\Gamma_2$ contains $\gamma_1\cup\gamma_2$} (see Figure \ref{figure-s12-1-case1} for an illustration). 
For this case, the remaining simple closed geodesic $\tilde{\gamma}$ in $\Gamma_2$ satisfies $\ell_{\tilde{\alpha}}(X)\leq L$ and $\tilde{\alpha},\ \gamma_1,\ \gamma_2$ bound a pair of pants in $Y$ as in Figure \ref{figure-s12-1-case1}. 
 \begin{figure}
    \centering
    \includegraphics[width=2.5 in]{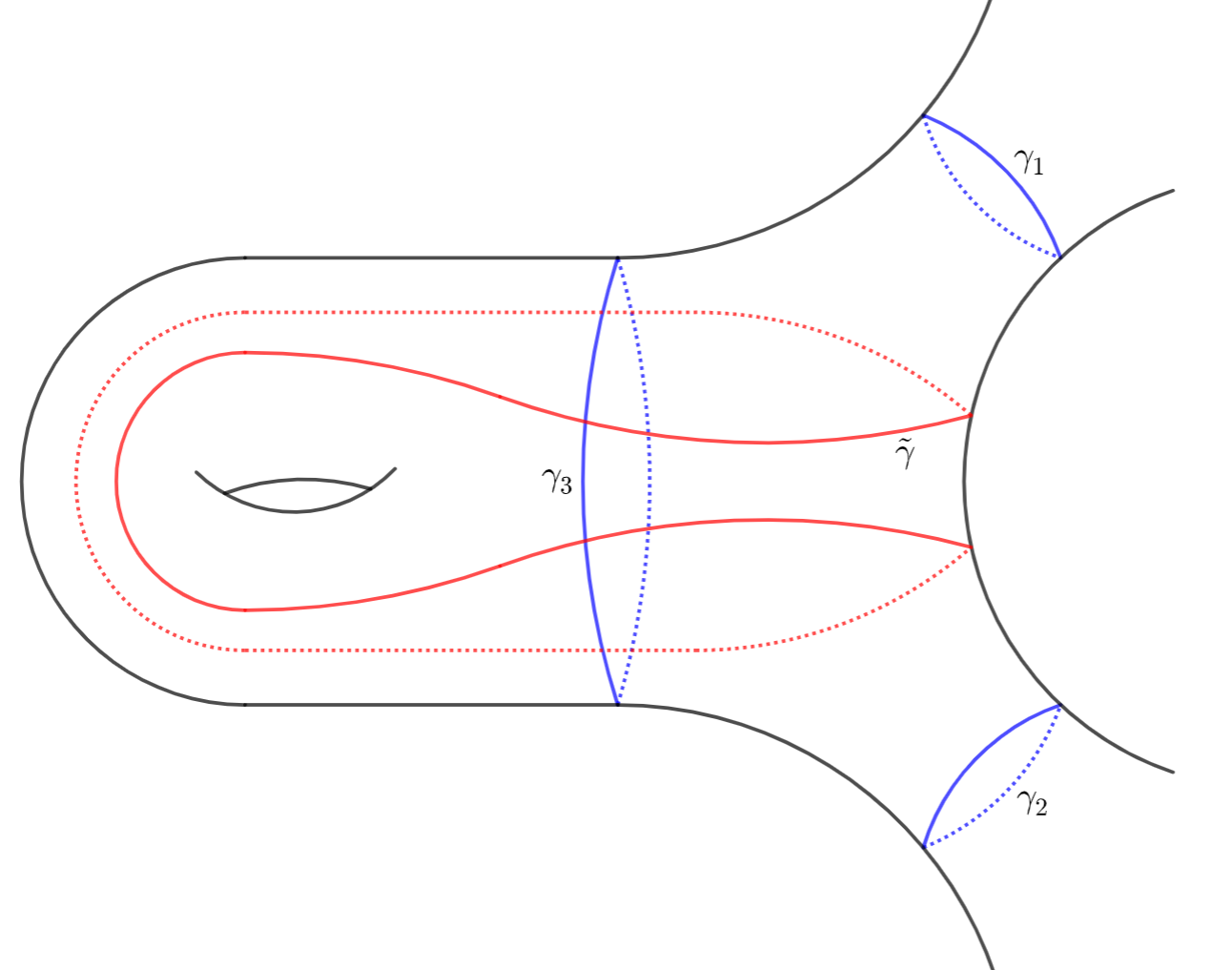}
    \caption{Case-1: 
 $\Gamma_1=(\gamma_1,\gamma_2,\gamma_3)$ and $\Gamma_2=(\gamma_1,\gamma_2,\tilde{\gamma})$ in $Y=S(\Gamma_1,\Gamma_2)\cong S_{1,2}$ with $\partial Y=\{\gamma_1,\gamma_2\}$ }
    \label{figure-s12-1-case1}
\end{figure}
Then it follows from \eqref{eq counting in S12 single}  that the number of such $\tilde\gamma$'s is at most $$
\frac{\ell(\gamma_1)}{\cR(\ell(\gamma_1),\ell(\gamma_2),L)}.$$

\underline{Case-2: $\Gamma_2$ contains only one of  $\gamma_1,\gamma_2$} (see Figure \ref{figure-s12-1-case2} for an illustration). 
 \begin{figure}
    \centering
    \includegraphics[width=2.5 in]{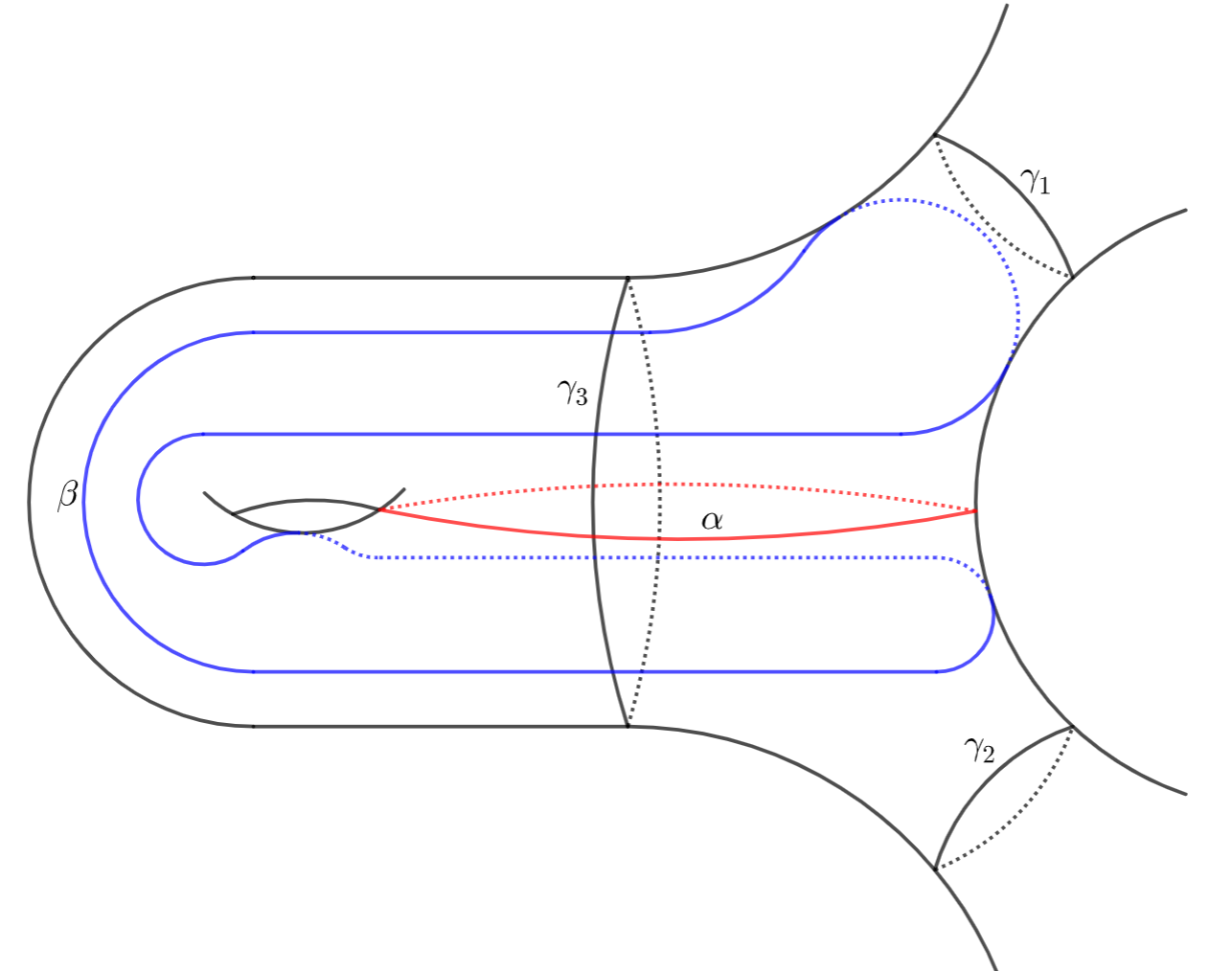}
    \caption{Case-2: 
 $\Gamma_1=(\gamma_1,\gamma_2,\gamma_3)$, $\Gamma_2=(\gamma_1,\alpha,\beta)$ in $Y=S(\Gamma_1,\Gamma_2)\cong S_{1,2}$ with $\partial Y=\{\gamma_1,\gamma_2\}$ }
    \label{figure-s12-1-case2}
\end{figure}
WLOG, one may assume that $\Gamma_2$ contains only $\gamma_1$. Consider the remaining two simple closed geodesics $\alpha,\beta$ in $\Gamma_2$, then we have $\ell(\alpha)+\ell(\beta)\leq 2L$ and $\alpha,\ \beta,\ \gamma_1$ bound a pair of pants in $Y$ as in  Figure \ref{figure-s12-1-case2}. Then it follows from \eqref{eq counting in S12 pairs}  that the number of such pairs of $(\alpha,\beta)$'s  is at most $$
\frac{\ell(\gamma_1)}{\cD(\ell(\gamma_1),2L,0)}.$$
 Combining these two cases, we have  
\[\#\pi^{-1}(\gamma_1,\gamma_2,\gamma_3)\leq 200\cdot \left(\frac{\ell(\gamma_1)}{\cR(\ell(\gamma_1),\ell(\gamma_2),L)}+\frac{\ell(\gamma_1)}{\cD(\ell(\gamma_1),2L,0)}\right).\]
Together with \eqref{count-s1}, one may complete the proof.
\end{proof}
\begin{rem*}
     The coefficient $200$ in the proof of Lemma \ref{lemma121} comes from the symmetry of the three boundary components of a pair of pants, the symmetry of $\Gamma_1$ and $\Gamma_2$,  and is not essential. What we need is a universal positive constant.
\end{rem*}
\begin{lemma}\label{lemma122}
For compact hyperbolic surface $X\in\M_g$ and $L>1$, 
    \begin{equation}\label{equation122}
    \begin{aligned}
C_{1,2}^2(X,L)
\prec&\sum_{(\gamma_1,\gamma_2,\alpha_1,\beta_1)\in\mathcal{G}_{1,2}^2(X)} 1_{[0,L]^4}(\ell(\gamma_1),\ell(\gamma_2),\ell(\alpha_1),\ell(\beta_1)) \\
\cdot&\frac{\ell(\gamma_2)}{\cD(\ell(\gamma_2),2L,0)},
\end{aligned}
\end{equation}
where $\mathcal{G}_{1,2}^2(X)$ consists of all simple closed multi-geodesics $(\gamma_1,\gamma_2,\alpha_1,\beta_1)$ such that $\gamma_1,\ \gamma_2$ cut off a subsurface $Y_0\cong S_{1,2}$ from $X$ and $Y_0\setminus\{\alpha_1,\beta_1\}\cong S_{0,3}\cup S_{0,3}$ (see Figure \ref{figure-s12-2} for an illustration). 

\end{lemma}

\begin{proof}
 For any $(\Gamma_1,\Gamma_2)\in\mathcal{C}_{1,2}^2(X,L)$,
WLOG, one may assume that $Y=S(\Gamma_1,\Gamma_2)$, $\partial Y=\{\gamma_1,\gamma_2\}$, $\Gamma_1$ only contains $\gamma_1$ and $\Gamma_2$ only contains $\gamma_2$. The remaining two simple closed geodesics $\alpha_1,\beta_1$ contained in $\Gamma_1$ separate $Y$ into two copies of $S_{0,3}$. Then $(\gamma_1,\gamma_2,\alpha_1,\beta_1)\in\mathcal{G}_{1,2}^2(X)$  and \begin{equation}\label{condition s122}
\ell(\gamma_1),\ell(\gamma_2),\ell(\alpha_1),\ell(\beta_1)\in [0, L].
\end{equation} 
Consider the map $$
 \pi:(\Gamma_1,\Gamma_2)\mapsto (\gamma_1,\gamma_2,\alpha_1,\beta_1).
 $$ 
Then we have 
 \begin{equation}\label{count-s2}
\begin{aligned}
C_{1,2}^2(X,L)&\leq\sum\limits_{(\gamma_1,\gamma_2,\alpha_1,\beta_1)\in\mathcal{G}_{1,2}^2(X)}1_{[0,L]^4}(\ell(\gamma_1),\ell(\gamma_2),\ell(\alpha_1),\ell(\beta_1))\\
     &\cdot\#\pi^{-1}(\gamma_1,\gamma_2,\alpha_1,\beta_1). 
 \end{aligned}
 \end{equation}
For any fixed $(\gamma_1,\gamma_2,\alpha_1,\beta_1)$ and $(\Gamma_1,\Gamma_2)\in\mathcal{C}_{1,2}^2(X,L)$ with 
 $$\pi(\Gamma_1,\Gamma_2)=(\gamma_1,\gamma_2,\alpha_1,\beta_1),$$
 we have $\Gamma_1$ is fixed, $\gamma_1\notin \Gamma_2, \gamma_2\in \Gamma_2$ and $P(\Gamma_2)\subset Y$. 
 \begin{figure}[b]
    \centering
    \includegraphics[width=2.5 in]{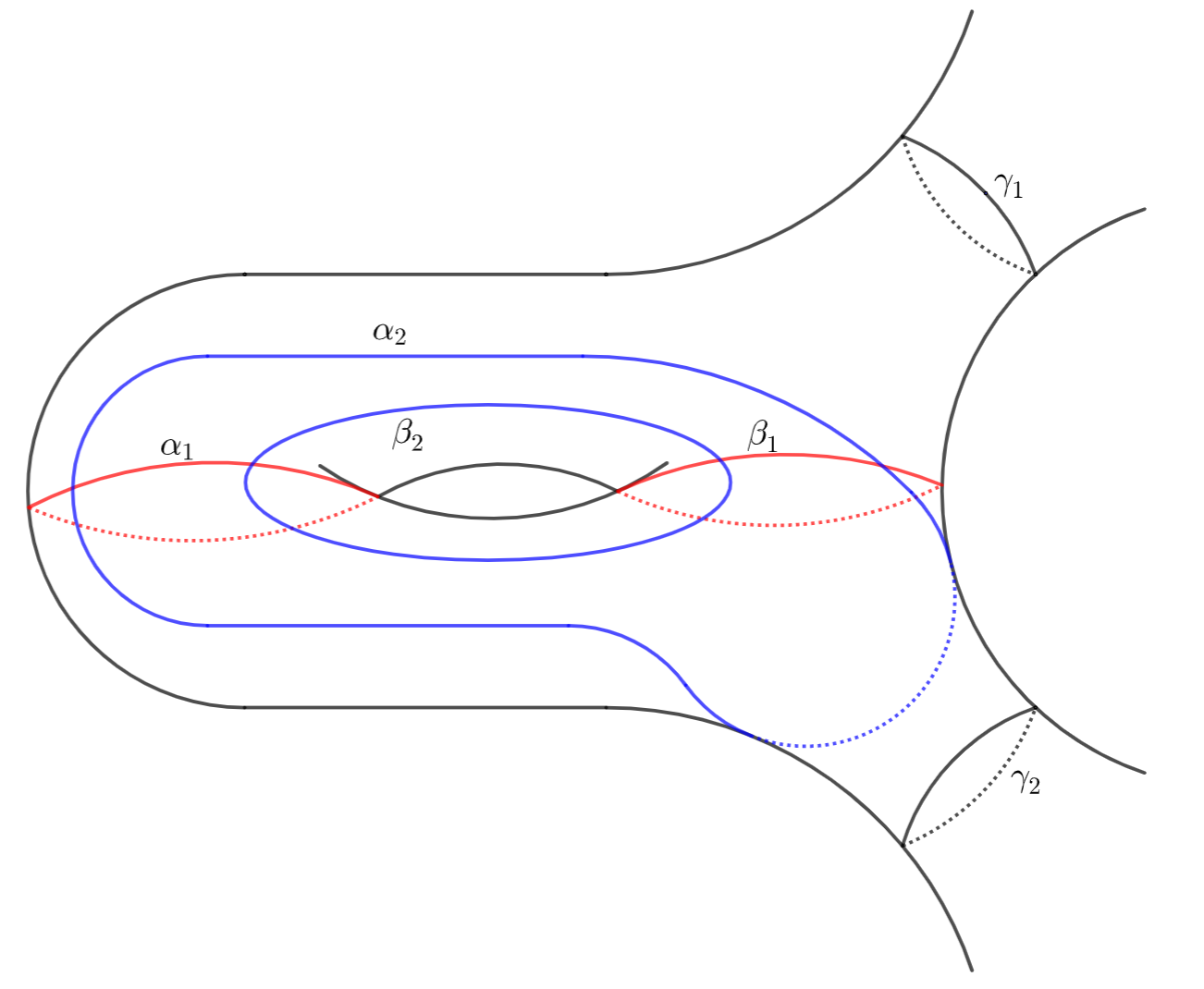}
    \caption{
    $\Gamma_1=\{\gamma_1,\alpha_1,\beta_1\}$ and $\Gamma_2=\{\gamma_2,\alpha_2,\beta_2\}$ in $Y=S(\Gamma_1,\Gamma_2)\cong S_{1,2}$ with $\partial Y=\{\gamma_1,\gamma_2\}$
    }
    \label{figure-s12-2}
\end{figure}
 Such  a $\Gamma_2$ is uniquely determined by the remaining two simple closed geodesics $\alpha_2,\ \beta_2$ contained in $Y$, and  
 $$Y\setminus\{\alpha_2,\beta_2\}\cong S_{0,3}\cup S_{0,3}.$$ 
 It is clear that $\ell(\alpha_2)+\ell(\beta_2)\leq 2L$ and $\alpha_2,\ \beta_2,\ \gamma_2$ bound a pair of pants in $Y$ as shown in Figure \ref{figure-s12-2}.  Then it follows from
\eqref{eq counting in S12 pairs} that there are at most $$
\frac{\ell(\gamma_2)}{\cD(\ell(\gamma_2),2L,0)}$$
such pairs of  $(\alpha_2,\beta_2)$'s. This implies that
\[\# \pi^{-1}(\gamma_1,\gamma_2,\alpha_1,\beta_1)\leq 
\frac{200\ell(\gamma_2)}{\cD(\ell(\gamma_2),2L,0)}.\]
Together with \eqref{count-s2}, one may complete the proof.
\end{proof}

\begin{lemma}\label{lemma123}
   For compact hyperbolic surface $X\in\M_g$ and $L>1$,    \begin{equation}\label{equation123}
    \begin{aligned}
        C_{1,2}^3(X,L)
        &\prec \sum_{(\gamma_1,\gamma_2,\alpha_1,\beta_1)\in\mathcal{G}_{1,2}^2(X)} 1_{D_{1,2}^3(L)}(\ell(\gamma_1),\ell(\gamma_2),\ell(\alpha_1),\ell(\beta_1)) \\
        &\cdot\frac{\ell(\gamma_2)}{\cD(\ell(\gamma_2),2L,0)},
    \end{aligned}
\end{equation}
where $\mathcal{G}_{1,2}^2(X)$ is defined in Lemma \ref{lemma122} and the domain $D_{1,2}^3(L)$ is defined as 
$$D_{1,2}^3(L)=\left\{(x_1,x_2,y_1,y_2)\in\mathbb{R}^4_{\geq 0};\ \begin{matrix}
    10\log L\leq x_1\leq L,\ 2x_1+x_2\leq 4L\\
    y_1,\ y_2\leq L
\end{matrix}\right\}.$$
\end{lemma}

\begin{proof}
 For any $(\Gamma_1,\Gamma_2)\in\mathcal{C}_{1,2}^3(X,L)$, WLOG, one may assume that $Y={S(\Gamma_1,\Gamma_2)}$, $\partial Y=\{\gamma_1,\gamma_2\}$, 
  both $\Gamma_1$ and $\Gamma_2$ contain $\gamma_1$ and do not contain $\gamma_2$. Assume that $$\Gamma_1=(\gamma_1,\alpha_1,\beta_1)\text{ and }\Gamma_2=(\gamma_1,\alpha_2,\beta_2).$$ 
  In this situation, we warn here that \emph{$\ell(\gamma_2)$ may exceed $L$}. Since $P(\Gamma_1)\cup P(\Gamma_2)$ fills $Y$, there is a connected component $C$ of $Y \setminus P(\Gamma_1)\cup P(\Gamma_2)$ such that $C$ is topologically a cylinder and $\gamma_2$ is a  connected component of  $\partial C$. The other connected component of $\partial C$, denoted by $\eta$, is a closed piecewise smooth geodesic loop, freely homotopical to $\gamma_2$. All geodesic arcs in $\eta$ are different parts of 
 arcs in $\alpha_1,\beta_1,\alpha_2,\beta_2$ as shown in Figure \ref{figure-s12-3}, then \begin{equation}\label{eq cylinder gamma2}
 \ell(\gamma_2)\leq \ell(\eta)\leq \ell(\alpha_1)+\ell(\beta_1)+\ell(\alpha_2)+\ell(\beta_2).
  \end{equation}
\begin{figure}[b]
    \centering
    \includegraphics[width=2.5 in]{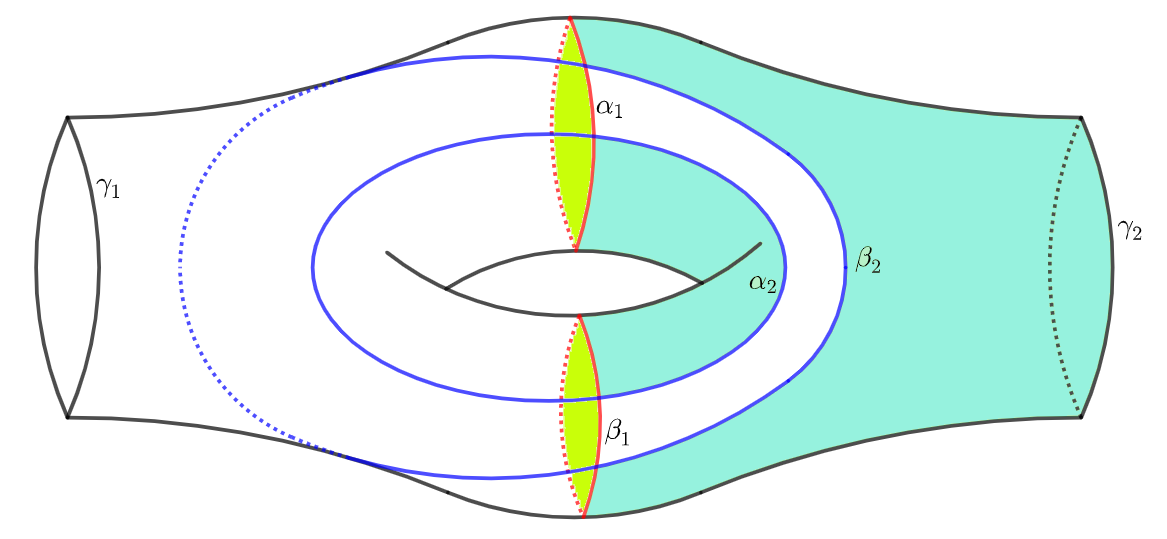}
    \caption{ $\Gamma_1=\{\gamma_1,\alpha_1,\beta_1\}$ and $\Gamma_2=\{\gamma_1,\alpha_2,\beta_2\}$ in $Y=S(\Gamma_1,\Gamma_2)\cong S_{1,2}$ with $\partial Y=\{\gamma_1,\gamma_2\}$}
    \label{figure-s12-3}
\end{figure}
Since $\Gamma_1,\Gamma_2\in \Nset(X,L)$,  we have \begin{equation}\label{eqs gamma2}
 \left\{
 \begin{aligned}
\ell(\gamma_1)+\ell(\alpha_1)+&\ell(\beta_1)\leq 2 L\\
\ell(\gamma_1)+\ell(\alpha_2)+&\ell(\beta_2)\leq 2 L\\
\ell(\gamma_1)\geq& 10\log L\\ 
 \end{aligned}\right..
 \end{equation}
    It follows from \eqref{eq cylinder gamma2}  and \eqref{eqs gamma2} that  \begin{equation}
\ell(\gamma_2)\leq 4L-2\ell(\gamma_1)\leq 4L-20\log L.
    \end{equation}
Then one may check that $(\gamma_1,\gamma_2,\alpha_1,\beta_1)\in\mathcal{G}_{1,2}^2(X)$ and \begin{equation}\label{condition s123}
\ell(\gamma_1),\ell(\alpha_1),\ell(\beta_1)\in[10\log L,L],\ 2\ell(\gamma_1)+\ell(\gamma_2)\leq 4L.
\end{equation} 
Consider the map $$
\pi:(\Gamma_1,\Gamma_2)\mapsto (\gamma_1,\gamma_2,\alpha_1,\beta_1).
$$
Then we have  \begin{equation}\label{count-s3}
\begin{aligned}
C_{1,2}^3(X,L)&\leq\sum\limits_{(\gamma_1,\gamma_2,\alpha_1,\beta_1)\in\mathcal{G}_{1,2}^2(X)}1_{D_{1,2}^3(L)}(\ell(\gamma_1),\ell(\gamma_2),\ell(\alpha_1),\ell(\beta_1))\\
    &\cdot\#\pi^{-1}(\gamma_1,\gamma_2,\alpha_1,\beta_1).
 \end{aligned}
 \end{equation}
For any fixed $(\gamma_1,\gamma_2,\alpha_1,\beta_1)$ and $(\Gamma_1,\Gamma_2)\in \mathcal{C}_{1,2}^3(X,L)$ with 
$$\pi(\Gamma_1,\Gamma_2)=(\gamma_1,\gamma_2,\alpha_1,\beta_1),$$
we have $\Gamma_1$ is fixed, $\gamma_1\in \Gamma_2, \gamma_2\notin \Gamma_2$   and $P(\Gamma_2)\subset Y$. Since $\ell(\alpha_2)+\ell(\beta_2)\leq 2L$,  it follows from \eqref{eq counting in S12 pairs} that there are at most $$ 
\frac{\ell(\gamma_2)}{\cD(\ell(\gamma_2),2L,0)}$$
such pairs of  $(\alpha_2,\beta_2)$'s. This implies that
\[\#\pi^{-1}(\gamma_1,\gamma_2,\alpha_1,\beta_1)\leq \frac{200\ell(\gamma_2)}{\cD(\ell(\gamma_2),2L,0)}.\]
Together with \eqref{count-s3}, one may complete the proof.
\end{proof}

Now we are ready to prove Proposition \ref{c12}.

\begin{proof}[Proof of Proposition \ref{c12}]

Following  Lemma \ref{lemma121}, Lemma \ref{lemma122} and Lemma \ref{lemma123}, for compact hyperbolic surface $X$ and $L>1$ we have 
 \begin{equation}\label{equ12main}
 \begin{aligned}
&\E\left[C_{1,2}(X,L)\right] =\sum\limits_{i=1}^3 \E\left[C_{1,2}^i(X,L)\right]\\
\prec & \E\Bigg[
\sum_{(\gamma_1,\gamma_2,\gamma_3)\in\mathcal{G}_{1,2}^1(X)} 1_{[0,L]^3}(\ell(\gamma_1),\ell(\gamma_2),\ell(\gamma_3))\\
\cdot &\left(\frac{\ell(\gamma_1)}{\cR(\ell(\gamma_1),\ell(\gamma_2),L)}+\frac{\ell(\gamma_1)}{\cD(\ell(\gamma_1),2L,0)}\right)\\
+&\sum_{(\gamma_1,\gamma_2,\alpha_1,\beta_1)\in\mathcal{G}_{1,2}^2(X)}\bigg( 1_{[0,L]^4}(\ell(\gamma_1),\ell(\gamma_2),\ell(\alpha_1),\ell(\alpha_2))
\cdot \frac{\ell(\gamma_2)}{\cD(\ell(\gamma_2),2L,0)}\\
+&1_{D_{1,2}^3(X)}(\ell(\gamma_1),\ell(\gamma_2),\ell(\alpha_1),\ell(\alpha_2))\cdot\frac{\ell(\gamma_2)}{\cD(\ell(\gamma_2),2L,0)}\bigg)\Bigg],
\end{aligned}
\end{equation} 
where $\mathcal{G}_{1,2}^1(X)$ and $\mathcal{G}_{1,2}^2(X)$ are defined in Lemma \ref{lemma121} and \ref{lemma122}.

Now we try to estimate the right hand side of \eqref{equ12main}. For the first term, assume 
$(\gamma_1,\gamma_2,\gamma_3)\in\mathcal{G}_{1,2}^1(X)$  and 
$Y_0\cong S_{1,2}$ is the subsurface in $X$ which is cutted off by $\gamma_1,\ \gamma_2$. Then  $X\setminus Y_0$ can be of type $$S_{g-2,2}\text{ or }S_{g_1,1}\cup S_{g_2,1}$$ with $g_1+g_2=g-1$. 
Hence we consider set $\mathcal{H}_{1,2}^1$ consists of all simple closed multi-curves with following types:
\begin{enumerate}
    \item $(\gamma_1,\gamma_2,\gamma_3)$ satisfying
    $$X\setminus\{\gamma_1,\gamma_2,\gamma_3\}\cong S_{0,3}\cup S_{1,1}\cup S_{g-2,2}$$
    where $\gamma_1,\ \gamma_2$ are the boundary curves of the part $S_{g-2,2}$, $\gamma_3$ is the boundary curve of the part $S_{1,1}$;
    \item $(\gamma_1,\gamma_2,\gamma_3)$ satisfying
    $$X\setminus\{\gamma_1,\gamma_2,\gamma_3\}\cong S_{0,3}\cup S_{1,1}\cup S_{g_1,1}\cup S_{g_2,1},$$
    where $\gamma_i$ is the boundary curve of the part $S_{g_i,1}\ (i=1,2)$, $\gamma_3$ is the boundary curve of the part $S_{1,1}$.
\end{enumerate}
For $L>1$, applying Mirzakhani's integration formula, i.e. Theorem \ref{thm Mirz int formula}  to all simple closed multi-curves in $\mathcal{H}_{1,2}^1$ and the function 
$$1_{[0,L]^3}(x_1,x_2,x_3)
\cdot \left(\frac{x_1}{\cR(x_1,x_2,L)}+\frac{x_1}{\cD(x_1,2L,0)}\right),$$
together with Theorem \ref{mirz07}, Theorem \ref{thm Vgn(x) small x},  and Theorem \ref{thm estimation R,D},  we have
       \begin{equation}\label{expectation121}
       \begin{aligned}
          &\E\Bigg[    \sum_{(\gamma_1,\gamma_2,\gamma_3)\in\mathcal{G}_{1,2}^1(X)} 1_{[0,L]^3}(\ell(\gamma_1),\ell(\gamma_2),\ell(\gamma_3))\\
\cdot &\left(\frac{\ell(\gamma_1)}{\cR(\ell(\gamma_1),\ell(\gamma_2),L)}+\frac{\ell(\gamma_1)}{\cD(\ell(\gamma_1),2L,0)}\right)\Bigg]\\  \prec &\frac{1}{V_g}
\int_{[0,L]^3}
\Big[ (1+x_1)\left(1+e^{\frac{L-x_1-x_2}{2}}\right)+(1+x_1)\left(1+e^{\frac{2L-x_1}{2}}\right)   \Big]V_{1,1}(x_3)
\\
\cdot& \left( 
V_{g-2,2}(x_1,x_2)+\sum_{(g_1,g_2)}V_{g_1,1}(x_1)V_{g_2,1}(x_2)
\right)x_1x_2x_3\cdot dx_1dx_2dx_3\\
\prec &\frac{1}{V_g}\int_{[0,L]^3}\Big[ (1+x_1)\left(1+e^{\frac{L-x_1-x_2}{2}}\right)+(1+x_1)\left(1+e^{\frac{2L-x_1}{2}}\right)   \Big]\\
\cdot& (1+x_3^2)\left(V_{g-2,2}+\sum_{(g_1,g_2)}V_{g_1,1}V_{g_2,1} \right)\sinh\frac{x_1}{2}\sinh\frac{x_2}{2}\cdot x_3\cdot dx_1dx_2dx_3\\
\prec&\frac{V_{g-2,2}+\sum\limits_{(g_1,g_2)}V_{g_1,1}V_{g_2,1}}{V_g}\cdot \left(L^5e^L+L^7e^{\frac L2}+   L^6 e^{\frac{3}{2}L}    \right),
       \end{aligned}
   \end{equation}

For the second term, assume 
$(\gamma_1,\gamma_2,\alpha_1,\beta_1)\in\mathcal{G}_{1,2}^2(X)$  and 
$Y_0\cong S_{1,2}$ is the subsurface in $X$ which is cutted off by $\gamma_1,\ \gamma_2$. Then  $X\setminus Y_0$ can be of type $$S_{g-2,2}\text{ or }S_{g_1,1}\cup S_{g_2,1}$$ with $g_1+g_2=g-1$. Hence we consider the set $\mathcal{H}_{1,2}^2$ consists of all simple closed curves with following types:
\begin{enumerate}
    \item $(\gamma_1,\gamma_2,\alpha_1,\beta_1)$ satisfying
    $$X\setminus\{\gamma_1,\gamma_2,\alpha_1,\beta_1\}\cong S_{0,3}\cup S_{0,3}\cup S_{g-2,2}$$
    where $\gamma_1,\ \gamma_2$ are boundary curves of the part $S_{g-2,2}$, $\alpha_1,\ \beta_1$ are common boundary curves of two different $S_{0,3}$;
    \item $(\gamma_1,\gamma_2,\alpha_1,\beta_1)$ satisfying
    $$X\setminus\{\gamma_1,\gamma_2,\alpha_1,\beta_1\}\cong S_{0,3}\cup S_{0,3}\cup S_{g_1,1}\cup S_{g_2,1},$$
    where $\gamma_i$ is the boundary curve of the part $S_{g_i,1}\ (i=1,2)$, $\alpha_1,\ \beta_1$ are common boundary curves of two different $S_{0,3}$.
\end{enumerate}

For $L>1$, applying Mirzakhani's integration formula, i.e. Theorem \ref{thm Mirz int formula} to all simple closed multi-curves contained in $\mathcal{H}_{1,2}^2$ and the function
$$1_{[0,L]^4}(x_1,x_2,y_1,y_2)\cdot\frac{x_2}{\cD(x_2,2L,0)},$$
together with Theorem \ref{mirz07},  Theorem \ref{thm Vgn(x) small x},  and Theorem \ref{thm estimation R,D}, we have  \begin{equation}
   \label{expectation122} \begin{aligned}
\nonumber&\E\Bigg[\sum_{(\gamma_1,\gamma_2,\alpha_1,\beta_1)\in\mathcal{G}_{1,2}^2(X)} 1_{[0,L]^4}(\ell(\gamma_1),\ell(\gamma_2),\ell(\alpha_1),\ell(\alpha_2))\cdot\frac{\ell(\gamma_2)}{\cD(\ell(\gamma_2),2L,0)}\Bigg]\\
\nonumber\prec&\frac{1}{V_g}\int_{[0,L]^4}(1+x_2)\left(1+e^{\frac{2L-x_2}{2}}\right)V_{0,3}(x_1,y_1,y_2)V_{0,3}(x_2,y_1,y_2)\\
\cdot&\left(V_{g-2,2}(x_1,x_2)+\sum_{(g_1,g_2)}V_{g_1,1}(x_1)V_{g_2,1}(x_2)\right)x_1x_2y_1y_2\cdot dx_1dx_2dy_1dy_2\\
\nonumber\prec&\frac{1}{V_g}\int_{[0,L]^4} (1+x_2)\left(1+e^{\frac{2L-x_2}{2}}\right)\left(V_{g-2,2}+\sum_{(g_1,g_2)}V_{g_1,1}V_{g_2,1}\right)\\
\nonumber\cdot&\sinh\frac{x_1}{2}\sinh\frac{x_2}{2}\cdot y_1y_2\cdot dx_1dx_2dy_1dy_2\\
\nonumber\prec&\frac{V_{g-2,2}+\sum\limits_{(g_1,g_2)}V_{g_1,1}V_{g_2,1}}{V_g}\cdot \left(L^5e^L+L^6e^{\frac{3}{2}L}\right).
    \end{aligned}
\end{equation}

Consider the remaining term on the right side of (\ref{equ12main}). Fot $L>1$, applying Mirzakhani's integration formula, i.e. Theorem \ref{thm Mirz int formula} to all simple closed multi-curves contained in $\mathcal{H}_{1,2}^2$ and function
 $$1_{D_{1,2}^3(L)}(x_1,x_2,y_1,y_2)\cdot \frac{x_2}{\cD(x_2,2L,0)},$$
together with Theorem \ref{mirz07},  Theorem \ref{thm Vgn(x) small x},  and Theorem \ref{thm estimation R,D}, we have
 \begin{equation}
   \label{expectation123}\begin{aligned}
&\E\left[\sum_{\substack{(\gamma_1,\gamma_2,\alpha_1,\beta_1)\\
\in\mathcal{G}_{1,2}^2(X)}} 1_{D_{1,2}^3(L)}(\ell(\gamma_1),\ell(\gamma_2),\ell(\alpha_1),\ell(\beta_1))\cdot\frac{200\ell(\gamma_2)}{\cD(\ell(\gamma_2),2L,0)}
        \right]\\
      \prec&\frac{1}{V_g}\int_{D_{1,2}^3(L)}(1+x_2)\left(1+e^{\frac{2L-x_2}{2}}\right)V_{0,3}(x_1,y_1,y_2)V_{0,3}(x_2,y_1,y_2)\\
        \cdot &\left(V_{g-2,2}(x_1,x_2)+\sum\limits_{(g_1,g_2)}V_{g_1,1}(x_1)V_{g_2,1}(x_2)\right)x_1x_2y_1y_2\cdot dx_1dx_2dy_1dy_2\\
     \prec&\frac{1}{V_g}\int_{D_{1,2}^3(L)}(1+x_2)\left(1+e^{\frac{2L-x_2}{2}}\right)\\
    \cdot&\left(V_{g-2,2}+\sum_{(g_1,g_2)}V_{g_1,1}V_{g_2,1}\right)\sinh\frac{x_2}{2}\sinh\frac{x_2}{2}\cdot y_1y_2\cdot dx_1dx_2dy_1dy_2\\
     \prec&\frac{V_{g-2,2}+\sum_{(g_1,g_2)}V_{g_1,1}V_{g_2,1}}{V_g}\!\cdot 
\! L^5\! \cdot\! \int_{\substack{
            10\log L\leq x\leq L\\
            0\leq x_2\leq 4L-2x_1
       }}\!
        \left(1+e^{\frac{2L-x_2}{2}}\right)\!e^{\frac{x_1+x_2}{2}}dx_1dx_2\\
        \prec&\frac{V_{g-2,2}+\sum\limits_{(g_1,g_2)}V_{g_1,1}V_{g_2,1}}{V_g}\cdot 
 L^5\cdot \left(e^{2L-5\log L}+Le^{\frac{3}{2}L}\right).
   \end{aligned}\end{equation}

By Theorem \ref{thm Vgn/Vgn+1} and Theorem \ref{thm sum-prod-V}, we have\begin{equation}\label{cond s12 vg-2/vg}
    \frac{  V_{g-2,2}+\sum\limits_{(g_1,g_2)}V_{g_1,1}V_{g_2,1}  }{V_g}\prec\frac{1}{V_g}\left(V_{g-2,2}+\frac{W_{2g-4}}{g}\right)\prec \frac{1}{g^2}.
\end{equation}
Then combining \eqref{equ12main}-\eqref{cond s12 vg-2/vg}
 we obtain $$
\begin{aligned}
    &\E\left[ C_{1,2}(X,L)\right]\prec\frac{  V_{g-2,2}+\sum\limits_{(g_1,g_2)}V_{g_1,1}V_{g_2,1}  }{V_g}
    \cdot \left(L^5e^L+L^7e^\frac{L}{2}+L^6e^{\frac{3L}{2}}+e^{2L}\right)\\
    \prec&\frac{V_{g-2,2}+\frac{W_{2g-4}}{g}}{V_g}
    \cdot \left(L^5e^L+L^7e^\frac{L}{2}+L^6e^{\frac{3L}{2}}+e^{2L}\right)\prec\frac{  e^{2L}  }{g^2}
\end{aligned}
$$
as desired.
\end{proof}

\subsubsection{Bounds for $\mathbb{E}_{\textnormal{WP}}^g\left[C_{0,4}(X,L)\right]$}\label{sec-4.3.3}
Our aim for  $\E\left[C_{0,4}(X,L)\right]$ is as follows. The proof is similar to the one in bounding $\E\left[C_{1,2}(X,L)\right]$.
\begin{proposition}\label{c04}
     For $L>1$ and large $g$,  
    $$
\E\left[C_{0,4}(X,L)\right]\prec \frac{Le^{2L}}{g^2}.
    $$
\end{proposition}

When $S(\Gamma_1,\Gamma_2)\cong S_{0,4}$, two boundary geodesics of  $S(\Gamma_1,\Gamma_2)$ may be the same closed geodesic in $X$, in this case, the completion $\overline{S(\Gamma_1,\Gamma_2)}\cong S_{1,2}$; otherwise $\overline{S(\Gamma_1,\Gamma_2)}\cong S(\Gamma_1,\Gamma_2)\cong S_{0,4}$. Moreover, each of $\Gamma_1$ and $\Gamma_2$ has exactly two closed geodesics contained in the boundary of  $S(\Gamma_1,\Gamma_2)$. Now we define
\begin{align*}
    &\MC_{0,4}^{0}(X,L):=\left\{(\Gamma_1,\Gamma_2)\in \MC_{0,4}(X,L), \ \overline{S(\Gamma_1,\Gamma_2)}\cong S_{0,4}\right\},\\
    &\MC_{0,4}^{1}(X,L):=\left\{(\Gamma_1,\Gamma_2)\in \MC_{0,4}(X,L), \ \overline{S(\Gamma_1,\Gamma_2)}\cong S_{1,2}\right\},
\end{align*}
and set $$
C_{0,4}^0(X,L)=\#\MC_{0,4}^0(X,L),\ C_{0,4}^1(X,L)=\#\MC_{0,4}^1(X,L).
$$
For $(\Gamma_1,\Gamma_2)$ in $\MC_{0,4}^1(X,L),$ view $S(\Gamma_1,\Gamma_2)$ as the result surface of cutting  $$\overline{S(\Gamma_1,\Gamma_2)}\cong S_{1,2}$$
 along a non-separating simple closed geodesic. With similar arguments as in  Lemma  \ref{lemma121}, Lemma \ref{lemma122}, Lemma \ref{lemma123} and Proposition \ref{c12}, one may deduce that
 \begin{proposition}\label{E s041}
      For $L>1$ and large $g$,
      $$
      \E\left[C_{0,4}^1(X,L) \right]\prec \frac{e^{2L}}{g^2}.
      $$
 \end{proposition}
  
Now we consider $\MC_{0,4}^0(X,L)$. Again we need to accurately classify elements in it according to the relative position of $(\Gamma_1,\Gamma_2)$ in $S(\Gamma_1,\Gamma_2)\cong S_{0,4}$, hence $\partial S(\Gamma_1,\Gamma_2)$ consists of four simple closed geodesics. There are three different possible cases:
\begin{enumerate}
    \item $\Gamma_1\cup\Gamma_2 \textit{ contains } \partial S(\Gamma_1,\Gamma_2)$;
    \item $\Gamma_1\cup\Gamma_2 \textit{ contains exactly } 3 \textit{ boundary geodesics of } S(\Gamma_1,\Gamma_2)$;
    \item $\Gamma_1\cup\Gamma_2 \textit{ contains exactly } 2 \textit{ boundary geodesics of } S(\Gamma_1,\Gamma_2)$.
\end{enumerate}
For $i=1,2,3$, denote by
$$\mathcal{C}_{0,4}^{0,i}(X,L)=\{(\Gamma_1,\Gamma_2)\in\mathcal{C}_{0,4}^0(X,L);\ (\Gamma_1,\Gamma_2)\text{ satisfies the $i$-th condition above}\}.
$$ and
$$C_{0,4}^{0,i}(X,L)=\#\mathcal{C}_{0,4}^{0,i}(X,L).$$
Also consider the set $\mathcal{G}_{0,4}^{0}(X)$ which consists of all simple closed multi-geodesics   \\$(\gamma_1,\gamma_2,\gamma_3,\gamma_4,\eta)$ satisfying $\gamma_1\cup\gamma_2\cup\gamma_3\cup\gamma_4$ cuts off a subsurface $Y\cong S_{0,4}$  in $X$ and $\eta$ bounds a $S_{0,3}$ in $Y$ along with $\gamma_1,\gamma_2$;

Now we start to estimate $C_{0,4}^{0,i}(X,L)$ for $i=1,2,3$.

\begin{lemma}\label{lemma0401}
   For $X\in \M_g$ and $L>1$, we have  \begin{equation}\label{equation0401}
    \begin{aligned}
   C_{0,4}^{0,1}(X,L)
\prec&\sum_{(\gamma_1,\gamma_2,\gamma_3,\gamma_4,\eta)\in\mathcal{G}_{0,4}^{0}(X)} 1_{D_{0,4}^{1,2}(L)}(\ell(\gamma_1),\ell(\gamma_2),\ell(\gamma_3),\ell(\gamma_4),\ell(\eta))\\
\cdot&\frac{\ell(\gamma_3)}{\cR(\ell(\gamma_3),\ell(\gamma_4),L)},
     \end{aligned}
 \end{equation}
    where the domain $D_{0,4}^{0,1}(L)$ is defined as
    \begin{align*}
    D_{0,4}^{0,1}(L)=\left\{(x_1,x_2,x_3,x_4,y)\in\mathbb{R}^5_{\geq 0};\ \begin{matrix}
        x_1,\ x_2,\ x_3,\ x_4,\ y\leq L;\\
        x_1+x_2,\ x_3+x_4\leq 2L-10\log L
    \end{matrix}\right\}.
    \end{align*}
\end{lemma}

\begin{proof} For any $(\Gamma_1,\Gamma_2)\in\mathcal{C}_{0,4}^{0,1}(X,L)$,
     WLOG, one may assume that $Y=S(\Gamma_1,\Gamma_2)$, $\partial Y=\{\gamma_1,\gamma_2,\gamma_3,\gamma_4\},$ and $\Gamma_1$ contains $\gamma_1,\gamma_2$. Then it follows that $\Gamma_2$ contains $\gamma_3,\gamma_4$. Assume 
     $$\Gamma_1=(\gamma_1,\gamma_2,\eta)\text{ and }\Gamma_2=(\gamma_3,\gamma_4,\xi)\ (\text{see Figure }\ref{figure-s04-1}).$$
     Then $(\gamma_1,\gamma_2,\gamma_3,\gamma_4,\eta)\in\mathcal{G}_{0,4}^{0}(X)$ and  their lengths satisfy \begin{equation}\label{condition s04011}
0\leq \ell(\gamma_1),\ \ell(\gamma_2),\ \ell(\gamma_3),\ \ell(\gamma_4),\ \ell(\eta)\leq L
      \end{equation}
      and \begin{equation}\label{condition s04012}
\ell(\gamma_1)+\ell(\gamma_2),\ \ell(\gamma_3)+\ell(\gamma_4)\leq 2L-10\log L.
        \end{equation}  
     Consider the map $$\pi:(\Gamma_1,\Gamma_2)\mapsto (\gamma_1,\gamma_2,\gamma_3,\gamma_4,\eta).$$ 

\begin{figure}[t]
    \centering
    \includegraphics[width=2.5 in]{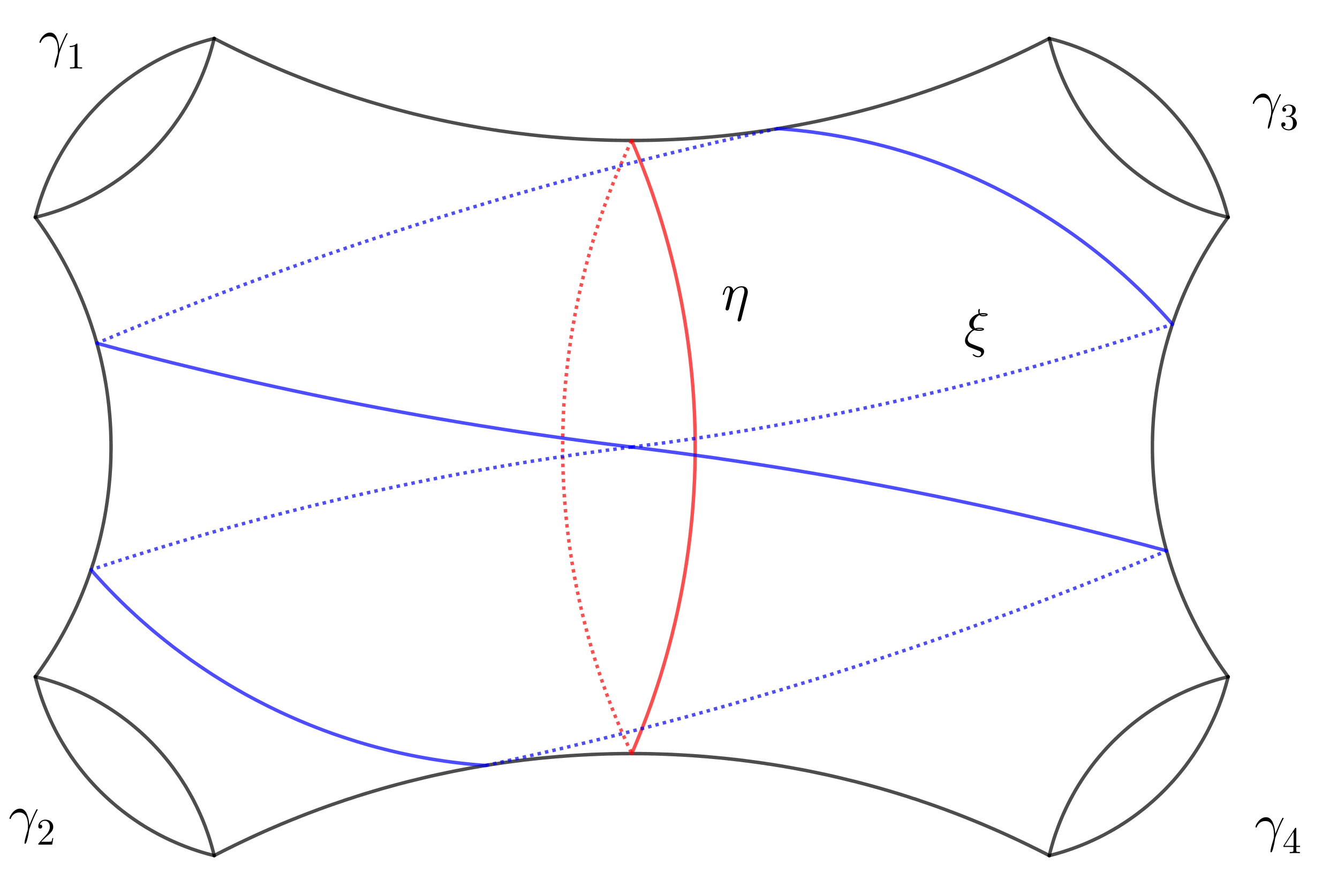}\caption{ $\Gamma_1=\{\gamma_1,\gamma_2,\eta\}$ and $\Gamma_2=\{\gamma_3,\gamma_4,\xi\}$ in $Y=S(\Gamma_1,\Gamma_2)\cong S_{0,4}$ with $\partial Y=\{\gamma_1,\gamma_2,\gamma_3,\gamma_4\}$}
    \label{figure-s04-1}
\end{figure}
\noindent Then we have
\begin{equation}\label{count-04-1}\begin{aligned}
    C_{0,4}^{0,1}(X,L)&\leq\sum\limits_{(\gamma_1,\gamma_2,\gamma_3,\gamma_4,\eta)\in\mathcal{G}_{0,4}^{0}(X)}1_{D_{0,4}^{1,2}(L)}(\ell(\gamma_1),\ell(\gamma_2),\ell(\gamma_3),\ell(\gamma_4),\ell(\eta))\\
    &\cdot\#\pi^{-1}(\gamma_1,\gamma_2,\gamma_3,\gamma_4,\eta).
\end{aligned}
\end{equation}
 For any fixed $(\gamma_1,\gamma_2,\gamma_3,\gamma_4,\eta)$ and $(\Gamma_1,\Gamma_2)\in \mathcal{C}_{0,4}^{0,1}(X,L)$ with $$\pi(\Gamma_1,\Gamma_2)=(\gamma_1,\gamma_2,\gamma_3,\gamma_4,\eta),$$
$\Gamma_1$ is fixed and $\Gamma_2$ is determined by $\xi$. Simple closed geodesic $\xi$ has length $\leq L$, and bounds an $S_{0,3}$ in $Y$ along with $\gamma_3,\ \gamma_4.$ Hence it follows from \eqref{eq counting in S04} that  
\[\#\pi^{-1}(\gamma_1,\gamma_2,\gamma_3,\gamma_4,\eta)\leq  \frac{200\ell(\gamma_3)}{\cR(\ell(\gamma_3),\ell(\gamma_4),L)}.\]
Together with \eqref{condition s04011}, \eqref{condition s04012} and \eqref{count-04-1}, one may complete the proof. 
\end{proof}

\begin{lemma}\label{lemma0402}
   For $X\in \M_g$ and $L>1$, we have \begin{equation}\label{equation0402}
    \begin{aligned}
   C_{0,4}^{0,2}(X,L)
     \prec &\sum_{(\gamma_1,\gamma_2,\gamma_3,\gamma_4,\eta)\in\mathcal{G}_{0,4}^{0,2}(X)}1_{D_{0,4}^{0,2}(L)}(\ell(\gamma_1),\ell(\gamma_2),\ell(\gamma_3),\ell(\gamma_4),\ell(\eta))\\
    \cdot &\frac{\ell(\gamma_2)}{\cR(\ell(\gamma_2),\ell(\gamma_4),L)},
     \end{aligned}
    \end{equation}
     where the domain $D_{0,4}^{0,2}(L)$ is defined as
     \begin{align*}
    D_{0,4}^{0,2}(L)=\left\{(x_1,x_2,x_3,x_4,y)\in\mathbb{R}^5_{\geq 0};\ \begin{matrix}
        10\log L\leq x_1\leq L,\ x_2,\ x_3,\ y\leq L;\\
        2x_1+x_2+x_3+x_4\leq 4L
    \end{matrix}\right\}.
    \end{align*}
\end{lemma}

\begin{proof}
 For any $(\Gamma_1,\Gamma_2)\in\mathcal{C}_{0,4}^{0,2}(X,L)$, WLOG, one may assume that $Y=S(\Gamma_1,\Gamma_2)$, $\partial Y=\{\gamma_1,\gamma_2,\gamma_3,\gamma_4\},$  $\Gamma_1$ contains $\gamma_1,\gamma_2$ and $\Gamma_2$ contains $\gamma_1,\gamma_3$. Assume $$\Gamma_1=(\gamma_1,\gamma_2,\eta)\text{ and }\Gamma_2=(\gamma_1,\gamma_3,\xi)\ (\text{see Figure }\ref{figure-s04-2}).$$ 
\begin{figure}[b]
    \centering
    \includegraphics[width=2.5 in]{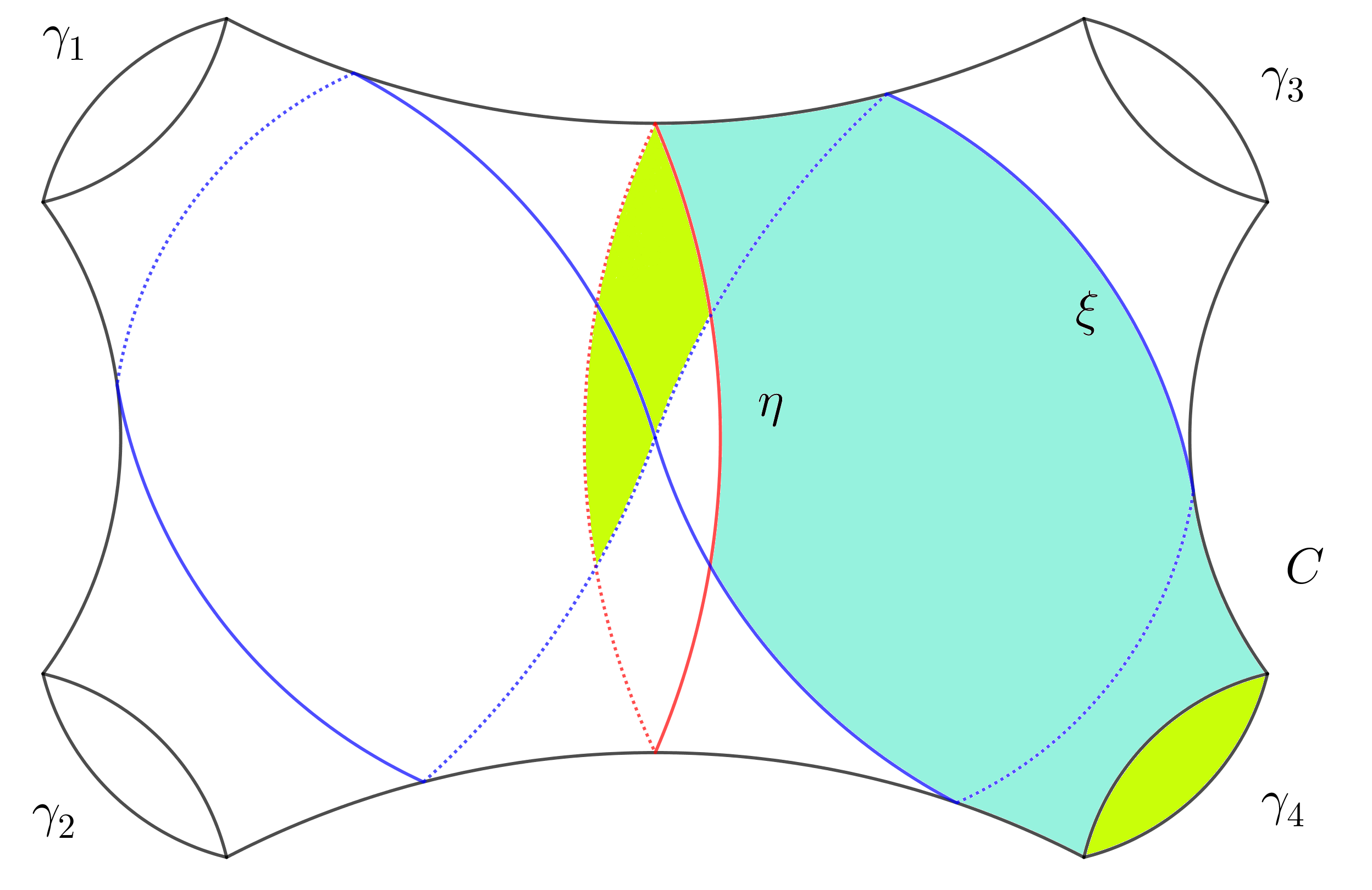}
    \caption{$\Gamma_1=\{\gamma_1,\gamma_2,\eta\}$ and $\Gamma_2=\{\gamma_1,\gamma_3,\xi\}$ in $Y=S(\Gamma_1,\Gamma_2)\cong S_{0,4}$ with $\partial Y=\{\gamma_1,\gamma_2,\gamma_3,\gamma_4\}$}
    \label{figure-s04-2}
\end{figure}
  In this case, $\ell(\gamma_4)$ may exceed $L$. However,    since $P(\Gamma_1)\cup P(\Gamma_2)$ fills $Y$, there is a connected component $C$ of $Y\setminus P(\Gamma_1)\cup P(\Gamma_2)$ such that $C$ is topologically a cylinder and $\gamma_4$ is a connected component of $\partial C$. The other connected component of $\partial C$, is the union of some geodesic arcs on $\eta,\ \xi.$ It follows that \begin{equation}\label{condition s04021}
    \ell(\gamma_4)\leq \ell(\eta)+\ell(\xi). \end{equation}
     Since $\Gamma_1,\ \Gamma_2\in\Nset(X,L),$ we have \begin{equation}\label{condition s04022}
    \left\{\begin{aligned}
\ell(\gamma_1)+\ell(\gamma_2)&+\ell(\eta)\leq 2L\\
\ell(\gamma_1)+\ell(\gamma_3)&+\ell(\xi)\leq 2L\\
\ell(\gamma_1)&\geq 10\log L\\
\end{aligned}\right..\end{equation}
    It follows from \eqref{condition s04021} and \eqref{condition s04022} that \begin{equation}
\ell(\gamma_1)+\ell(\gamma_2)+\ell(\gamma_3)+\ell(\gamma_4)\leq 4L-\ell(\gamma_1)\leq 4L-10\log L.
    \end{equation}
    Hence $(\gamma_1,\gamma_2,\gamma_3,\gamma_4,\eta)\in\mathcal{G}_{0,4}^0(X)$ and their lengths satisfy
    \begin{equation}\label{condition s0402}
    \begin{aligned}
&\ell(\gamma_1),\ell(\gamma_2),\ell(\gamma_3),\ell(\eta)\in[10\log L,L], \\
&2\ell(\gamma_1)+\ell(\gamma_2)+\ell(\gamma_3)+\ell(\gamma_4)\leq 4L.
     \end{aligned}\end{equation}
    Consider the map $$
    \pi:(\Gamma_1,\Gamma_2)\mapsto (\gamma_1,\gamma_2,\gamma_3,\gamma_4,\eta).
    $$ 
    Then
\begin{equation}\label{count-04-2}\begin{aligned}
    C_{0,4}^{0,2}(X,L)&\leq\sum\limits_{(\gamma_1,\gamma_2,\gamma_3,\gamma_4,\eta)\in\mathcal{G}_{1,2}^0(X)}1_{D_{0,4}^{0,2}(L)}(\ell(\gamma_1),\ell(\gamma_2),\ell(\gamma_3),\ell(\gamma_4),\ell(\eta))\\
    &\cdot\#\pi^{-1}(\gamma_1,\gamma_2,\gamma_3,\gamma_4,\eta).
\end{aligned}
\end{equation}
For any fixed $(\gamma_1,\gamma_2,\gamma_3,\gamma_4,\eta)$ and $(\Gamma_1,\Gamma_2)\in \mathcal{C}_{0,4}^{0,2}(X,L)$ with $$\pi(\Gamma_1,\Gamma_2)=(\gamma_1,\gamma_2,\gamma_3,\gamma_4,\eta),$$
$\Gamma_1$ is fixed and $\Gamma_2$ is determined by $\xi$.
Simple closed geodesic $\xi$ has length $\leq L$, and bounds an $S_{0,3}$ in $Y$ along with $\gamma_1,\gamma_3.$  It follows from \eqref{eq counting in S04}
 that 
\[\#\pi^{-1}(\gamma_1,\gamma_2,\gamma_3,\gamma_4,\eta)\leq   \frac{200\ell(\gamma_2)}{\cR(\ell(\gamma_2),\ell(\gamma_4),L)}.\]
Together with \eqref{condition s0402} and \eqref{count-04-2}, one may complete the proof. 
\end{proof}

\begin{lemma}\label{lemma0403}
 For $X\in \M_g$ and $L>1$, we have  \begin{equation}\label{equation0403}
    \begin{aligned}
  C_{0,4}^{0,3}(X,L)
\prec&\sum_{(\gamma_1,\gamma_2,\gamma_3,\gamma_4,\eta)\in\mathcal{G}_{0,4}^0(X)}1_{D_{0,4}^{0,3}(L)}(\ell(\gamma_1),\ell(\gamma_2),\ell(\gamma_3),\ell(\gamma_4),\ell(\eta))\\
\cdot&\frac{\ell(\gamma_3)}{\cR(\ell(\gamma_3),\ell(\gamma_4),L)},
     \end{aligned}
    \end{equation}
     where the domain $D_{0,4}^{0,3}(L)$ is defined as
     \begin{align*}
    D_{0,4}^{0,3}(L)=\left\{(x_1,x_2,x_3,x_4,y)\in\mathbb{R}^5_{\geq 0};\ \begin{matrix}
        x_1,\ x_2\in[10\log L ,L],\ y\leq L,\\
        2x_1+2x_2+x_3+x_4\leq 4L
    \end{matrix}\right\}.
    \end{align*}
\end{lemma}
\begin{proof}
 For any $(\Gamma_1,\Gamma_2)\in\mathcal{C}_{0,4}^{0,3}(X,L)$,
    WLOG, one may assume that $Y=S(\Gamma_1,\Gamma_2)$, $\partial Y=\{\gamma_1,\gamma_2,\gamma_3,\gamma_4\}$ and both  $\Gamma_1,\Gamma_2$ contain $\gamma_1,\gamma_2$.  Assume 
    $$\Gamma_1=(\gamma_1,\gamma_2,\eta)\text{ and }\Gamma_2=(\gamma_1,\gamma_2,\xi)\ (\text{see Figure }\ref{figure-s04-3}).$$ Then both $\eta$ and $\xi$ will bound a $S_{0,3}$ in $Y$ along with $\gamma_1\cup\gamma_2$.
\begin{figure}[b]
    \centering
    \includegraphics[width=2.5 in]{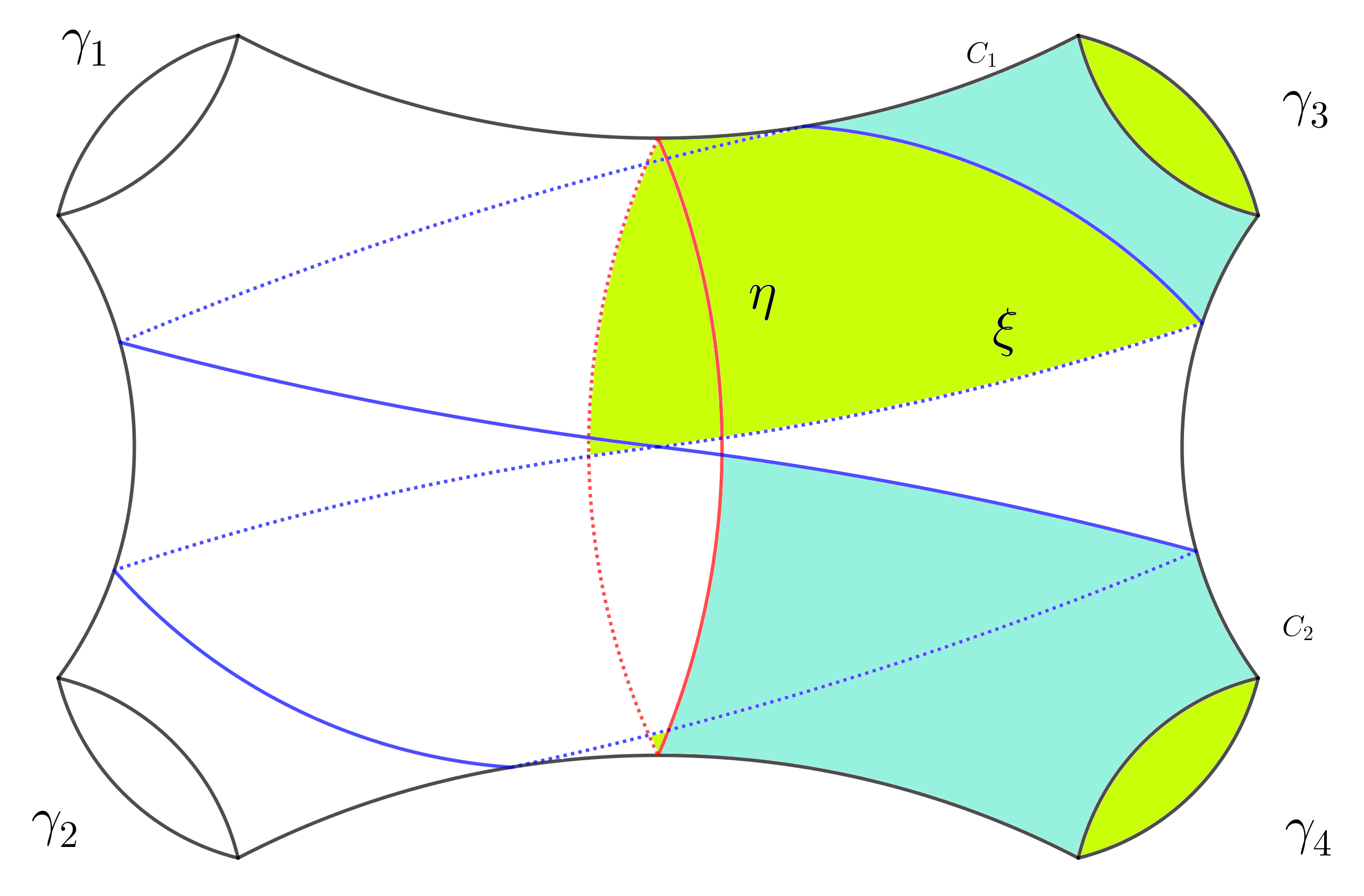}
    \caption{$\Gamma_1=\{\gamma_1,\gamma_2,\eta\}$ and $\Gamma_2=\{\gamma_1,\gamma_2,\xi\}$ in $Y=S(\Gamma_1,\Gamma_2)\cong S_{0,4}$ with $\partial Y=\{\gamma_1,\gamma_2,\gamma_3,\gamma_4\}$}
    \label{figure-s04-3}\end{figure}
     In this case, both  $\ell(\gamma_3)$ and $\ell(\gamma_4)$ may exceed $L$. Since $P(\Gamma_1)\cup P(\Gamma_2)$ fills $Y$, there are  two connected components $C_1,\ C_2$ of $Y\setminus P(\Gamma_1)\cup P(\Gamma_2)$ such that both $C_1,\ C_2$ are topologically  cylinders, $\gamma_3$ is a connected component of $\partial C_1$  and $\gamma_4$ is a connected component of $\partial C_2$. The other connected components of $\partial C_1$ and $\partial C_2$, are the union of different geodesic arcs on $\eta,\xi.$  It is clear that \begin{equation}\label{condition s04031}
\ell(\gamma_3)+\ell(\gamma_4)\leq \ell(\eta)+\ell(\xi).     \end{equation} 
     Since $\Gamma_1,\Gamma_2\in\Nset(X,L),$   we have \begin{equation}\label{condition s04032}
    \left\{\begin{aligned}
\ell(\gamma_1)+\ell(\gamma_2)&+\ell(\eta)\leq 2L\\
\ell(\gamma_1)+\ell(\gamma_2)&+\ell(\xi)\leq 2L\\
\ell(\gamma_1)&\geq 10\log L\\
\ell(\gamma_2)&\geq 10\log L\\
\end{aligned}\right..\end{equation}
Then it follows from \eqref{condition s04031} and \eqref{condition s04032} that $$
\ell(\gamma_1)+\ell(\gamma_2)+\ell(\gamma_3)+\ell(\gamma_4)\leq 4L-\ell(\gamma_1)-\ell(\gamma_2)\leq 4L-20\log L.
    $$
 Hence $(\gamma_1,\gamma_2,\gamma_3,\gamma_4,\eta)\in\mathcal{G}_{1,2}^0(X)$ and their lengths satisfy
    \begin{equation}\label{condition s0403}
 \begin{aligned}
&\ell(\gamma_1),\ell(\gamma_2),\ell(\eta)\in[10\log L,L], \\
&2\ell(\gamma_1)+2\ell(\gamma_2)+\ell(\gamma_3)+\ell(\gamma_4)\leq 4L.
     \end{aligned}   \end{equation}
      Consider the map $$
    \pi:(\Gamma_1,\Gamma_2)\mapsto (\gamma_1,\gamma_2,\gamma_3,\gamma_4,\eta).
    $$ 
    Then
\begin{equation}\label{count-04-3}\begin{aligned}
    C_{0,4}^{0,3}(X,L)&\leq\sum\limits_{(\gamma_1,\gamma_2,\gamma_3,\gamma_4,\eta)\in\mathcal{G}_{1,2}^0(X)}1_{D_{0,4}^{0,3}(L)}(\ell(\gamma_1),\ell(\gamma_2),\ell(\gamma_3),\ell(\gamma_4),\ell(\eta))\\
    &\cdot\#\pi^{-1}(\gamma_1,\gamma_2,\gamma_3,\gamma_4,\eta).
\end{aligned}
\end{equation}
 For any fixed $(\gamma_1,\gamma_2,\gamma_3,\gamma_4,\eta)$ and $(\Gamma_1,\Gamma_2)\in \mathcal{C}_{0,4}^{0,3}(X,L)$ with $$\pi(\Gamma_1,\Gamma_2)=(\gamma_1,\gamma_2,\gamma_3,\gamma_4,\eta),$$
$\Gamma_1$ is fixed and $\Gamma_2$ is determined by $\xi$.
Simple closed geodesic $\xi$ has length $\leq L$, and bounds an $S_{0,3}$ in $Y$ along with $\gamma_1,\ \gamma_3.$  It follows from \eqref{eq counting in S04}
 that 
\[\#\pi^{-1}(\gamma_1,\gamma_2,\gamma_3,\gamma_4,\eta)\leq   \frac{200\ell(\gamma_2)}{\cR(\ell(\gamma_2),\ell(\gamma_4),L)}.\]
Together with \eqref{condition s0403} and \eqref{count-04-3}, one may complete the proof.
\end{proof}

Assume $\mathcal{H}_{0,4}^0$ is the set consists of all simple closed geodesic multi-curves $\Gamma=(\gamma_1,\gamma_2,\gamma_3,\gamma_4,\eta)$ such that $\gamma_1\cup\gamma_2\cup\gamma_3\cup\gamma_4$ cuts off a subsurface $Y\cong S_{0,4}$ in $X$, and $\eta$ bounds a pair of pants in $Y$ along with $\gamma_1,\ \gamma_2$. Then the complement surface $X\setminus\Gamma$ could be one of the following five types: 
\begin{enumerate}
    \item $X\setminus\Gamma\cong S_{0,3}\cup S_{0,3}\cup S_{g-3,4}$;
    \item $X\setminus\Gamma\cong S_{0,3}\cup S_{0,3}\cup S_{g_1,1}\cup S_{g_2,3}$ for some $g_1\geq 1$ and $g_1+g_2=g-2$;
    \item $X\setminus\Gamma\cong S_{0,3}\cup S_{0,3}\cup S_{g_1,1}\cup S_{g_2,1}\cup S_{g_3,2}$ for some $g_1,\ g_2,\ g_3\geq 1$ and  $g_1+g_2+g_3=g-1$;
    \item $X\setminus\Gamma\cong S_{0,3}\cup S_{0,3}\cup S_{g_1,1}\cup S_{g_2,1}\cup S_{g_3,1}\cup S_{g_4,1}$ for some  $g_1,\ g_2,\ g_3,\ g_4\geq 1$ and $g_1+g_2+g_3+g_4=g$;
    \item $X\setminus\Gamma\cong S_{0,3}\cup S_{0,3}\cup S_{g_1,2}\cup S_{g_2,2}$ for some $g_1,g_2\geq 1$ and $g_1+g_2=g-2$.
\end{enumerate}  
Recall the definition of \eqref{e-wpvol}. For any $\nu=(x_1,x_2,x_3,x_4,y)\in\mathbb{R}^5_{\geq 0}$, 
$$V_g(\Gamma,\nu)=\Volwp\big(\M\left( S_g(\Gamma),\ell(\Gamma)=\nu\right)\big)$$  is the Weil-Petersson volume of the moduli space of Riemann surfaces that is homeomorphic to  $S_g\setminus\Gamma$ with geodesic boundaries of lengths $\ell(\gamma_i)=x_i\ (1\leq i\leq 4)$ and $\ell(\eta)=y$.
Then by Theorem \ref{thm Vgn/Vgn+1} and Theorem \ref{thm sum-prod-V} we have \begin{equation}\label{s040 ineq2}
    \begin{aligned}
&\sum\limits_{\Gamma\in\mathcal{H}_{0,4}^0}V_g(\Gamma,0)\\
        =&V_{g-3,4}+\sum_{g_1+g_2=g-2}V_{g_1,1}V_{g_2,3}
    +\sum_{g_1+g_2+g_3=g-1}V_{g_1,1}V_{g_2,1}V_{g_3,2}\\+&    \sum_{g_1+g_2+g_3+g_4=g}V_{g_1,1}V_{g_2,1}V_{g_3,1}V_{g_4,1}+\sum_{g_1+g_2=g-2}V_{g_1,2}V_{g_2,2}\\
    \prec&W_{2g-4}\left(1+\frac{1}{g}+\frac{1}{g^2}+\frac{1}{g^3} \right)\prec\frac{V_g}{g^2}.
    \end{aligned}
\end{equation}

Now we are ready to eatimate $\E\left[C_{0,4}^0(X,L)\right]$.

\begin{proposition}\label{E s040}
    For $L>1$ and large $g$,
    $$
    \E\left[C_{0,4}^0(X,L) \right]\prec \frac{Le^{2L}}{g^2}.
    $$
\end{proposition}

\begin{proof}
   Following
Lemma \ref{lemma0401}, Lemma \ref{lemma0402} and 
 Lemma \ref{lemma0403},  we have that for $L>1$,
    \begin{equation}\label{equ04main2}
    \begin{aligned}
    &\E\Big[C_{0,4}^0(X,L)\Big]\\
     \leq &\E\Bigg[\sum_{\Gamma\in\mathcal{H}_{0,4}^0} \Big( 1_{D_{0,4}^{0,1}(L)}(\ell(\gamma_1),\ell(\gamma_2),\ell(\gamma_3),\ell(\gamma_4),\ell(\eta)) \cdot\frac{\ell(\gamma_3)}{\cR(\ell(\gamma_3),\ell(\gamma_4),L)}\\
     +&1_{D_{0,4}^{0,2}(L)}(\ell(\gamma_1),\ell(\gamma_2),\ell(\gamma_3),\ell(\gamma_4),\ell(\eta))\cdot \frac{\ell(\gamma_2)}{\cR(\ell(\gamma_2),\ell(\gamma_4),L)}\\
  +&1_{D_{0,4}^{0,3}(L)}(\ell(\gamma_1),\ell(\gamma_2),\ell(\gamma_3),\ell(\gamma_4),\ell(\eta))\cdot \frac{\ell(\gamma_3)}{\cR(\ell(\gamma_3),\ell(\gamma_4),L)}    \Big)\Bigg],
    \end{aligned}
    \end{equation}
    where the domains $D_{0,4}^{0,1}(L),\ D_{0,4}^{0,2}(L),\ D_{0,4}^{0,3}(L)$ have been defined in Lemma \ref{lemma0401}, Lemma \ref{lemma0402} and Lemma \ref{lemma0403} respectively.
 
 For $L>1$, applying Mirzakhani's integration formula Theorem \ref{thm Mirz int formula} to function 
 $$1_{D_{0,4}^{0,1}(L)}(x_1,x_2,x_3,x_4,y)\cdot \frac{x_3}{\cR(x_3,x_4,L)}$$
 and all simple closed multi-curves in $\mathcal{H}_{0,4}^{0}$, together with  
Theorem \ref{mirz07}, 
 Theorem \ref{thm Vgn(x) small x}  and Theorem \ref{thm estimation R,D},  if $\nu=(x_1,x_2,x_3,x_4,y)$, we have \begin{equation}
    \label{expectation0401}\begin{aligned}
&\E\Bigg[\sum_{\Gamma\in\mathcal{H}_{0,4}^0} 1_{D_{0,4}^{0,1}(L)}(\ell(\gamma_1),\ell(\gamma_2),\ell(\gamma_3),\ell(\gamma_4),\ell(\eta))\cdot\frac{\ell(\gamma_3)}{\cR(\ell(\gamma_3),\ell(\gamma_4),L)}\Bigg]\\
     \prec&\frac{1}{V_g}\int_{D_{0,4}^{0,1}(L)}(1+x_3)\left(1+e^{\frac{L-x_3-x_4}{2}}\right)\cdot\sum\limits_{\Gamma\in\mathcal{H}_{0,4}^{0}}V_g(\Gamma,\nu)\\
\cdot & \prod\limits_{i=1}^4 x_i\cdot ydx_1dx_2dx_3dx_4dy \\
\prec&\frac{1}{V_g}\int_{D_{0,4}^{0,1}(L)}(1+x_3)\left(1+e^{\frac{L-x_3-x_4}{2}}\right)
 \cdot \sum\limits_{\Gamma\in\mathcal{H}_{0,4}^{0}}V_g(\Gamma,0)\\
     \cdot& \prod\limits_{i=1}^4\sinh\frac{x_i}{2}\cdot y\cdot dx_1dx_2dx_3dx_4dy\\ \prec&\frac{\sum\limits_{\Gamma\in\mathcal{H}_{0,4}^{0}}V_g(\Gamma,0)}{V_g}\cdot L^3
    \cdot  \int_{\tiny
    \begin{aligned}
        x_1+x_2&\leq 2L-10\log L\\
        x_3+x_4&\leq 2L-10\log L\\
    \end{aligned}
    }
    \left(1+e^{\frac{L-x_3-x_4}{2}}\right)\\
    \cdot &e^{\frac{x_1+x_2+x_3+x_4}{2}}dx_1dx_2dx_3dx_4 \\
\prec&\frac{\sum\limits_{\Gamma\in\mathcal{H}_{0,4}^{0}}V_g(\Gamma,0)}{V_g}\cdot L^3 \cdot\left( 
     L^2 e^{2L-10\log L}+L^3e^{\frac{3L}{2}-5\log L}
      \right).
    \end{aligned}\end{equation}

Similarly, applying Mirzakhani's integration formula Theorem \ref{thm Mirz int formula} to functions 
$$1_{D_{0,4}^{0,2}(L)}(x_1,x_2,x_3,x_4,y)\cdot \frac{x_2}{\cR(x_2,x_4,L)}$$
and
$$1_{D_{0,4}^{0,3}(L)}(x_1,x_2,x_3,x_4,y)\cdot \frac{x_3}{\cR(x_3,x_4,L)}$$
 and all simple closed geodesics in $\mathcal{G}_{0,4}^0(X)$, together with Theorem \ref{mirz07}, 
 Theorem \ref{thm Vgn(x) small x}  and Theorem \ref{thm estimation R,D},   we have 
    \begin{align}\label{expectation0402}
\nonumber&\E\Bigg[\sum_{\Gamma\in\mathcal{H}_{0,4}^{0}} 1_{D_{0,4}^{0,2}(L)}(\ell(\gamma_1),\ell(\gamma_2),\ell(\gamma_3),\ell(\gamma_4),\ell(\eta))\cdot \frac{\ell(\gamma_2)}{\cR(\ell(\gamma_2),\ell(\gamma_4),L)}
     \Bigg]\\
     \prec&\frac{1}{V_g}\int_{D_{0,4}^{0,2}(L)}(1+x_2)\left(1+e^{\frac{L-x_2-x_4}{2}}\right)
     \cdot \sum\limits_{\Gamma\in\mathcal{H}_{0,4}^0}V_g(\Gamma,0)\\
    \nonumber \cdot& \prod\limits_{i=1}^4\sinh\frac{x_i}{2}\cdot y\cdot dx_1dx_2dx_3dx_4dy\\
\nonumber\prec&\frac{\sum\limits_{\Gamma\in\mathcal{H}_{0,4}^0}V_g(\Gamma,0)}{V_g}\cdot  L^3\cdot \int_{\substack{
      0\leq x_1,x_2,x_3\leq L\\
        x_1+x_2+x_3+x_4\leq 4L-10\log L
  }
    }
    \left(1+e^{\frac{L-x_2-x_4}{2}}\right)\\
   \nonumber \cdot&e^{\frac{x_1+x_2+x_3+x_4}{2}}dx_1dx_2dx_3dx_4\\
\nonumber\prec&\frac{\sum\limits_{\Gamma\in\mathcal{H}_{0,4}^0}V_g(\Gamma,0)}{V_g}\cdot L^3 \cdot\left( 
     L^3     e^{2L-5\log L}
    +
     L^3e^{\frac{3L}{2}}\right)
    \end{align}
and \begin{equation}
    \label{expectation0403}\begin{aligned}
&\E\Bigg[\sum\limits_{\Gamma\in\mathcal{H}_{0,4}^{0}}1_{D_{0,4}^{0,3}(L)}(\ell(\gamma_1),\ell(\gamma_2),\ell(\gamma_3),\ell(\gamma_4),\ell(\eta))\cdot \frac{\ell(\gamma_3)}{\cR(\ell(\gamma_3),\ell(\gamma_4),L)}  
     \Big)\Bigg]\\
     \prec&\frac{1}{V_g}\int_{D_{0,4}^{0,3}(L)}(1+x_3)\left(1+e^{\frac{L-x_3-x_4}{2}}\right)
     \cdot \sum\limits_{\Gamma\in\mathcal{H}_{0,4}^0}V_g(\Gamma,0)
\\
   \cdot& \prod\limits_{i=1}^4\sinh\frac{x_i}{2}\cdot y\cdot dx_1dx_2dx_3dx_4dy\\   \prec&\frac{\sum\limits_{\Gamma\in\mathcal{H}_{0,4}^0}V_g(\Gamma,0)}{V_g}\cdot L^3
    \cdot  \int_{\substack{
      0\leq x_1,x_2\leq L\\
        x_1+x_2+x_3+x_4\leq 4L-20\log L
    }
    }
    \left(1+e^{\frac{L-x_3-x_4}{2}}\right)\\
    \cdot& e^{\frac{x_1+x_2+x_3+x_4}{2}}dx_1dx_2dx_3dx_4\\
  \prec&\frac{\sum\limits_{\Gamma\in\mathcal{H}_{0,4}^0}V_g(\Gamma,0)}{V_g}\cdot L^3 \cdot\left( 
     L^3     e^{2L-10\log L}
    +
     L^3e^{\frac{3L}{2}}\right).
    \end{aligned}\end{equation}
Then combining \eqref{equ04main2}, \eqref{expectation0401}, \eqref{expectation0402} and \eqref{expectation0403} we have for $L>1$, 
\begin{equation}\label{s040 ineq1}
\E\left[C_{0,4}^0(X,L)\right]
    \prec\frac{\sum\limits_{\Gamma\in\mathcal{H}_{0,4}^0}\textnormal{Vol}_{\textnormal{WP}}(\Gamma,0)}{V_g}\cdot Le^{2L}.
  \end{equation}
Therefore, by \eqref{s040 ineq2} and \eqref{s040 ineq1} we obtain $$
\E\left[C_{0,4}^0(X,L)\right]=O\left(\frac{Le^{2L}}{g^2}\right)
$$
as desired.
\end{proof}

Now we are ready to prove Proposition \ref{c04}.

\begin{proof}[Proof of Proposition
\ref{c04}] 
Since $\MC_{0,4}(X,L)=\MC_{0,4}^0(X,L)\cup\MC_{0,4}^1(X,L)$, it follows from  Proposition \ref{E s041} and  Proposition \ref{E s040} that $$
\E\left[C_{0,4}(X,L)\right]\leq \E\left[C_{0,4}^1(X,L)\right]+\E\left[C_{0,4}^0(X,L)\right]\prec \frac{Le^{2L}}{g^2}.
$$    
This completes the proof.
\end{proof}

\subsection{Estimations of $\E\left[D(X,L) \right]$}
For this part, we always assume that $g>2$. We will show that as $g\to \infty$,
\[\E\left[D(X,L_g) \right]=o\left(\E\left[\Nnumber(X,L_g)\right]^2\right).\]
More precisely,

\begin{proposition}\label{D(x,L)}
    For $L>1$, we have
    $$
    \E\left[D(X,L) \right]\prec \frac{e^{2L}}{g^2L^6}.
    $$
\end{proposition}
\begin{proof}
For $(\Gamma_1,\Gamma_2)\in\MD(X,L)$, the two pairs of  pants $P(\Gamma_1)$ and $P(\Gamma_2)$  share one or two simple closed geodesic boundary components.
For the first case, assume  $$\Gamma_1=(\gamma_1,\gamma_2,\eta)\text{ and }\Gamma_2=(\gamma_3,\gamma_4,\eta).$$ 
By the definition of $\Nset(X,L)$, we have
$$10\log L\leq \eta\leq L$$
and
$$\ell(\gamma_1)+\ell(\gamma_2),\ \ell(\gamma_3)+\ell(\gamma_4) \in[0,2L-10\log L].$$
For the second case, assume 
$$\Gamma_1=(\gamma_1,\alpha,\beta)\text{ and }\Gamma_2=(\gamma_2,\alpha,\beta).$$
By the definition of $\Nset(X,L)$, we have
$$\ell(\gamma_1),\ \ell(\gamma_2),\ \ell(\alpha),\ \ell(\beta) \in[0,L].$$
Consider functions 
$$\psi_{1,L}(x_1,x_2,x_3,x_4,y)=1_{[10\log L,L]}(y)\cdot 1_{[0,2L-10\log L]^2}(x_1+x_2,x_3+x_4)$$
and 
$$\psi_{2,L}(x_1,x_2,y_1,y_2)=1_{[0,L]^4}(x_1,x_2,y_1,y_2).$$
   Then we have
   \begin{equation}\label{ineq D1}
        \begin{aligned}
            D(X,L)
\leq&\sum_{(\gamma_1,\gamma_2,\gamma_3,\gamma_4,\eta)\in\mathcal{G}_{0,4}^0(X)}\psi_{1,L}(\ell(\gamma_1),\ell(\gamma_2),\ell(\gamma_3),\ell(\gamma_4),\ell(\eta))
\\
+&\sum_{(\gamma_1,\gamma_2,\alpha,\beta)\in\mathcal{G}_{1,2}^2(X)}\psi_{2,L}(\ell(\gamma_1),\ell(\gamma_2),\ell(\alpha),\ell(\beta)).
        \end{aligned}
    \end{equation}
     where $\mathcal{G}_{0,4}^0(X)$ is defined in Subsection \ref{sec-4.3.3}, and $\mathcal{G}_{1,2}^2(X)$ is defined in Lemma \ref{lemma122}.
Set $$
\textbf{cond}=\left\{\begin{matrix}
    &0\leq y\leq L\\
    &x_1+x_2\leq 2L-10\log L\\
    &x_3+x_4\leq 2L-10\log L\\
\end{matrix}.\right.
$$
Applying Mirzakhani's integration formula Theorem \ref{thm Mirz int formula} to function $\psi_{1,L}$ and all simple closed multi-curves in $\mathcal{H}_{0,4}^0$, together with  Theorem \ref{thm Vgn/Vgn+1}, 
 Theorem \ref{thm Vgn(x) small x} and \eqref{s040 ineq2}, for $L>1$   we have 
     \begin{align}\label{ineq D02} 
\nonumber&\E\left[\sum_{\Gamma\in\mathcal{G}_{0,4}^0(X)}\psi_{1,L}(\ell(\gamma_1),\ell(\gamma_2),\ell(\gamma_3),\ell(\gamma_4),\ell(\eta))\right]\\
\prec&\frac{1}{V_g}\int_{\textbf{cond}}\sum_{\mathcal{H}_{0,4}^0}V_g(\Gamma,\nu)
\cdot \prod\limits_{i=1}^4 x_i\cdot y\cdot dx_1dx_2dx_3dx_4dy\\
\nonumber\prec&\frac{1}{g^2}\int_{\textbf{cond}}\prod\limits_{i=1}^4\sinh\frac{x_i}{2}\cdot y\cdot dx_1dx_2dx_3dx_4dy\prec\frac{e^{2L}}{g^2L^6},
     \end{align}
where in the last inequality we apply
\beqar
&&\int_{\substack{u>0,\ v>0, \\ u+v\leq 2L-10\log L}}\sinh\frac{u}{2}\sinh\frac{v}{2} dudv\\
&& \prec \int_{\substack{u>0,\ v>0,\\ u+v\leq 2L-10\log L}}e^{\frac{u+v}{2}} dudv\\
&&\prec \int_{0}^{2L-10\log L}\left(e^{\frac{x}{2}}\int_{0}^{2L-10\log L-u}e^{\frac{v}{2}}dv \right)du\asymp\frac{e^L}{L^4}.
\eeqar
For the remaining term, applying Mirzakhani's integration formula Theorem \ref{thm Mirz int formula} to function $\psi_{2,L}$ and all simple closed multi-curves in $\mathcal{H}_{1,2}^2$, together with Theorem \ref{thm Vgn/Vgn+1}, 
 Theorem \ref{thm Vgn(x) small x} and \eqref{cond s12 vg-2/vg}, we have for $L>1$, \begin{equation}
     \label{ineq D03}\begin{aligned}
         &\E\left[ \sum_{\Gamma\in\mathcal{G}_{1,2}^2(X)}\phi_{2,L}\left(\ell(\gamma_1),\ell(\gamma_2),\ell(\alpha),\ell(\beta)\right)\right]\\
       \prec&\frac{1}{V_g}\int_{[0,L]^4}\sum\limits_{\Gamma\in\mathcal{H}_{1,2}^2}V_g(\Gamma,\nu)\cdot x_1x_2y_1y_2\cdot dx_1dx_2dy_1dy_2\\
        =&\frac{1}{V_g}\int_{[0,L]^4}\left(V_{g-2,2}(x_1,x_2)+\sum_{(g_1,g_2)}V_{g_1,1}(x_1)V_{g_2,1}(x_2) \right)\\
        \cdot& x_1x_2y_1y_2\cdot dx_1dx_2dy_1dy_2\\
        \prec&\frac{V_{g-2,2}+\sum_{(g_1,g_2)}V_{g_1,1}V_{g_2,1}}{V_g}\int_{[0,L]^4}\sinh\frac{x_1}{2}\sinh\frac{x_2}{2}\cdot y_1y_2\cdot dx_1dx_2dy_1dy_2\\
        \prec&\frac{L^4e^L}{g^2}.
     \end{aligned} \end{equation}
 Combining \eqref{ineq D1}, \eqref{ineq D02} and \eqref{ineq D03}, we have $$
 \E\left[D(X,L)\right]\prec\frac{1}{g^2}\left(\frac{e^{2L}}{L^6}+L^4e^L\right)\prec\frac{e^{2L}}{g^2L^6}.
 $$
 This completes the proof.
\end{proof}

\subsection{Finish  of the proof}
Now we are ready to complete the proof of Theorem \ref{thm upper bound of ns-sys}.

\begin{proof}[Proof of Theorem \ref{thm upper bound of ns-sys}]
Take $L=L_g=\log g-\log \log g +\omega(g)>1$ with $\omega(g)=o(\log \log g)$.
    By \eqref{ineq Nset markov} and \eqref{separate N into ABCD} we have \begin{equation}\label{proof upper ineq}
        \begin{aligned}
           &\Prob\left(X\in \sM_g; \ \Nnumber(X,L_g)=0   \right)\\
           \leq& \frac{\left|\E\left[B(X,L_g)\right]-\E\left[\Nnumber(X,L_g)\right]^2\right|}{\E\left[\Nnumber(X,L_g)\right]^2}\\
+&\frac{\E\left[A(X,L_g)\right]+\E\left[C(X,L_g)\right]+\E\left[D(X,L_g)\right]}{\E\left[\Nnumber(X,L_g)\right]^2}.\\
        \end{aligned}
    \end{equation}
    By Proposition \ref{p-e-1} and Proposition \ref{estimation B(X,L)}  we have 
    \begin{equation}\label{ineq B-N^2}
        \frac{\left|\E\left[B(X,L_g)\right]-\E\left[\Nnumber(X,L_g)\right]^2\right|}{\E\left[\Nnumber(X,L_g)\right]^2}=O\left(\frac{\log L_g}{L_g}\right).
    \end{equation}
    By \eqref{compare AN} and Proposition \ref{p-e-1}, for $L=L_g=\log g-\log \log g+\omega(g)$ we have \begin{equation}
        \label{ineq A only}
\frac{\E\left[A(X,L_g)\right]}{\E\left[\Nnumber(X,L_g)\right]^2}\prec \frac{1}{\E\left[\Nnumber(X,L_g)\right]}=O\left(\frac{1}{e^{\omega(g)}} 
  \right).
    \end{equation}
By Proposition \ref{p-e-1}, Proposition \ref{prop C geq3}, Proposition \ref{c12} and Proposition \ref{c04}, fix $0<\epsilon<\frac{1}{2}$, we have 
    \begin{equation}
        \label{ineq C only}
\frac{\E\left[C(X,L_g)\right]}{\E\left[\Nnumber(X,L_g)\right]^2}\prec \left(   \frac{L_g^{65}e^{\epsilon L_g}}{g} +\frac{L_ge^{6L_g}}{g^9}+\frac{1+L_g}{L_g^2}  \right)=O\left(\frac{1}{L_g}\right).
    \end{equation}
By Proposition \ref{p-e-1} and Proposition \ref{D(x,L)} we have
 \begin{equation}
        \label{ineq D only}
\frac{\E\left[D(X,L_g)\right]}{\E\left[\Nnumber(X,L_g)\right]^2}=O\left(\frac{1}{L_g^8}\right).
    \end{equation}
    Therefore if $\lim\limits_{g\to\infty}\omega(g)=+\infty$, by \eqref{proof upper ineq}, \eqref{ineq B-N^2}, \eqref{ineq A only}, \eqref{ineq C only} and \eqref{ineq D only}, we have $$
\lim_{g\to\infty}\Prob\left(X\in \sM_g; \ \Nnumber(X,L_g)=0\right)=0.
    $$
It follows that  \eqref{eq up f8} holds by  \eqref{link f-8 with Nstar}. This finishes the proof of Theorem  \ref{thm upper bound of ns-sys}.
\end{proof}
\bibliographystyle{plain}
\bibliography{ref}

\end{document}